\documentclass[final,10pt]{elsarticle}

\usepackage{amsfonts}       
\usepackage{amsthm}       
\usepackage{amssymb}
\usepackage{fullpage}
\usepackage{hyperref}
\usepackage{pdfpages}
\usepackage{comment}

\usepackage{algorithm}

\usepackage{algorithmicx}
\usepackage{algpseudocode}

\usepackage{thmtools}
\usepackage{thm-restate}
\declaretheorem[name=Proposition,numberwithin=section]{proposition}
\theoremstyle{definition}
\newtheorem{definition}{Definition}[section]
\newtheorem{theorem}{Theorem}[section]
\newtheorem{corollary}{Corollary}[theorem]

\newtheorem{remark}[theorem]{Remark}
\usepackage{geometry}
\usepackage{color}

\newcommand{\proj}{P}
\newcommand{\vecspan}{\text{span}}
\newcommand{\ve}{\boldsymbol}
\usepackage{mathtools}
\newcommand{\defeq}{\vcentcolon =}
\newcommand{\range}{\text{range}}
\newcommand{\residual}{\tilde{\ve{r}}}
\newcommand{\gResidual}{\ve{r}}
\newcommand{\stateVec}{\ve{x}}
\newcommand{\iter}{k}
\newcommand{\mat}{\mathbf}
\newcommand{\Amat}{{\mat A}}
\newcommand{\bvec}{{\mat b}}
\newcommand{\nat}[1]{\mathbb{N}(#1)}
\newcommand{\RR}[1]{\mathbb{R}^{#1}}
\newcommand{\vecSpace}{\mathbb}
\newcommand{\velocity}{\ve{f}}
\newcommand{\identity}{{\mat I}}
\newcommand{\graphObj}{\mathcal}
\newcommand{\stateVecIter}{\stateVec^\iter}
\newcommand{\romStateVec}{\hat{\stateVec}}
\newcommand{\romStateVecIter}{\romStateVec^\iter}
\newcommand{\stateVecLast}{\stateVec^{\iter - 1}}
\newcommand{\residualIter}{\residual^\iter}
\newcommand{\gResidualIter}{\gResidual^\iter}
\newcommand{\param}{\boldsymbol{\mu}}
\newcommand{\paramSet}{\mathcal{D}}
\newcommand{\stateVecSize}{n}
\newcommand{\paramSetDim}{{n_{\param}}}
\newcommand{\interest}{\ve{z}}
\newcommand{\interestIter}{\ve{z}^\iter}
\newcommand{\interestFunc}{g}
\newcommand{\paramSetTrain}{\paramSet_{\text{train}}}
\newcommand{\ntimeTot}{\mathsf T}
\newcommand{\romBasis}{\mat{V}}
\newcommand{\romBasisStiefel}{\romBasis_*}

\newcommand{\romBasisCol}{\ve{v}}
\newcommand{\trialBasisOffset}{\bar{\ve{x}}}
\newcommand{\testBasis}{\mat{W}}
\newcommand{\romSize}{p}

\newcommand{\treeNodes}{\graphObj{V}}
\newcommand{\treeEdges}{\graphObj{E}}
\newcommand{\tree}{T}
\newcommand{\testTrialTrans}{\mat{A}^{\stateVecSize}(\stateVecIter; \param)}
\newcommand{\testTrialTransGen}{\mat{A}^{\stateVecSize}(\stateVec; \param)}
\newcommand{\initBasis}{\mat{\Phi}}
\newcommand{\initBasisCol}{\boldsymbol{\phi}}
\newcommand{\phiVec}{\boldsymbol{\phi}}
\newcommand{\ancestor}{\psi}
\newcommand{\children}{C}
\newcommand{\parent}{P}
\newcommand{\parentSpace}{\vecSpace{U}}
\newcommand{\childSpace}{\vecSpace{W}}
\newcommand{\leafSpace}{\vecSpace{L}}
\newcommand{\rootSpace}{\vecSpace{V}}
\newcommand{\frontier}{\graphObj{F}}
\newcommand{\frontierglobal}{\bar\frontier}
\newcommand{\decomp}{\graphObj{U}}
\newcommand{\treeLeaves}{\graphObj{L}}
\newcommand{\genFrontier}{\tilde{\frontier}}
\newcommand{\refine}[1]{#1^+}
\newcommand{\sieve}[2]{\proj_{#2}\left\{#1\right\}}
\newcommand{\matsieve}[2]{\proj_{#2}\left[#1\right]}
\newcommand{\tRefine}[2]{R_#2\left(#1\right)}
\newcommand{\pRefine}[3]{R_#3\left(#1; #2\right)}
\newcommand{\pRefineD}[2]{R\left(#1; #2\right)}
\newcommand{\romBasisC}{\romBasis^H}
\newcommand{\romStateVecC}{\romStateVec^H}
\newcommand{\romBasisF}{\romBasis^h}
\newcommand{\romStateVecF}{\romStateVec^h}
\newcommand{\prolong}{\mat{I}^h_H}
\newcommand{\adjoint}{\hat{\mat y}}
\newcommand{\adjointF}{\adjoint^h}
\newcommand{\adjointC}{\adjoint^H}
\newcommand{\prolongadjoint}{\adjoint^h_H}
\newcommand{\errorInd}{\ve{\delta}}
\newcommand{\errorIndC}{\errorInd^H}
\newcommand{\errorIndF}{\errorInd^h}
\newcommand{\frontierglobalF}{\frontierglobal^h}
\newcommand{\frontierglobalC}{\frontierglobal^H}
\newcommand{\snaps}{\mat{X}}
\newcommand{\romSnaps}{\hat{\snaps}}
\newcommand{\compSnaps}{q}
\newcommand{\compTime}{r}

\newcommand{\projSnaps}{\snaps_{\perp}}

\newcommand{\romProjSnaps}{\romSnaps_{\perp}}

\newcommand{\compVecs}{\mat{\Psi}}
\newcommand{\compVecsDummy}{\mat{\Xi}}

\newcommand{\romCompVecs}{\hat{\mat{\Psi}}}
\newcommand{\romCompVecsDummy}{\hat{\mat{\Xi}}}

\newcommand{\romInitBasis}{\hat{\initBasis}}
\newcommand{\romInitBasisCol}{\hat{\initBasisCol}}
\newcommand{\romInitBasisIter}{\romInitBasis^{(\compTime)}}

\newcommand{\corrector}{\hat{\mat{Z}}}
\newcommand{\compSize}{s}
\newcommand{\transComp}{\tilde{\mat{\Psi}}}
\newcommand{\transCompDummy}{\tilde{\mat{\Xi}}}
\newcommand{\projector}{\mat{P}}
\newcommand{\leafBasis}{\mat{Q}}
\newcommand{\canonicalVec}[1]{\mat{e}_{#1}}

\newcommand{\frontierH}{\mathcal{H}}
\newcommand{\projMetric}{\textsc{ProjectedMetric}}
\newcommand{\metric}{\hat{\mat{M}}}
\newcommand{\initBasisZ}{\initBasis^{(0)}}
\newcommand{\initBasisIter}{\initBasis^{(\compTime)}}
\newcommand{\initBasisIterLast}{\initBasis^{(\compTime - 1)}}
\newcommand{\romInitBasisZ}{\hat{\initBasis}^{(0)}}
\newcommand{\initBasisColZ}{\initBasisCol^{(0)}}
\newcommand{\initBasisColIter}{\initBasisCol^{(r)}}

\newcommand{\romInitBasisColIter}{\romInitBasisCol^{(r)}}
\newcommand{\tol}{\varepsilon}
\newcommand{\romTol}{\tol_{ROM}}
\newcommand{\stiefel}[1]{\mathbb{R}^{#1}_*}
\newcommand{\R}[1]{\mathbb{R}^{#1}}
\newcommand{\Rn}{\R{\stateVecSize}}

\newcommand{\decompSize}[1]{n_{#1}}
\newcommand{\initRomSize}{\romSize_0}
\newcommand{\initRomSizeIter}{\initRomSize^{(\compTime)}}
\newcommand{\initRomSizeIterLast}{\initRomSize^{(\compTime - 1)}}
\newcommand{\frontierBefore}{\frontier^{(\compTime - 1)}}
\newcommand{\initBasisSize}{\romSize_0}
\newcommand{\zero}{{\mat 0}}
\newcommand{\initBasisSizeBefore}{\initBasisSize^{(\compTime - 1)}}
\newcommand{\frontierMeet}{\frontier_M}
\newcommand{\globalancestor}{\chi}

\newcommand{\initBasisColIterBefore}{\initBasisCol^{(\compTime - 1)}}
\newcommand{\compressionCoeffs}[1]{\mat{A}^{(#1)}}
\newcommand{\explicitFlag}{\textsc{IsExplicit}}
\newcommand{\defConj}{\textsc{DeferredConjugation}}
\newcommand{\frontierIter}{\frontier^{(\compTime)}}
\newcommand{\treeSnaps}{\snaps}
\newcommand{\leafBasisCol}{\ve{q}}
\newcommand{\treeSnapsTrans}{\mat{Y}}
\newcommand{\dof}{\ve{d}}
\newcommand{\dofTrans}{\widetilde{\dof}}
\newcommand{\voltage}{v}
\newcommand{\recVoltage}{w}
\newcommand{\dofVoltage}{\mathcal{X}_\voltage}
\newcommand{\dofRecVoltage}{\mathcal{X}_\recVoltage}
\newcommand{\snapCount}{a}

\newcommand{\todo}[1]{{\color{red} #1}}

\DeclareMathOperator*{\argmin}{arg\,min}

\begin{document}
\numberwithin{equation}{section}
\begin{frontmatter}
  \title{Online adaptive basis refinement and compression for reduced-order models via vector-space sieving}

	\author[stanford]{Philip A.\ Etter\corref{stanfordcor}}
\ead{paetter@stanford.edu}
\address[stanford]{Stanford University}
\fntext[stanfordcor]{Institute for Computational and Mathematical Engineering, 496 Lomita Mall, Stanford University, Stanford, CA 94305-3035.}
	\author[sandia]{Kevin T.\ Carlberg\fnref{sandiacor}}
\ead{ktcarlb@sandia.gov}
\ead[url]{sandia.gov/~ktcarlb}

\address[sandia]{Sandia National Laboratories}
\fntext[sandiacor]{7011 East Ave, MS 9159, Livermore, CA 94550.
This paper describes objective technical results and analysis. Any subjective
views or opinions that might be expressed in the paper do not necessarily
represent the views of the U.S. Department of Energy or the United States
Government. Sandia National Laboratories is a multimission laboratory managed
and operated by National Technology \& Engineering Solutions of Sandia, LLC, a
wholly owned subsidiary of Honeywell International Inc., for the U.S.\
Department of Energy's National Nuclear Security Administration under contract
DE-NA0003525.}

\begin{abstract}
In many applications, projection-based reduced-order models
(ROMs) have demonstrated the ability to provide rapid approximate solutions to
high-fidelity full-order models (FOMs). However, there is no \textit{a priori}
assurance that these approximate solutions are accurate; their accuracy
depends on the ability of the low-dimensional trial basis to represent the FOM
solution.
As a result, ROMs can generate inaccurate approximate solutions, e.g., when the FOM
solution at the online prediction point is not well represented by training
data used to construct the trial basis.  To address this fundamental
deficiency of standard model-reduction approaches, this work proposes a novel
online-adaptive mechanism for efficiently enriching the trial
basis in a manner that ensures convergence of the ROM to the FOM, yet does not
incur any FOM solves. The mechanism is based
on the previously proposed adaptive $h$-refinement method for ROMs
\cite{carlberg2015adaptive}, but improves upon this work in two crucial ways.
First, the proposed method enables basis refinement with respect to
\textit{any orthogonal basis} (not just the Kronecker basis), thereby
generalizing the refinement mechanism and enabling it to be tailored to the
physics characterizing the problem at hand. Second, the proposed method
provides a fast online algorithm for periodically compressing the enriched basis via an
efficient proper orthogonal decomposition (POD) method,
which does not
incur any operations that scale with the FOM dimension.
These two features allow the
proposed method to serve as (1) a \textit{failsafe mechanism for ROMs}, as the
method enables the ROM to satisfy any prescribed error tolerance online (even
in the case of inadequate training), and (2) an \textit{efficient online
basis-adaptation mechanism}, as the combination of basis enrichment and
compression enables the basis to adapt online while controlling its
dimension.
\end{abstract}

\begin{keyword}

adaptive refinement\sep adaptive coarsening\sep model reduction\sep
dual-weighted residual\sep adjoint error estimation
\end{keyword}

\end{frontmatter}

\section{Introduction}\label{sec:intro}

Physics-based modeling and simulation now plays an essential role across a
wide range of design, control, decision-making, and discovery applications in
science and engineering. However, as such simulations are playing an
increasingly important role, greater demands are being
placed on their fidelity. As a result, these models are often characterized by
fine spatiotemporal resolution that results in large-scale models whose
simulation can consume months on a supercomputer. This computational burden
precludes such high-fidelity models from being employed in important
\textit{real-time} or \textit{many-query} scenarios that require the
(parameterized) model to be simulated very quickly (e.g., model predictive
control) or thousands of times (e.g., design optimization).


Naturally, the importance of real-time and many-query problems has resulted in a demand for fast
approximation techniques to mitigate the computational bottleneck of
simulating the high-fidelity full-order model (FOM). Projection-based
reduced-order models (ROMs) comprise a promising class of such techniques that
provide fast (yet often accurate) approximations by reducing the
dimensionality and complexity of the FOM. ROMs are typically deployed in two
stages. First, during the (training) \emph{offline stage}, these methods execute
computationally expensive training tasks to construct a low-dimensional
`trial' subspace.  In many cases (e.g., proper orthogonal decomposition, the
reduced-basis method), these training tasks entail simulating the FOM for
several points in parameter space.  
%
Second, during the (deployed) \emph{online stage}, the ROM computes fast
approximate solutions at arbitrary points in the parameter space via projection:
it computes a solution in the low-dimensional trial subspace by enforcing the
high-fidelity-model residual to be orthogonal to a test
subspace of the same (low) dimension.  If the FOM is nonlinear in the state or
nonaffine in functions of the parameters, `hyper-reduction' methods are
required to ensure the complexity of the ROM remains independent of the FOM
dimension; see
Refs.~\cite{rozza2007reduced,benner2013survey,hesthaven2015certified} for
reviews of the subject.

The predictive accuracy of the ROM depends on the ability of the trial
subspace to represent the FOM solution at the online prediction point of
interest.  In many practical scenarios, the trial subspace
is deficient for this purpose; for example, if the basis constructed using
proper orthogonal decomposition (POD) and the physics characterizing the online
prediction point (e.g., discontinuities) were not observed during training,
then the ROM will be incapable of resolving this phenomenon and will yield an
inaccurate online approximation.  This simple observation exposes a
fundamental deficiency of standard model-reduction approaches: ROMs are not
ensured to produce accurate approximations 
when they are deployed at online points
whose FOM-solution characteristics were absent from the training data. In
other words, ROMs are subject to
generalization error.  
In our experience, the inability of ROMs to provide \textit{a priori}
assurances of accurate predictions at arbitrary online prediction points is
\textit{the primary obstacle} to the widespread adoption of ROMs in science
and engineering.

Researchers have developed two categories of approaches that aim to overcome this
deficiency.  
In our view, such
a method should satisfy two desiderata: it should
\begin{enumerate}
\item ensure monotone convergence of the ROM to the FOM, and
\item incur an operation count that is independent of the FOM dimension.
\end{enumerate}
Unfortunately, no currently available method satisfies both desiderata;
however, some methods satisfy one of them.  
The
first (and largest) category of approaches reverts to the FOM when the ROM is
detected to be inaccurate, and subsequently enriches the ROM with information
gleaned from the FOM solve.
In the simplest case, one can add the FOM solution (possibly computed
over a subdomain only) to the trial basis and proceed with the enriched ROM
\cite{arian2000trust, ryckelynck2005priori,weickum2009multi,kim2009skipping,
teng2015subspace, ohlberger2015error}. A more involved approach enriches the
trial subspace by generating a Krylov subspace using the FOM
\cite{carlberg2009adaptive,carlbergKrylov2015}.  While this category of
methods indeed enriches the trial subspace with missing solution
characteristics, it incurs an operation count that depends on the FOM
dimension, which ultimately precludes online efficiency.

The second category of approaches involves adapting the low-dimensional basis
online without explicitly solving the full-order model. The first class of methods in this
category comprises `dynamic sampling' approaches.  Peherstorfer et
al.~\cite{peherstorfer2015online} propose to continually sub-sample new
entries of the FOM residual in order to compute online updates to an adaptive
basis used to represent both the velocity and the state; they also propose
variants that employ gappy POD to compensate for limited measurements
\cite{peherstorfer2016dynamic} and propagate coherent structures in
transport-dominated problems \cite{peherstorfer2018model}. While these methods
are characterized by an operation
count that is independent of the FOM dimension, they
do not ensure convergence to the FOM. The second class of
methods corresponds to `geometric subspace' techniques, which employ geodesic or
tangent-space structure to adapt the trial basis. Within this class, Zimmerman
et al.\  \cite{zimmermann2018geometric} propose performing online geometric updates of
the trial basis by solving a Grassmannian Rank-One Subspace Estimation (GROUSE)
optimization problem. Alternatively Peng et al.\
\cite{peng2014online} propose employing tangent-space information to enrich the
trial basis. Unfortunately, the former does not ensure 
convergence to the FOM, while the later performs POD on FOM snapshot
data, which incurs an operation count that depends on the FOM dimension. The
third class of methods corresponds to `adaptive local--global' methods \cite{efendiev2016online,
yang2017online}, which construct both a global ROM and a set of local ROMs for
different regions of the spatial domain. When the global ROM is deemed to be
inaccurate, these methods enrich the global trial basis with the solution
computed from approximately solving the FOM linear system using a local ROM.
However, these methods also fail to ensure convergence to
the FOM.

In this work, we propose an adaptive method that---when equipped with
hyper-reduction---satisfies both of the above
desiderata. The method is based on the previously proposed $h$-refinement
method for ROMs \cite{carlberg2015adaptive}, which is
analogous to mesh-adaptive $h$-refinement for finite-element, finite-volume,
and discontinuous-Galerkin discretizations. ROM $h$-refinement enriches the low-dimensional
trial basis online by `splitting' a given basis vector into multiple vectors
with disjoint discrete support. The approach identifies basis vectors to
split using a dual-weighted-residual approach that aims to reduce the error in
an output quantity of interest. Ultimately, the method generates a hierarchy
of subspaces online that converges to the full space while ensuring an
operation count that is independent of the FOM dimension; this provides a
failsafe mechanism for the ROM, as it enables the ROM to satisfy any
prescribed error tolerance regardless of its original fidelity.  Despite these
attractive attributes, this method has two noticeable areas for improvement:
\begin{enumerate}
	\item \label{drawback:split}The method always performs online basis refinement by `splitting' basis
	vectors entry-wise. While this refinement mechanism can lead to rapid
		convergence for some problems, the convergence rate may be quite slow for
		others, i.e., many splits may be required before the basis can adequately
		represent the FOM solution.
	\item \label{drawback:reset}The ROM dimension is controlled by simply \textit{resetting} the
		refined basis to the original basis after a prescribed number of time
		steps,
		before which the basis dimension grows monotonically. This reset
		effectively discards all information gained about the characteristics of
		the online FOM solution during the refinement procedure.
\end{enumerate}
This paper proposes techniques that address each of these drawbacks.  To
address drawback \ref{drawback:split}, we propose a generalization of the
$h$-refinement basis-splitting mechanism. In particular, we propose to perform
online refinement by projecting trial basis vectors onto progressively finer
orthogonal decompositions of their encompassing vector space, which we take to be $\Rn$.
Because we allow the refinement mechanism to be constructed from \textit{any} orthogonal decomposition of $\Rn$, this approach enables a more general class of
refinement mechanisms than simple entry-wise splitting.  This allows the
practitioner to tailor the refinement mechanism to the particular physics of
the problem at hand.  
 We present this contribution in sections \ref{sec:math_framework},
 \ref{sec:basis_refinement}.
To address drawback \ref{drawback:reset}, we propose a
fast online method for basis compression that computes the POD of refined-ROM
solutions while incurring an $\stateVecSize$-independent operation count.  Not
only does this basis-compression procedure control the ROM dimension online,
it also enables the basis to adaptively evolve in time.  Numerical experiments
demonstrate that both of these contributions enable the new approach to yield
significant performance improvements over the original $h$-refinement method.
We present this contribution in section \ref{sec:online_basis_compression}.


We briefly remark that some `adaptive' methods exist that tailor the ROM to
specific regions of the parameter space \cite{amsallem2008interpolation,
amsallem2009method, eftang2010hp, haasdonk2011training, drohmann2011adaptive,
peherstorfer2014localized}, time domain \cite{drohmann2011adaptive,
dihlmann2011model}, and the state space \cite{amsallem2012nonlinear,
peherstorfer2014localized}. The connotation of `adaptive' is different in the
context of the present work. The works cited above are `adaptive' in that they
construct separate ROMs for each region offline with the objective of reducing
the ROM dimension. This is quite different from the `adaptive' method we
propose, which performs \textit{a posteriori} refinement of the ROM during the
online stage. 

The remainder of the paper is organized as follows. Section \ref{sec:probForm}
formulates the problem, including the FOM (section \ref{sec:FOM}), the ROM
(section \ref{sec:ROM}), and objectives of this work (section \ref{sec:objectives}).
Section \ref{sec:math_framework} provides the mathematical framework of the
proposed refinement mechanism, including notation (section
\ref{sec:notation}), a description of vector space sieving (section
\ref{sec:sieve}), the refinement tree (section
\ref{sec:refinement_tree}), and frontiers (section
\ref{sec:frontiers}). Section \ref{sec:schema} provides the algorithm schema.
Section \ref{sec:basis_refinement} describes the basis refinement mechanism,
including frontier refinement (section \ref{sec:frontierRefine}),
dual-weighted-residual error indicators (section \ref{sec:dwrei}), the
refinement algorithm (section \ref{sec:refineAlg}), and a technique to resolve
ill-conditioning and ensure linear independence of the refined basis (section
\ref{sec:illCond}). Section \ref{sec:tree_construction} describes construction
of the refinement tree. Section \ref{sec:online_basis_compression} describes
the online basis-compression algorithm, including compression via
metric-corrected POD (section \ref{sec:compressionmetric}), a description on
computing the metric (section \ref{sec:metric}), a characterization of the
`meet' of frontiers (section \ref{sec:meetFrontiers}), a method for
updating the projected metrics (section \ref{sec:update_proj_metrics}), and
finally the basis-compression algorithm itself (section
\ref{sec:basiscompressionalg}). Section \ref{sec:complexity} discusses the
algorithm's complexity, and section \ref{sec:numericalExp} provides numerical experiments that
illustrate the benefits of the proposed method. Section \ref{sec:conclusions}
provides conclusions and an outlook for future work.  Finally, 
\ref{sec:proofs} provides proposition proofs.

\section{Problem formulation}\label{sec:probForm}

\subsection{Full-order model}\label{sec:FOM}

We consider solving a time sequence of parameterized systems of algebraic
equations
\begin{equation} \label{fom}
	\residualIter(\stateVecIter; \param) = 0 \,,\quad
\iter \in \nat{\ntimeTot},
\end{equation}
where $\nat{\ntimeTot} \equiv \{1, \ldots, \ntimeTot\}$, $\stateVecIter \in \Rn$ denotes the
state of the system at the $\iter$th time instance, $\param \in
\paramSet \subset \RR{\paramSetDim}$ denotes input parameters, and
$\residualIter \colon \Rn \times \paramSet \rightarrow \Rn$ denotes the
residual operator 
at the $\iter$th time instance.\footnote{In this work, we consider $\subset$ to denote
subsets, and $\subsetneq$ to denote proper subsets.}
This problem setting arises in a broad range of applications, e.g., from the
spatial discretization of elliptic partial differential equations (PDEs) (e.g., where
$\residualIter:(\stateVec; \param) \mapsto \Amat(\param) \stateVec(\param)
- \bvec(\param)$ and $\ntimeTot = 1$ for elliptic linear PDEs), or from the implicit time discretization of
systems of ordinary differential equations (ODEs) $\dot{\stateVec} = \velocity(\stateVec; \mu)$ (e.g., where
$\residualIter \colon (\ve{x}; \param) \mapsto \ve{x} - \stateVecLast - \Delta
t \, \velocity(\ve{x}; \param)$ in the case of backward Euler).
Often, the practitioner is primarily interested in a quantity 
that is a functional of the state, i.e.,
\begin{equation} \label{interest_fom}
	\interestIter \defeq \interestFunc(\stateVecIter; \param) \,,\quad\
\iter \in \nat{\ntimeTot},
\end{equation}
where $\interestIter \in \mathbb{R}$ and $\interestFunc :
\mathbb{R}^{\stateVecSize} \times \paramSet \rightarrow \mathbb{R}$. 

In many scenarios, the dimension $\stateVecSize$ of the full-order-model (FOM)
equations (\ref{fom}) is large, which makes computing their solution
prohibitively expensive in real-time and many-query settings that demand fast
evaluation of
the input--output map $\param \mapsto \{ \interest^1, \ldots, \interest^\ntimeTot
\}$.  This work considers projection-based reduced-order models (ROMs) to
mitigate this computational burden.

\subsection{Reduced-order model}\label{sec:ROM}

Model-reduction methods typically execute two stages. First, these methods
perform a
computationally expensive (training) \textit{offline stage}
to construct (1) a low-dimensional trial
subspace spanned by the columns of the basis (in matrix form) $\romBasis \in \stiefel{\stateVecSize \times
\romSize}$ (with $\romSize\ll\stateVecSize$), and (2) an associated test
subspace spanned by the columns of the basis
matrix $\testBasis \in \stiefel{\stateVecSize \times \romSize}$.
Here $\stiefel{m\times n}$ denotes the set of full-column-rank $m\times n$
matrices (i.e., the non-compact Stiefel manifold).
In the case of POD, the offline stage comprises
solving the FOM equations
\eqref{fom} for $\param\in\paramSetTrain\subsetneq \paramSet$,
collecting the associated solution snapshots $\stateVecIter(\param)$ for
$\iter=1,\ldots,\ntimeTot$ and $\param\in\paramSetTrain$, and setting the trial basis to the dominant left singular
vectors of the `snapshot matrix' whose columns correspond to these solution snapshots.
Then, during the computationally inexpensive (deployed) \textit{online stage}, 
these methods compute approximate solutions to the FOM equations \eqref{fom}
via projection: they seek approximate solutions in the 
affine trial subspace $\trialBasisOffset +
\range(\romBasis) \subset \mathbb{R}^{\stateVecSize}$, where
$\trialBasisOffset\in\Rn$ denotes a reference state,
and enforce orthogonality of the FOM residual to the test basis, i.e.,
they compute generalized coordinates
$\romStateVecIter(\param) \in \mathbb{R}^\romSize$ satisfying 
\begin{equation} \label{rom}
	\testBasis^T \residualIter\left(\trialBasisOffset + \romBasis
	\romStateVecIter; \param\right) = \zero\, .
\end{equation}
The ROM approximate solution then corresponds to 
$\trialBasisOffset + \romBasis
	\romStateVecIter(\param)$.
When the residual operator $\residualIter$ is affine in its first argument and
affine in functions of its second argument,
solving the ROM equations (\ref{rom}) can be performed efficiently by
precomputing low-dimensional affine operators during the offline stage and
solving a $\romSize \times
\romSize$ linear system during the online stage. However, when the residual operator $\residualIter$
is nonlinear, then `hyper-reduction' methods 
such as
empirical
interpolation \cite{barrault2004eim,chaturantabut2010nonlinear}, collocation
\cite{legresley2006application, astrid2008missing, ryckelynck2005priori},
or gappy proper orthogonal decomposition (POD)
\cite{astrid2008missing, carlberg2013gnat}
must be employed to ensure that assembling
the reduced equations (\ref{rom}) incurs an $\stateVecSize$-independent operation count. While
we do not explicitly consider these hyper-reduction methods in this work, the
proposed methods are 
forward compatible with these techniques. Future work will consider their
integration.

Many ROM techniques employ a test basis $\testBasis$ corresponding to a linear
transformation of the trial basis $\romBasis$ such that
\begin{equation} \label{test_trial}
	\testBasis = \testTrialTrans \, \romBasis \,,
\end{equation}
where $\testTrialTrans:\RR{\stateVecSize}\times \paramSet \rightarrow
\mathbb{R}^{\stateVecSize \times \stateVecSize}$ denotes
a transformation matrix that depends in general on the state and parameters.
For example, Galerkin projection
employs $\testTrialTrans = \identity$; balanced truncation employs $\testTrialTrans =
\mat{Q}$, where $\mat{Q}$ is the observability Gramian of the linear
time-invariant system; least-squares Petrov--Galerkin projection
\cite{legresley2006application,bui2008model,bui2008parametric,CarlbergGappy,carlbergGalDiscOpt} 
employs $\testTrialTrans = \frac{\partial \residualIter }{\partial \ve{x}}(\stateVecIter; \param)$; for
linearized compressible-flow problems, $\testTrialTrans$ can be chosen to
ensure stability \cite{barone2009stable}. When the test basis can be
expressed in the form of Eq.~\eqref{test_trial} for some 
transformation matrix $\testTrialTrans$, the Petrov--Galerkin projection
characterizing the ROM equations
(\ref{rom}) is
equivalent to Galerkin projection performed on a modified residual
$\gResidualIter:(\stateVec;\param)\mapsto \testTrialTransGen^T
\residualIter(\stateVec;\param)$, and
Eq.~\eqref{rom} is equivalent to 
\begin{equation} \label{galerkin_rom}
	\romBasis^T \gResidualIter\left(\trialBasisOffset + \romBasis
	\romStateVecIter; \param\right) = \zero \,.
\end{equation}
The remainder of the paper restricts attention to the Galerkin projection
\eqref{galerkin_rom}, as it corresponds to a wide range of practical
Petrov--Galerkin ROMs characterized by a test basis of the form
\eqref{test_trial}.

\subsection{Objectives}\label{sec:objectives}

If the projection error of the FOM solution $\stateVecIter(\param)$ onto the
trial subspace $\trialBasisOffset+\range(\romBasis)$ is large at a particular
time instance, then the ROM solution $\trialBasisOffset + \romBasis
\romStateVecIter(\param)$  will provide a poor approximation to the FOM
solution $\stateVecIter(\param)$.  
The original
adaptive $h$-refinement method \cite{carlberg2015adaptive} provided a
promising approach for enriching the basis
$\romBasis$ in this scenario. However, as described in the
introduction, this method exhibits two significant shortcomings:
\begin{enumerate}
\item The method always performs refinement by `splitting' basis vectors entry-wise.
	This refinement mechanism may not always lead to fast convergence. For
		example, when the FOM corresponds to the discretization of a PDE problem,
		basis splitting can introduce sharp gradients in the refined basis
		vectors; if the PDE solution is expected to be smooth,
		then this refinement mechanism can yield slow convergence. 
		Our first objective is to address this
		shortcoming by introducing a new mathematical framework for basis
		refinement, which we present in section \ref{sec:math_framework}. We then
		use this framework to generalize the original basis-splitting refinement
		mechanism in sections \ref{sec:basis_refinement} and
		\ref{sec:tree_construction}. The resulting extension enables a broader
		class of refinement mechanisms than simply entry-wise basis splitting.
\item The method controls the ROM dimension by simply resetting the basis to
	the original one after a prescribed number of time steps. Our second
		objective is to improve upon this reset strategy. In particular, we
		observe that basis refinement provides valuable information about online FOM solution
		components that were not representable with the original basis. Our
		objective is to devise a method of periodically compressing the refined
		basis to control its dimensionality without discarding this additional
		information. We address this objective in section
		\ref{sec:online_basis_compression}.
\end{enumerate}

\section{Mathematical framework} \label{sec:math_framework}

We assume that we are given an initial basis $\initBasis \in
\stiefel{\stateVecSize \times \romSize_0}$ such that $\romBasis = \initBasis$ initially.  We seek to enrich
this basis $\initBasis$ by recursively decomposing columns of $\initBasis$ until 
the basis is sufficiently rich to accurately represent the FOM
solution $\stateVecIter(\param)$. In this section, we establish
mathematical preliminaries that will be leveraged to achieve this goal. 
Section \ref{sec:notation} introduces required notation.
Section \ref{sec:sieve} introduces \textit{vector-space sieving}, which is 
the fundamental refinement mechanism we employ.
Section \ref{sec:refinement_tree} introduces the \textit{refinement
tree} data structure that helps to prescribe the refinement strategy. Section
\ref{sec:frontiers} introduces the notion of a refinement-tree
\textit{frontier}, which allows the algorithm to monitor the current level of
refinement applied to a basis vector.

\subsection{Notation}\label{sec:notation}

We use somewhat nonstandard notation for matrices that enables their
entries to be indexed by elements of arbitrary sets (not just integers).
For this purpose, we write an $A \times B$ matrix for arbitrary sets $A$ and $B$ as
\begin{equation}
M \in \mathbb{R}^{A \times B} \, 
\Leftrightarrow
	M:A\times B\longrightarrow \RR{}
\end{equation}
such that
$\mathbb{R}^{A \times B}$ denotes the set of mappings from 
$A \times B$ to $\mathbb{R}$. We denote
element $(a, b)$ of $M$ as $M_{a, b}\equiv M(a,b)\in\RR{}$ for $a \in A$ and $b \in B$. For
these generalized matrices, the matrix product is defined as expected,
i.e.,
for $M_1 \in \mathbb{R}^{A \times B}$ and $M_2 \in \mathbb{R}^{B \times C}$,
the matrix product $M_1 M_2 \in \mathbb{R}^{A \times C}$ is 
\begin{equation}
[M_1 M_2]_{a, c} = \sum_{b \in B} [M_1]_{a, b} [M_2]_{b, c} \,.
\end{equation}
Given an array of these matrices $M_{ij} \in \mathbb{R}^{A_i \times B_j}$,
they can be concatenated into a new matrix according to
\begin{equation}
M = \left[\begin{array}{ccc} M_{11} & \cdots & M_{1r} \\ \vdots & \ddots &
\vdots \\ M_{s1} & \cdots & M_{sr} \end{array}\right] \in \mathbb{R}^{\left(\bigsqcup_i A_i \right) \times \left(\bigsqcup_j B_j\right)} \,,
\end{equation}
where $\bigsqcup_i A_i$ denotes the disjoint union of the sets $A_i$. We use
the disjoint union because the sets $A_i$
may have non-empty intersection, and we must treat the same index appearing in
$A_i$ and 
$A_j$ with $i\neq j$
differently.
In this notation, the 
standard set of $n
\times m$ real-valued matrices is formally represented as
$\mathbb{R}^{\nat{n} \times \nat{m}}$ (not 
$\mathbb{R}^{n \times m}$) although we use the two interchangably.

Finally, if $M \in \mathbb{R}^{A
\times B}$, then for $\tilde{A} \subset A$ and $\tilde{B} \subset B$, we denote the
$\tilde{A}, \tilde{B}$ submatrix of $M$ as
\begin{equation}
M_{\tilde{A}, \tilde{B}} \in \mathbb{R}^{\tilde{A} \times \tilde{B}} \,,
\end{equation}
and its entries are simply the corresponding ones in $M$, i.e.,
$(M_{\tilde{A}, \tilde{B}})_{a, b} = M_{a, b}$ for $a\in\tilde A$ and
$b\in\tilde B$. If $\tilde{A} =
A$ or $\tilde{B} =B$, we use a colon as a shorthand, i.e., $M_{:,\tilde{B}}
\equiv M_{A, \tilde{B}}$ and $M_{\tilde{A}, :} \equiv M_{\tilde{A}, B}$.

\subsection{Vector-space sieving} \label{sec:sieve}
We enrich the initial basis
$\initBasis$
using the idea of `sieving' a
vector through a chosen decomposition of a $\Rn$. 
More precisely, consider a decomposition $\decomp \equiv \{ \parentSpace_1, \ldots,
\parentSpace_{\decompSize{\decomp}} \}$ of a `parent' inner-product space 
$(\rootSpace, \langle \cdot,\cdot \rangle)$
(e.g., $\rootSpace=\Rn$ with $\langle \cdot,\cdot \rangle$ the Euclidean inner
product) such that
\begin{equation}
\rootSpace = \sum_{\parentSpace \in \decomp} \parentSpace \,.
\end{equation}
Then, any vector $\phiVec \in \rootSpace$ 
can be decomposed into components $\phiVec_{\parentSpace} \in \parentSpace$
for all $\parentSpace\in\decomp$,
i.e.,
\begin{equation} \label{sieve}
\phiVec = \sum_{\parentSpace \in \decomp} \phiVec_\parentSpace \, .
\end{equation}
If the vector spaces $\parentSpace_i$ are orthogonal in the inner product
$\langle\cdot,\cdot\rangle$, then this
decomposition is unique and is given by the orthogonal projection of $\phiVec$
onto the vector spaces $\parentSpace_i$, i.e.,
\begin{equation} \label{sieve_proj}
\phiVec_\parentSpace =
	\proj_{\parentSpace}(\phiVec),\quad\parentSpace\in\decomp \,,
\end{equation}
where 
$\proj_{\parentSpace}$ denotes
the orthogonal projector onto $\parentSpace$, i.e.,
\begin{equation}
\proj_{\parentSpace}:\phiVec \mapsto \argmin_{\ve{u} \in \parentSpace}
	\|\phiVec - \ve{u}\|,
\end{equation}
where $\|v\|\equiv\sqrt{\langle v, v\rangle}$.
Henceforth, we consider all decompositions to be orthogonal.  We now
formally define an orthogonal decomposition of a vector space and the 
sieve of a vector through one.

\begin{definition}[orthogonal decomposition]
We refer to a collection of vector spaces $\decomp \equiv \{ \parentSpace_1, \ldots,
	\parentSpace_{\decompSize{\decomp}} \}$ with each $\parentSpace_i$ a
	nontrivial subspace of an inner-product space $\rootSpace$ as an
	\textbf{orthogonal decomposition} of $\rootSpace$ if
\begin{enumerate}
\item the subspaces sum to $\rootSpace$, i.e.,
\begin{equation}
	\rootSpace = \sum_{\parentSpace \in \decomp} \parentSpace,\quad\text{and}
\end{equation}
\item the subspaces are orthogonal,
	i.e.,
\begin{equation}
\parentSpace_i \perp \parentSpace_j, \quad \forall\parentSpace_i, \parentSpace_j \in
	\decomp,\ i\neq j\,,
\end{equation}
where the operator $ \perp$ is defined as follows: $\parentSpace_i
		\perp \parentSpace_j $ if and only if $\langle u_i,u_j\rangle=0$ for all
		$u_i\in\parentSpace_i,\,u_j\in \parentSpace_j$.
\end{enumerate}
\end{definition}

\begin{definition}[sieve]
	We denote the \textbf{sieve} of a vector $\phiVec \in \rootSpace$ through an
	orthogonal decomposition $\decomp \equiv \{ \parentSpace_1, \ldots,
	\parentSpace_{\decompSize{\decomp}} \}$ of an inner-product space
	$\rootSpace$ by $\sieve{\phiVec}{\decomp}$ and define it as
\begin{equation}
\sieve{\phiVec}{\decomp} \equiv \left\{ \proj_{\parentSpace_1}(\phiVec), \ldots,
	\proj_{\parentSpace_{\decompSize{\decomp}}}(\phiVec)\right\} \subset \rootSpace \,.
\end{equation}
When $\rootSpace = \mathbb{R}^\stateVecSize$ and the above is desired in matrix form, we instead write
\begin{equation}
\matsieve{\phiVec}{\decomp} \equiv \left[\begin{array}{ccc}
\proj_{\parentSpace_1}(\phiVec) & \cdots &
\proj_{\parentSpace_m}(\phiVec)\end{array}\right] \in \R{\nat{\stateVecSize} \times \decomp} \,.
\end{equation}
\end{definition}

\begin{remark}[Vector-space sieving as a generalization of basis splitting]
Consider the particular case where the subspaces $\parentSpace_i$ are each spanned individually
	by Kronecker basis vectors (i.e., the canonical unit vectors). In this case,
	sieving a basis vector corresponds to `splitting' it entry-wise into vectors that have
	disjoint (discrete) support and whose nonzero elements are equal to the
	corresponding elements of the original vector. This is equivalent to the
	original $h$-refinement basis-splitting mechanism
	\cite{carlberg2015adaptive}; thus, vector-space sieving comprises a
	generalization of this mechanism. 
\end{remark}

If the initial basis $\initBasis$ is
not rich enough to accurately represent the FOM solution
$\stateVecIter(\param)$,
and we are given an orthogonal decomposition $\decomp\equiv\{\parentSpace_1
,\ldots, \parentSpace_{\decompSize{\decomp}}\}$, we can sieve selected columns of
$\initBasis$ through this decomposition 
to enrich this basis in the hope of better
representing $\stateVecIter(\param)$. 
We adopt this refinement mechanism for two reasons:

\begin{enumerate}
\item We desire a refinement mechanism that
	produces hierarchical trial subspaces,
\begin{equation}
	\range\left(\romBasis^{(0)}\right) \subset
	\range\left(\romBasis^{(1)}\right) \subset
	\range\left(\romBasis^{(2)}\right) \subset \cdots\,,
\end{equation}
where $\romBasis^{(0)} = \initBasis$ and $\romBasis^{(i)}$ denote progressive
		refinements of the original basis $\romBasis^{(0)}$. This
		property is desirable because it ensures that refinement equips the ROM with
		strictly higher fidelity. Vector-space sieving achieves this property via
		recursion:
		after sieving a vector $\phiVec$ through a decomposition $\decomp\equiv\{\parentSpace_1
,\ldots, \parentSpace_{\decompSize{\decomp}}\}$ to obtain the vectors
		$\phiVec_{\parentSpace}$ for $\parentSpace\in\decomp$, we can continue to sieve
		the vector $\phiVec$ by decomposing the subspaces $\parentSpace\in\decomp$
		into even finer 
		subspaces. For example, consider an orthogonal decomposition $\mathcal
		W\equiv\{\childSpace_1,\ldots,\childSpace_{n_{\mathcal W}}\}$ of
		$\parentSpace_1$.
		Sieving $\phiVec_{\parentSpace_1}$ through this decomposition yields
		vectors $\phiVec_{\parentSpace_1\childSpace}$, $\childSpace\in\mathcal
		W$ satisfying
\begin{equation} \label{double_sieve}
	\phiVec_{\parentSpace_1} = \sum_{\childSpace\in\mathcal
		W} \phiVec_{\parentSpace_1\childSpace} \,.
\end{equation}
The resulting refined basis satisfies
\begin{equation}
	\vecspan(\{\phiVec_{\parentSpace_1}, \ldots,
	\phiVec_{\parentSpace_{\decompSize{\decomp}}}\}) \subset
	\vecspan(\{\phiVec_{\parentSpace_1\childSpace_1}, \ldots,
	\phiVec_{\parentSpace_1\childSpace_{n_{\mathcal W}}}, \phiVec_{\parentSpace_2}, \ldots,
	\phiVec_{\parentSpace_{\decompSize{\decomp}}}\}) \,.
\end{equation}
Therefore, vector-space sieving can generate hierarchical subspaces by
		recursively decomposing $\rootSpace$ in this manner. This recursive
		decomposition naturally gives rise to a tree data structure, which we
		characterize in section \ref{sec:refinement_tree}.
\item We desire a progressive refinement mechanism that
	ensures the ROM converges to the FOM, as this property ensures that the
		refined ROM 
		can (eventually) recover the FOM solution $\stateVecIter(\param)$ given any
		initial basis. Again,
		this can be achieved by vector-space sieving. Suppose that
		$\rootSpace=\Rn$ and
		we employ an orthogonal decomposition that is as fine as possible, i.e.,
$\decomp\equiv\{\parentSpace_1
,\ldots, \parentSpace_{\stateVecSize}\}$,
		where $\dim(\parentSpace_i)=1$, $i=1,\ldots,\stateVecSize$. Each $\parentSpace_i$ is then spanned by a single vector $\leafBasisCol_i$, and sieving $\phiVec$ through this decomposition gives
\begin{equation}
	\phiVec = \sum_{i = 1}^{\stateVecSize} \alpha_i \leafBasisCol_i \,,
\end{equation}
		Assuming that $\alpha_i \neq 0$ for $i=1,\ldots,\stateVecSize$, we have 
\begin{equation}
	\vecspan(\{\phiVec_{\parentSpace_1}, \ldots, \phiVec_{\parentSpace_n}\}) =
	\vecspan(\{\leafBasisCol_1, \ldots,
	\leafBasisCol_n\}) = \Rn \,.
\end{equation}
Thus, barring degenerate cases, sieving a vector through the finest
		decompositions of $\mathbb{R}^\stateVecSize$ recovers the FOM, as the FOM
		is equivalent to a ROM
		with trial subspace of $\Rn$.
\end{enumerate}

We now describe the refinement tree, which equips vector-space sieving with
these two properties.

\subsection{The refinement tree} \label{sec:refinement_tree}

Recursively decomposing $\mathbb{R}^\stateVecSize$ into finer 
subspaces can be conceptualized as constructing a tree, as every decomposition
of a vector space $\parentSpace$ into $\childSpace_1, \ldots, \childSpace_m$
yields a natural parent--child relationship between the parent $\parentSpace$
and its children $\childSpace_1, \ldots, \childSpace_m$.

To make this  tree structure explicit, we define a \textit{refinement tree}.
We begin by specifying notation.  A directed graph $G = (\treeNodes,
\treeEdges)$ is given by a set of vertices $\treeNodes$ and a set of directed
edges $\treeEdges$. Each edge $e \in \treeEdges$ takes the form of a
tuple of two vertices $e = (v_1, v_2)$ with $v_1, v_2 \in \treeNodes$, where
$(v_1, v_2)$ is interpreted as an edge from $v_1$ to $v_2$. 
We refer to a directed graph $T$ as a rooted tree if it is acyclic 
and every vertex $v \in \treeNodes$ has in-degree one (i.e., one edge
enters $v$), except for a unique \textit{root} vertex $r \in \treeNodes$ that
has in-degree zero. For any non-root vertex $v \in \treeNodes$, if
$(u, v)$ is the unique edge entering vertex $v$, then we refer to $u$ as the
\textit{parent} of $v$ and conversely $v$ as a \textit{child}
of $u$. In general, we denote the parent of a vertex $v \in \treeNodes$ as $\parent_\tree(v)$. If there is a (potentially trivial) directed path from $u\in
\treeNodes$ to $v\in \treeNodes$, then we say that $u$ is an \textit{ancestor}
of $v$ and conversely $v$ is a \textit{descendant} of $u$. We refer to vertices with
out-degree zero (i.e., vertices with no children) as
\textit{leaves}; we denote the set of leaves
by
$\treeLeaves\subset\treeNodes$
. We
denote the children of a vertex $v$ in a tree $T$ as $\children_T(v)
\subsetneq \treeNodes$.

\begin{definition}[$\rootSpace$-refinement tree] A
	$\rootSpace$-\textbf{refinement tree} $\tree \equiv (\treeNodes, \treeEdges)$ for
	an inner-product space $\rootSpace$ is a rooted tree where the vertex set
	$\treeNodes$ is composed of subspaces of $\rootSpace$. Moreover, it has the
	following properties:

\begin{enumerate}
	\item\label{prop:rootV} \textit{The root is $\rootSpace$}. That is,
		$r=\rootSpace$.
	\item\label{prop:childrenOrth} \textit{Children correspond to an orthogonal decomposition of the parent}. The children of any vertex
		$\parentSpace \in \treeNodes$ correspond to an orthogonal decomposition of
		$\parentSpace$, i.e.,
 for all $\parentSpace \in \treeNodes$, we have
\begin{equation}
	\parentSpace = \sum_{(\parentSpace, \childSpace) \in \treeEdges}
	\childSpace \,,
\end{equation}
	and the children of $\parentSpace$ are all orthogonal, i.e.,
\begin{equation}
\childSpace \perp \mathbb{Y}, \qquad \forall (\parentSpace, \childSpace),
	(\parentSpace, \mathbb{Y}) \in \treeEdges \,.
\end{equation}
This property ensures that recursive sieving of a single vector will produce hierarchical trial subspaces.
\item \label{prop:leavesDimOne}\textit{Leaves have dimension one}. The leaves of the tree correspond to
	vector spaces that cannot be decomposed further, i.e.,
\begin{equation}
	\dim (\leafSpace) = 1, \qquad \forall \leafSpace \in \treeLeaves \,,
\end{equation}
This property enables recursive sieving of a single vector to recover a basis
		for the original space $\rootSpace$.
\end{enumerate}
\end{definition}

We now provide two basic results that lend intuition into
$\rootSpace$-refinement trees. \ref{sec:proofs} contains all proofs.

\begin{restatable}{proposition}{propsubset} \label{subset_proposition}
For any $\parentSpace, \childSpace \in \treeNodes$, $\childSpace$ is a
	descendant of $\parentSpace$ in the 
	$\rootSpace$-refinement tree $\tree \equiv (\treeNodes, \treeEdges)$
	iff $\childSpace \subset \parentSpace$.
\end{restatable}

\begin{restatable}{proposition}{propdesctree} \label{desc_proposition}
For any $\parentSpace, \childSpace \in \treeNodes$, $\parentSpace$ is not descendant from $\childSpace$ and $\childSpace$ is not descendant from $\parentSpace$ iff $\childSpace \perp \parentSpace$.
\end{restatable}

Furthermore, there exists a natural partial ordering on the vertex set
$\treeNodes$ of a $\rootSpace$-refinement tree $\tree$ given by the inclusion
relation $\subset$. Proposition \ref{subset_proposition} shows that the
inclusion relation $\subset$ is equivalent to the descendant relation in the
topology of the tree $\tree$.  This gives a notion of incomparability between
two spaces,

\begin{definition}[incomparability]\label{def:incomparability}
Two vector spaces $\parentSpace, \childSpace$ are \textbf{incomparable} if neither $\parentSpace \subset \childSpace$ nor $\childSpace \subset \parentSpace$.
\end{definition}

 \noindent The following corollary follows immediately from Propositions
 \ref{subset_proposition} and \ref{desc_proposition} and Definition
 \ref{def:incomparability}.

\begin{corollary} \label{incomparable_prop}
For two spaces $\parentSpace, \childSpace \in \treeNodes$, the following are equivalent:
\begin{enumerate}
\item $\parentSpace, \childSpace \in \treeNodes$ are incomparable.
\item $\parentSpace \perp \childSpace$.
\item $\parentSpace$ is not descendant from $\childSpace$ and $\childSpace$ is not descendant from $\parentSpace$.
\end{enumerate}
\end{corollary}

\noindent The following corollary also follows immediately from Proposition
\ref{desc_proposition}.

\begin{corollary} \label{leaves_corollary}
The leaves of a 
	$\rootSpace$-refinement tree 
	form an incomparable set.
\end{corollary}

\begin{remark}
Note that the leaves of a $\rootSpace$-refinement tree $\tree$ correspond to
	an orthogonal basis $\leafBasisCol_1, \ldots, \leafBasisCol_n$ for $\rootSpace $. This is
	because $\rootSpace$-refinement-tree Properties \ref{prop:rootV} and
	\ref{prop:childrenOrth} recursively imply that
\begin{equation} \label{leaf_decomp}
	\rootSpace = \sum_{\leafSpace \in \treeLeaves} \leafSpace \,.
\end{equation}
	Corollary \ref{leaves_corollary} tells us that these leaf spaces
	$\leafSpace$ are orthogonal, and Property \ref{prop:leavesDimOne} implies
	that each leaf space $\leafSpace$ corresponds to a single vector
	$\leafBasisCol_i$.
\end{remark}

\subsection{Frontiers}\label{sec:frontiers}

To use a refinement tree to sieve columns of the initial basis 
$\initBasis\in\stiefel{\stateVecSize \times \romSize_0}$, we must be able
to select a specific decomposition of $\mathbb{R}^\stateVecSize$ from an
$\mathbb{R}^\stateVecSize$-refinement tree.  To determine which
decompositions are admissible for given a refinement tree, we introduce the
notion of a `frontier'.

\begin{definition}[frontier]
An orthogonal decomposition $\decomp$ of $\rootSpace$ is a \textbf{frontier}
	of a $\rootSpace$-refinement tree $\tree \equiv (\treeNodes, \treeEdges)$ if
	$\decomp \subset \treeNodes$, i.e., all of the elements of $\decomp$
	correspond to vertices in the tree $\tree$.
\end{definition}

\begin{figure}[h]
\includegraphics[width=0.3\textwidth]{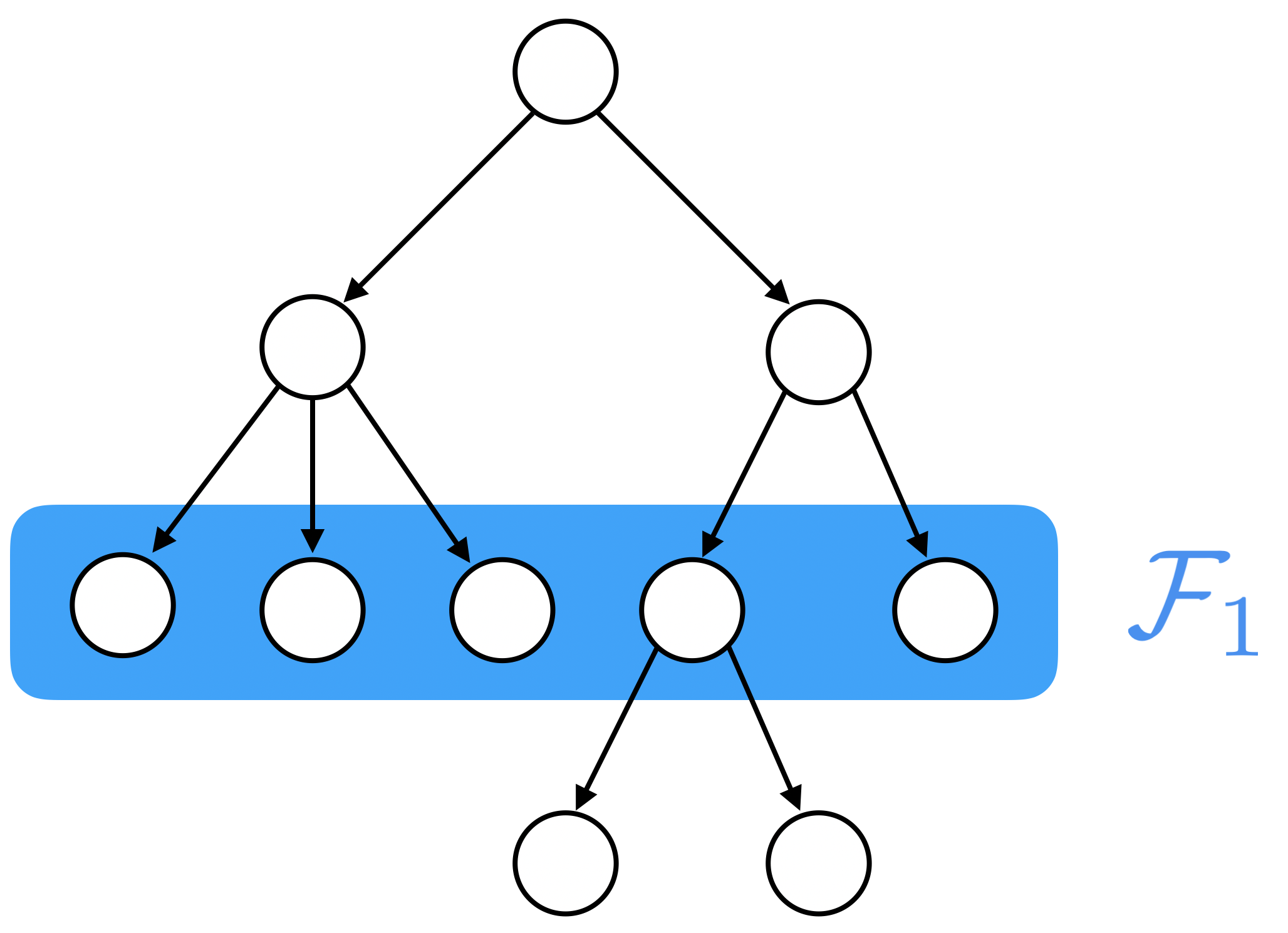}
\includegraphics[width=0.3\textwidth]{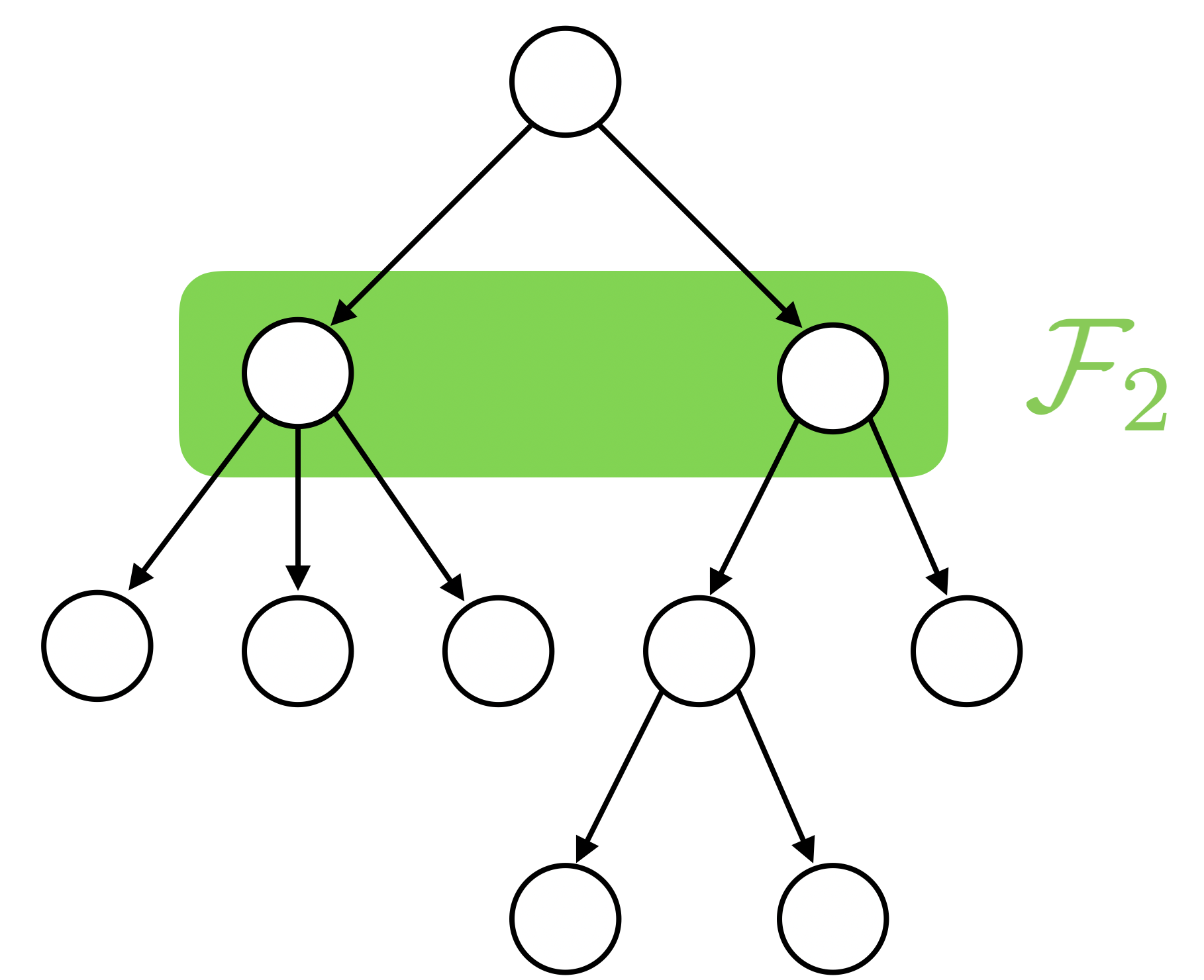}
\includegraphics[width=0.3\textwidth]{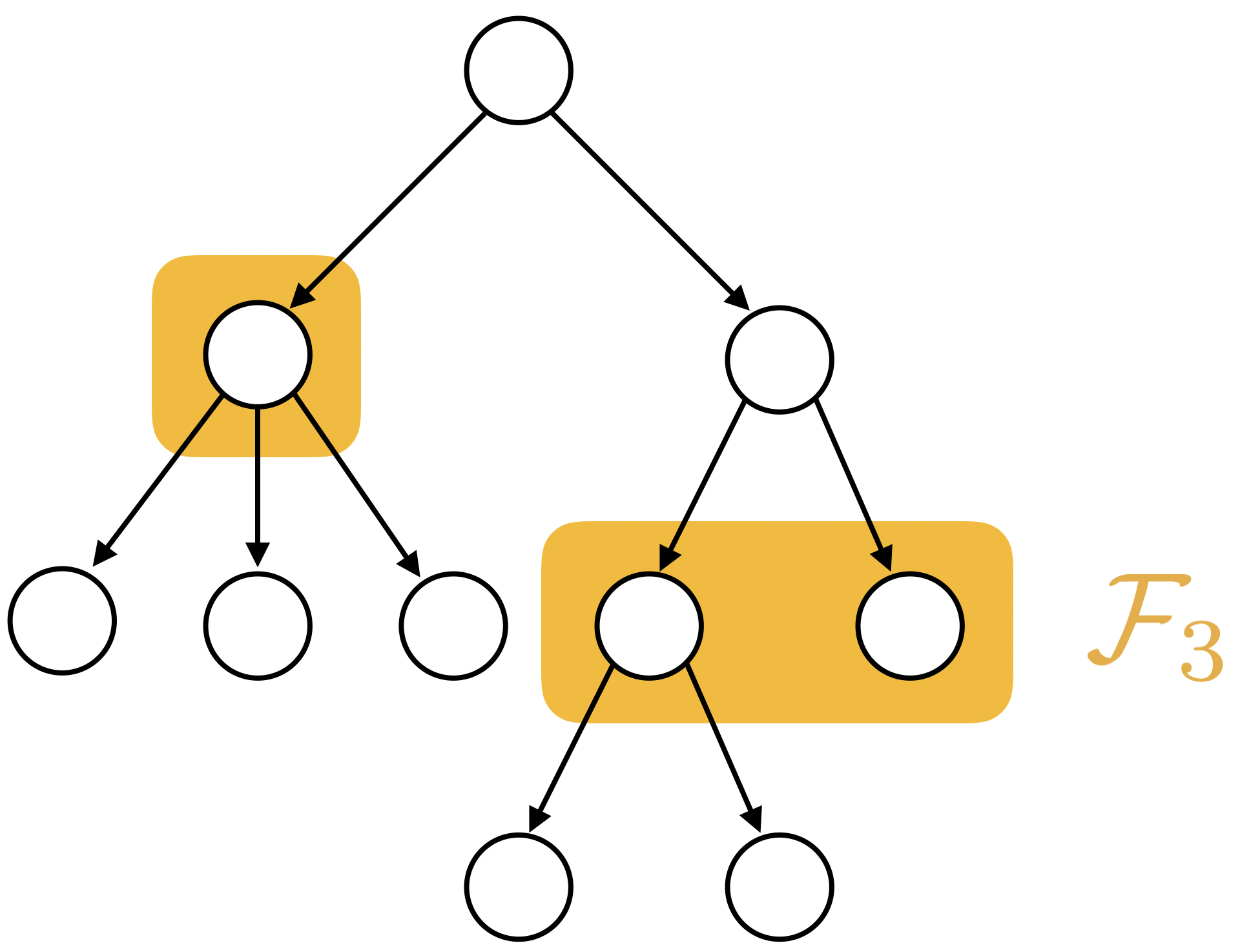}
\centering
\caption{An example of three different frontiers $\frontier_1, \frontier_2,
	\frontier_3$ in a single refinement tree.}
\label{fig:frontier_examples}
\end{figure}

To lend insight into the structure of frontiers, we provide a characterization
of frontiers in terms of incomparable vector spaces. In particular,
a frontier 
$\frontier$ of a refinement tree $\tree$ comprises
a subset of the vertices $\treeNodes$ in which every leaf is
descendant from exactly one vertex in $\frontier$. 
This can be seen explicitly
in fig.\ \ref{fig:frontier_examples}, which provides examples of
frontiers. To show this characterization rigorously, we use the following
proposition.

\begin{restatable}[characterization of frontiers]{proposition}{propcharfrontiers} \label{prop:character_frontiers}
Given a $\rootSpace$-refinement tree $\tree \equiv (\treeNodes, \treeEdges)$,
	$\frontier \subset \treeNodes$ is a frontier iff every leaf $\leafSpace\in\treeLeaves$ is descendant from exactly one space $\parentSpace \in \frontier$.
\end{restatable}

To enable the comparison of two frontiers $\frontier_1$ and $\frontier_2$ such
that one can reason about which is `finer', we introduce a natural partial
ordering on the set of orthogonal decompositions of $\rootSpace$.

\begin{definition}[partial order $\preceq$]
For two orthogonal decompositions $\decomp_1, \decomp_2$ of $\rootSpace$, we write $\decomp_1 \preceq \decomp_2$ if for every vector space $\childSpace \in \decomp_1$, there is a vector space $\parentSpace \in \decomp_2$ such that $\childSpace \subset \parentSpace$. If this is the case, we say that $\decomp_1$ is \textbf{dominated} by $\decomp_2$.
\end{definition}

\begin{figure}[h]
\includegraphics[width=0.3\textwidth]{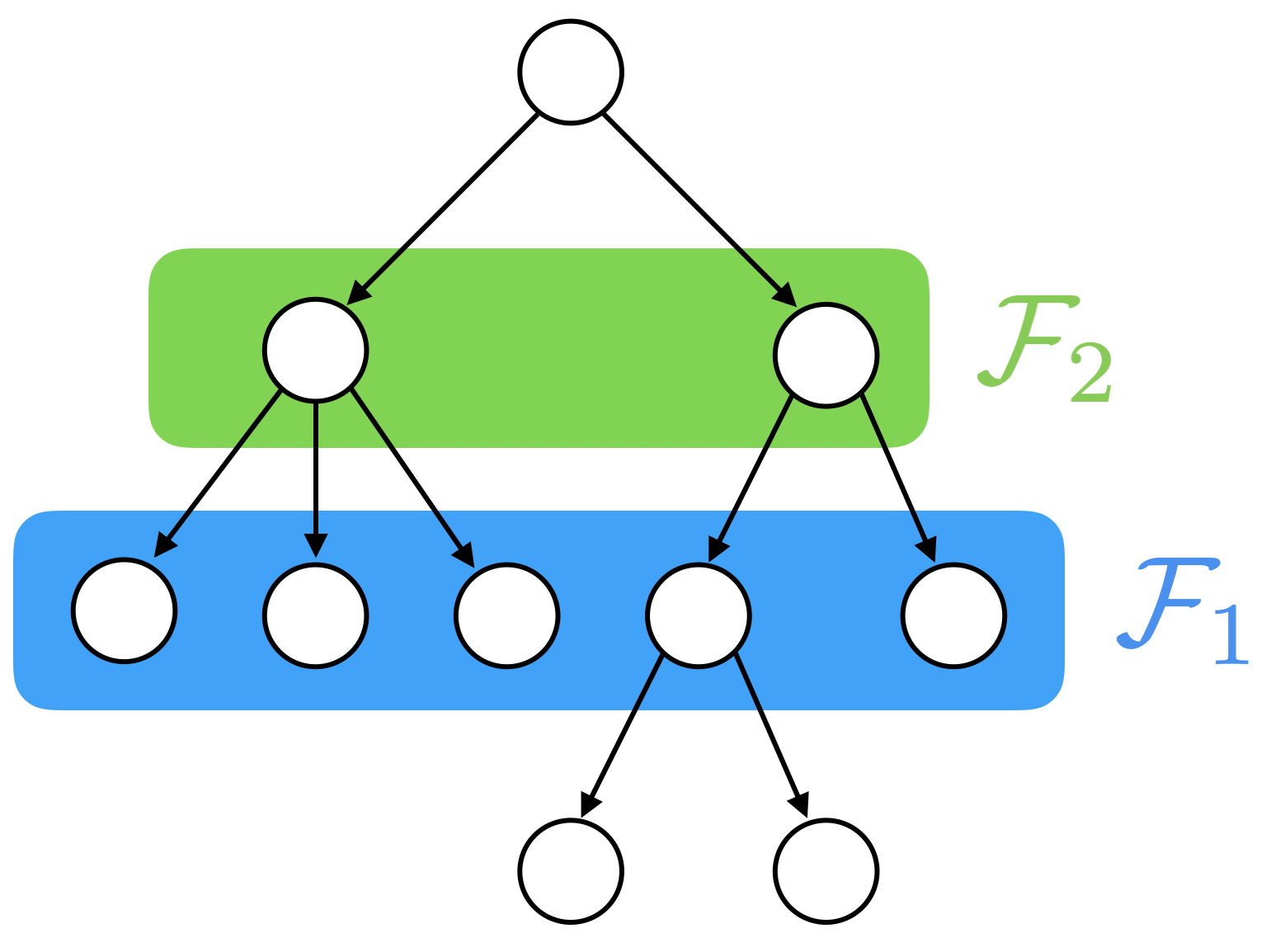}
\includegraphics[width=0.3\textwidth]{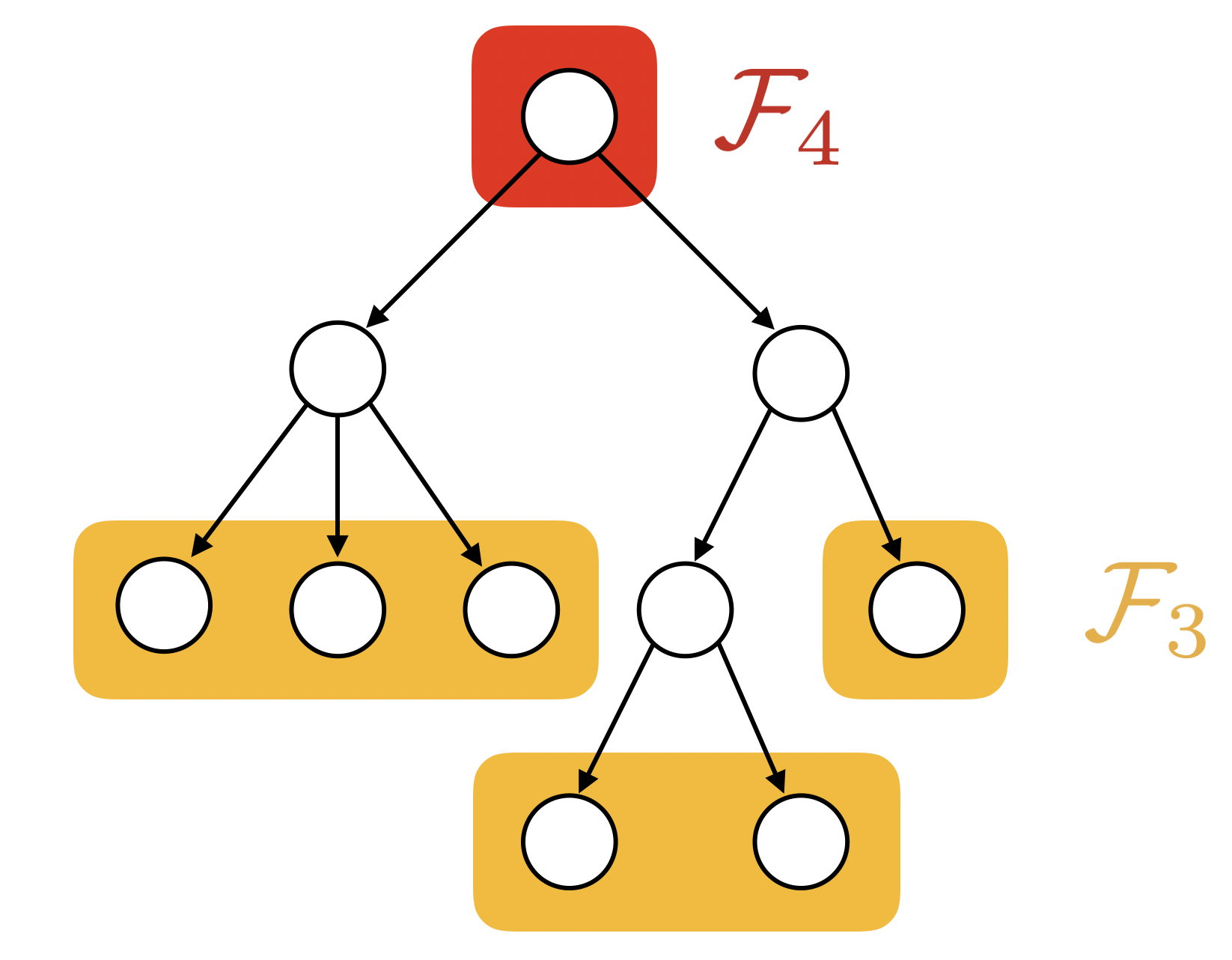}
\centering
\caption{Two examples of the binary relation $\preceq$ of frontiers, as
	$\frontier_1\preceq\frontier_2$ and $\frontier_3\preceq\frontier_4$.
	Indeed, all frontiers above are
	comparable with $\frontier_3 \preceq \frontier_1 \preceq \frontier_2 \preceq
	\frontier_4$.}
\end{figure}

This partial ordering extends to frontiers of a refinement tree $\tree$.
Intuitively, $\decomp_1 \preceq \decomp_2$ means that $\decomp_1$ is a finer
decomposition than $\decomp_2$ and that $\decomp_1$ can be obtained by
refining $\decomp_2$. This is borne out by the following proposition.

\begin{restatable}[ancestor map]{proposition}{propancestor} \label{prop:ancestor}
If $\decomp_1$ and $\decomp_2$ are orthogonal decompositions of $\rootSpace$ and
	$\decomp_1 \preceq \decomp_2$, then there exists a unique ancestor map
$\ancestor_{\decomp_1,\decomp_2} : \decomp_1 \longrightarrow \decomp_2$
	with the property that
	$\childSpace \subset \ancestor_{\decomp_1,\decomp_2}(\childSpace)$ for 
$\childSpace\in\decomp_1$. This map also has the
	property,
\begin{equation} \label{frontier_decomposition}
\parentSpace = \sum_{\childSpace \in \ancestor_{\decomp_1,\decomp_2}^{-1}(\parentSpace)} \childSpace, \qquad \forall \parentSpace \in \decomp_2 \,.
\end{equation}
\end{restatable}
\noindent We supress the subscripts of the ancestor map when the associated
frontiers are obvious from context.
Fig.\ \ref{fig:ancestor} provides a visual example of an ancestor map.

\begin{figure}[h]
\includegraphics[width=0.35\textwidth]{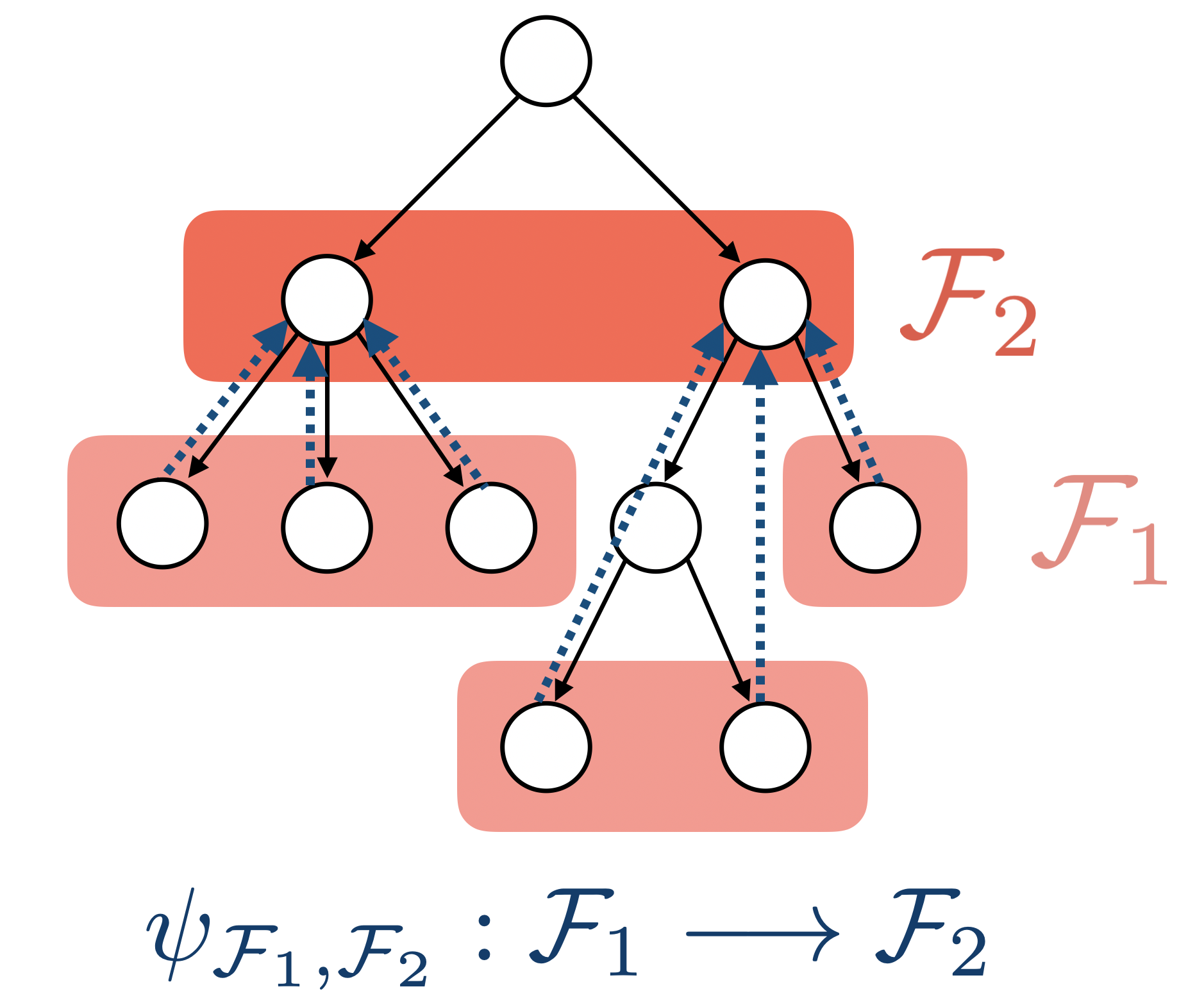}
\centering
	\caption{An example of an ancestor map $\ancestor_{\frontier_1,\frontier_2}
	$ (shown with dotted arrows) where
	$\frontier_1 \preceq \frontier_2$. The ancestor map associates each vertex
	in $\frontier_1$ with its unique ancestor in $\frontier_2$.}
\label{fig:ancestor}
\end{figure}

\begin{corollary} \label{cor:ancestor_vec}
If $\decomp_1 \preceq \decomp_2$ are orthogonal decompositions of $\rootSpace$
	and $\ancestor_{\decomp_1,\decomp_2}$ is the ancestor map between them, then for any vector $\phiVec \in \rootSpace$ and space $\parentSpace \in \decomp_2$,
\begin{equation}
\proj_{\parentSpace}(\phiVec) = \sum_{\childSpace \in
	\ancestor_{\decomp_1,\decomp_2}^{-1}(\parentSpace)}
	\proj_{\childSpace}(\phiVec) \,.
\end{equation}
\end{corollary}

\begin{corollary}\label{cor:ancestor_span}
If $\decomp_1 \preceq \decomp_2$ are orthogonal decompositions of $\rootSpace$, then
\begin{equation}
	\vecspan(\sieve{\phiVec}{\decomp_2}) \subset
	\vecspan(\sieve{\phiVec}{\decomp_1}) \,.
\end{equation}
\end{corollary}

\section{Algorithm schema} \label{sec:schema}

With mathematical preliminaries now established, we provide an overview of the
proposed refinement algorithm. Suppose we are given an initial basis
$\romBasis=\initBasis\equiv[\initBasisCol_1\ \cdots\ 
\initBasisCol_{\romSize_0}]\in\stiefel{\stateVecSize \times \romSize_0}$ as
well as
an
$\mathbb{R}^\stateVecSize$-refinement tree $\tree$. To
perform refinement, 
we maintain a frontier
$\frontier_i$ in the tree $\tree$
for each basis vector $\initBasisCol_i$,
$i=1,\ldots,\romSize_0$. Because the basis begins in its 
initial unrefined state, we initially set all frontiers $\frontier_i$ to the coarsest
possible value, namely the root-node state $\frontier_i \gets \{ \mathbb{R}^\stateVecSize \}$. 
Now, whenever the ROM solution 
$\trialBasisOffset + \romBasis
	\romStateVecIter(\param)$ is deemed to be an inaccurate approximation of the
	FOM solution 
$\stateVecIter(\param)$, the algorithm
performs basis refinement, which consists of first finding new (finer) frontiers
$\frontier_i'$, $i=1,\ldots,\initRomSize$ satisfying
\begin{equation} \label{eq:new_frontier}
\frontier_i' \preceq \frontier_i \,.
\end{equation}
Next, the algorithm sets $\frontier_i \gets \frontier_i'$ and sieves basis
vectors $\initBasisCol_i$ through 
refined frontier $\frontier_i$ for $i=1,\ldots,\romSize_0$ to arrive at a new
enriched basis 
\begin{equation}
\romBasis = \left[\begin{array}{cccc} \matsieve{\initBasisCol_1}{\frontier_1}
&  \ldots &
\matsieve{\initBasisCol_{\romSize_0}}{\frontier_{\romSize_0}}
\end{array}\right] \in \R{\nat{\stateVecSize} \times \frontierglobal} \,,
\end{equation}
where $\frontierglobal \equiv \bigsqcup_i \frontier_i$. If this new enriched
basis
remains insufficient, we can further refine the frontiers; this can proceed
recursively until the desired level of fidelity is achieved.
However, we must address three principal problems in
developing such an algorithm:

\begin{enumerate}
	\item\label{schema:refinement} \textit{Refinement of the frontiers $\frontier_i$ into $\frontier_i'$}.
	Ideally, the frontiers $\frontier_i'$ would balance the accuracy benefit of
		increased fidelity with the cost drawback of increased dimensionality.
		Moreover, the algorithm should determine both (1) the frontiers
		$\frontier_i$ to refine, and (2) the manner in which they should be
		refined. For example, in the case of a propagating shock, refinement of
		the frontiers $\frontier_i$ should be performed in the vicinity of the
		shock. To address this, we propose an approach that extends the
		dual-weighted-residual error indicator technique from the original
		$h$-refinement method \cite{carlberg2015adaptive} to the present 
		framework. These error indicators provide a heuristic guide for assessing which
		frontiers $\frontier_i$ offer the greatest refinement benefit in terms of
		minimizing the quantity-of-interest error.  Section
		\ref{sec:basis_refinement} presents this approach.
\item \textit{Construction of the $\mathbb{R}^\stateVecSize$-refinement tree
	$T$}. The refinement tree $\tree$ determines the hierarchical structure of
		the frontiers $\frontier_i$, and should be designed such that relatively few
		refinement steps are needed to enable the basis to accurately represent
		the FOM solution. There are a number of considerations one may want to
		take into account when constructing this tree. For example, it may be desirable
		to preserve spatial coherence in the refinement hierarchy of $\tree$, so
		that the vector spaces in $\treeNodes$ correspond to contiguous regions of
		the spatial domain; in this case, the refined basis will be sparse, which
		can improve computational efficiency. Alternatively, it may be desirable to
		preserve coherence in the frequency domain, or even a combination of the
		two.  In any case, the optimal tree $\tree$ is clearly highly problem dependent.
		Nonetheless, we provide a data-driven method for constructing this
		refinement
		tree, which is applicable to situations
		where no such problem-specific information is available
		besides collected snapshot data.  This technique comprises an extension of
		the recursive $k$-means clustering approach proposed in the original
		$h$-refinement work \cite{carlberg2015adaptive}.
		Section \ref{sec:tree_construction} presents this approach.
\item \textit{Compression of the refined basis $\romBasis$ when necessary}.
	Whenever the basis $\romBasis$ is refined in the above manner, the
		basis dimension increases. To prevent this dimension from increasing
		monotonically over time, we require an approach to control the
		basis dimension.
		The original 
		$h$-refinement work \cite{carlberg2015adaptive} simply \textit{reset}
		the basis $\romBasis$ to the initial basis $\initBasis$ after a prescribed
		number of time steps. In the present mathematical framework, this
		corresponds to simply resetting the frontiers $\frontier_i \gets \{
			\mathbb{R}^\stateVecSize \}$ periodically. However, this approach is undesirable for
		several reasons.  First, the fact that the initial basis $\initBasis$
		required refinement indicates that it is deficient; 
		resetting the basis simply reintroduces these deficiencies. Second,
		refining the basis provides valuable information about the particular
		deficiency of the original basis;  
		resetting the basis effectively discards this important information. To
		address these drawbacks, we propose to perform an online-efficient POD of 
		solution snapshots computed with the refined ROM (after projecting out
		solution components in the initial basis $\initBasis$), and subsequently
		append the resulting POD modes to the original basis $\initBasis$. Naively
		implemented, this approach incurs an $\stateVecSize$-dependent operation
		count.  However, we have developed an algorithm that employs the structure
		of the refinement tree $\tree$ to perform an efficient POD whose operation
		count depends only on the refined-ROM dimension $\romSize$. 
		Section \ref{sec:online_basis_compression} presents this algorithm.
\end{enumerate}

\section{Basis refinement} \label{sec:basis_refinement}

In this section, we present our approach for refining the frontiers
$\frontier_i$, which corresponds to component \ref{schema:refinement}  of the
algorithm schema in section \ref{sec:schema}. This requires additional
notation.  In particular, we must establish notation for 
canonical refinements of an arbitrary frontier $\frontier_i$.

\subsection{Frontier refinement}\label{sec:frontierRefine}

First, we define the process of decomposing a vector space $\parentSpace$ in a
refinement tree $\tree$.

\begin{definition}[$\tree$-refinement]
For a given $\rootSpace$-refinement tree $\tree \equiv (\treeNodes, \treeEdges)$,
	we denote $\tree$-\textbf{refinement} of a vector space $\parentSpace \in
	\treeNodes$ by $\tRefine{\parentSpace}{\tree}$ and define it as
\begin{equation}
\tRefine{\parentSpace}{\tree} \equiv \begin{cases}
\children_\tree(\parentSpace) & \children_\tree(\parentSpace) \neq \emptyset
\\ \{ \parentSpace \} & \text{otherwise.}\end{cases}
\end{equation}
Recall that $\children_\tree(\parentSpace)$ denotes the children of
	$\parentSpace$ in tree $\tree$.
\end{definition}

The $\tree$-refinement of $\parentSpace$ corresponds to the decomposition of
$\parentSpace$ given by the children of $\parentSpace$ in the tree $\tree$,
unless $\parentSpace$ is a leaf of the tree, in which case the decomposition
of $\parentSpace$ is given by itself. Now, for any given frontier
$\frontier$ in a $\tree$-refinement tree, there is a natural notion of the
`next level' of refinement. We can simply take the refinement of every
subspace in $\frontier$.

\begin{definition}[full refinement]
The \textbf{full refinement} $\refine{\frontier}$ of a frontier $\frontier$ is
	given by the $\tree$-refinement of all spaces in $\frontier$, i.e.,
\begin{equation}
\refine{\frontier} \equiv \bigcup_{\parentSpace \in \frontier} \tRefine{\parentSpace}{\tree} \,.
\end{equation}
\end{definition}

\begin{figure}[h]
\includegraphics[width=0.3\textwidth]{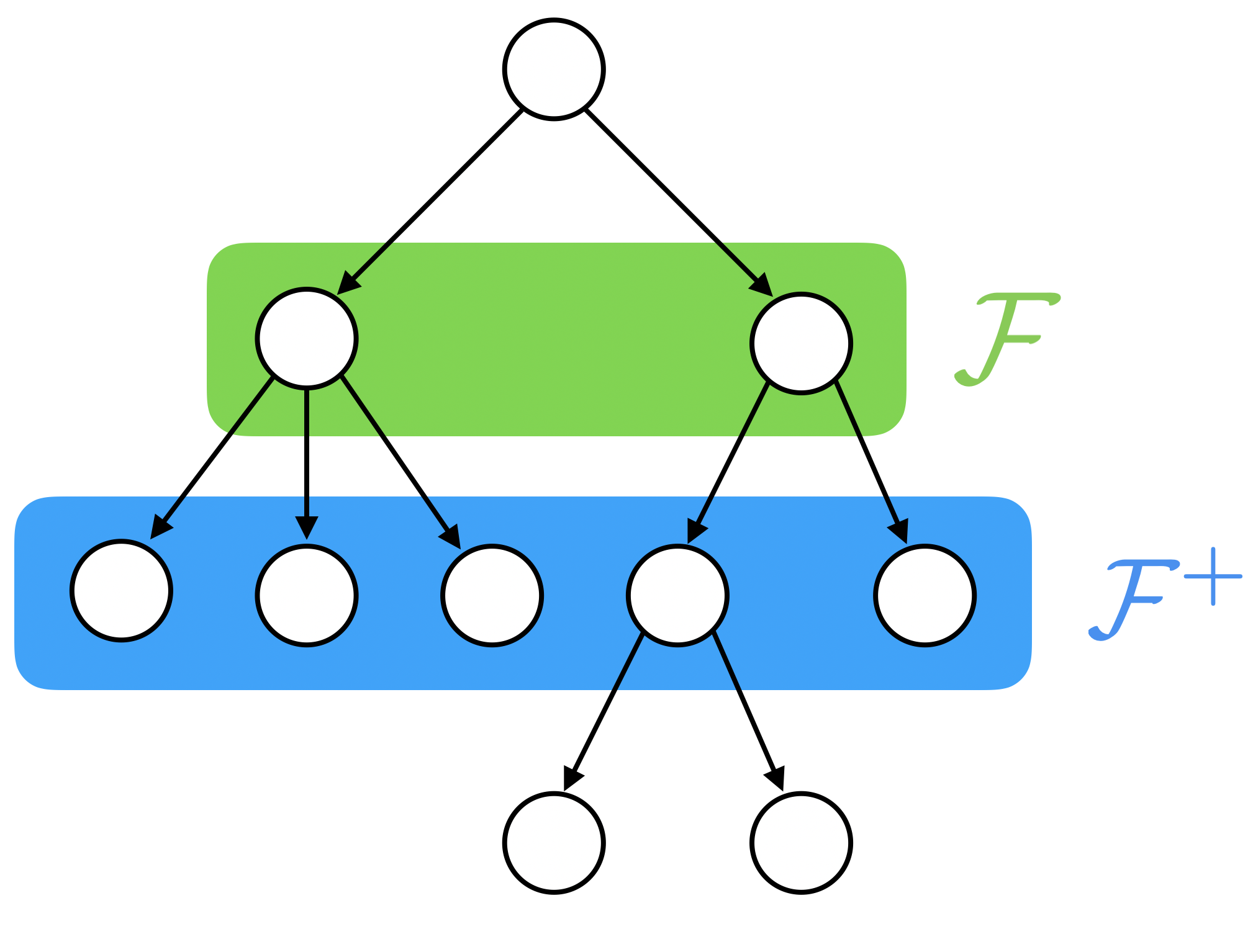}
\centering
\caption{An example of a full refinement $\refine{\frontier}$ of a frontier
	$\frontier$.}
\end{figure}

\begin{remark}
Proposition (\ref{subset_proposition}) implies that $\frontier^+ \preceq \frontier$.
\end{remark}

Thus, there is always a simple way to perform refinement of a frontier
$\frontier$: by taking the full refinement $\refine{\frontier}$.
However, such a strategy is aggressive; in practice, we
aim to consider more tailored refinements of the frontier $\frontier$.
To achieve this, we note that rather than refining every vector space in
$\frontier$, we can refine a subset of these vector spaces.
This leads to the definition of a partial refinement of
a frontier $\frontier$.

\begin{definition}[partial refinement]
The \textbf{partial refinement} of a frontier $\frontier$ of a
	$\rootSpace$-refinement tree $\tree$ at vector spaces $\parentSpace_1,
	\ldots, \parentSpace_m \in \frontier$ is given by
\begin{equation}
\pRefine{\frontier}{\parentSpace_1, \ldots, \parentSpace_m}{\tree} \equiv
	\left(\frontier \setminus \left\{\parentSpace_1, \ldots,
	\parentSpace_m\right\}\right) \cup \left(\bigcup_{i=1}^m \tRefine{\parentSpace_i}{\tree} \right) \,.
\end{equation}
In the absence of a refinement tree, the \textbf{partial refinement} of any
	orthogonal decomposition $\decomp$ of
	$\rootSpace$ using orthogonal decompositions $\decomp_1, \ldots, \decomp_m$
	of vector spaces $\parentSpace_1, \ldots, \parentSpace_m \in \decomp$ is given by
\begin{equation}
\pRefineD{\decomp}{\decomp_1, \ldots, \decomp_m} \equiv \left( \decomp
	\setminus \left\{\parentSpace_1, \ldots, \parentSpace_m \right\} \right)
	\cup \left(\bigcup_{i=1}^m \decomp_i\right) \,.
\end{equation}
\end{definition}

\begin{remark}\label{rem:dom}
Note that refinements are always dominated by the original decompositions from
	which they were refined, i.e.,
\begin{align}
\pRefine{\frontier}{\parentSpace_1, \ldots, \parentSpace_m}{\tree} &\preceq \frontier\,,\\
\pRefineD{\decomp}{\decomp_1, \ldots \decomp_m} &\preceq \decomp \,.
\end{align}
\end{remark}

\begin{remark} \label{re:refinement_still_frontier}
The partial refinement $\pRefine{\frontier}{\parentSpace_1, \ldots,
	\parentSpace_m}{\tree}$ of a frontier $\frontier$ at vector spaces
	$\parentSpace_1, \ldots, \parentSpace_m\in\frontier$ is also a frontier. Moreover, if
	$\frontier_1, \ldots, \frontier_m$ are frontiers of the subtrees of $\tree$
	rooted at $\parentSpace_1, \ldots, \parentSpace_m \in \frontier$, respectively, then the partial refinement $\pRefineD{\frontier}{\frontier_1, \ldots, \frontier_m}$ is also a frontier.
\end{remark}

\subsection{Dual-weighted-residual error indicators} \label{sec:dwrei}

The workhorse of the refinement portion of our algorithm is the goal-oriented
dual-weighted-residual error-indicator approach from the original
$h$-refinement method \cite{carlberg2015adaptive}. This approach ascribes an
error indicator to every element of a frontier $\frontier$, thereby enabling
the method to refine only the elements of the frontier $\frontier$ associated
with the largest approximated errors.

We begin by assuming the context of section \ref{sec:schema}, i.e., we are
given an initial basis
$\romBasis=\initBasis\equiv[\initBasisCol_1\ \cdots\ 
\initBasisCol_{\romSize_0}]\in\stiefel{\stateVecSize \times \romSize_0}$
and an $\mathbb{R}^\stateVecSize$-refinement tree $\tree$. The current
`coarse' basis is given 
by the sieve of the basis vectors $\initBasisCol_1,
\ldots, \initBasisCol_{\romSize_0}$ through frontiers $\frontier_1, \ldots,
\frontier_{\romSize_0}$, i.e.,
\begin{equation} \label{eq:coarse}
\romBasisC = \left[\begin{array}{ccc} \matsieve{\initBasisCol_1}{\frontier_1}
 & \cdots &
\matsieve{\initBasisCol_{\romSize_0}}{\frontier_{\romSize_0}}
\end{array}\right] \in \R{\nat{\stateVecSize} \times \frontierglobalC} \,,
\end{equation}
where $\frontierglobalC \equiv \bigsqcup_i \frontier_i$. If the coarse
basis $\romBasisC$ is deficient, we would like enrich the
basis $\romBasisC$ by refining the frontiers $\frontier_1, \ldots,
\frontier_{\romSize_0}$. A naive way to perform this refinement would
be simply to apply full refinement to each frontier
$\frontier_1, \ldots, \frontier_{\romSize_0}$, i.e.,
\begin{equation}
\romBasisF \equiv \left[\begin{array}{ccc}
\matsieve{\initBasisCol_1}{\refine{\frontier_1}} &
\cdots &
\matsieve{\initBasisCol_{\romSize_0}}{\refine{\frontier_{\romSize_0}}}
\end{array}\right] \in \R{\nat{\stateVecSize} \times \frontierglobalF} \,,
\end{equation}
where $\frontierglobalF \equiv \bigsqcup_i \refine{\frontier_i}$ and $\romBasisF$
denotes the full refinement of the previous basis $\romBasisC$. This
aggressive approach is tantamount to performing \textit{uniform} refinement. Instead,
we aim to devise an \textit{adaptive} approach that performs refinement only on
basis vectors contributing most to the quantity-of-interest error.

For each frontier $\frontier_i$ in Eq.\ (\ref{eq:coarse}) and its full
refinement $\refine{\frontier_i}$, there exists an ancestor map 
\begin{equation}
\ancestor_i\equiv\ancestor_{i,\refine{\frontier_i},\frontier_i} : \refine{\frontier_i} \longrightarrow \frontier_i \,.
\end{equation}
These ancestor maps together induce a global ancestor map from $\frontierglobalF
\equiv \bigsqcup_i \refine{\frontier_i}$ to $\frontierglobalC \equiv \bigsqcup_i
\frontier_i$, i.e.,
\begin{equation} \label{eq:global_ancestor_map}
\ancestor\equiv
	\ancestor_{\frontierglobalF,\frontierglobalC}
	: \frontierglobalF \longrightarrow \frontierglobalC 
\end{equation}
such that $\left.\ancestor\right|_{\refine{\frontier_i}} = \ancestor_i$. The
indicator matrix 
$\prolong \in \mathbb{R}^{\frontierglobalF \times \frontierglobalC}$
for this global ancestor map has entries
\begin{equation} \label{eq:prolong_def}
\left(\prolong\right)_{\childSpace, \parentSpace} \equiv \begin{cases} 
      1 & \ancestor(\childSpace) = \parentSpace \\
      0 & \text{otherwise.} 
   \end{cases}
\end{equation}
In analogue to
the prolongation
operator from $h$-refinement for finite elements, 
we refer to the matrix $\prolong$
as the \textbf{prolongation operator} from the coarse basis $\romBasisC$ to the
fine basis $\romBasisF$. Indeed, 
corollary (\ref{cor:ancestor_vec}) implies
\begin{equation} \label{eq:prolong}
\romBasisC = \romBasisF \prolong \,.
\end{equation}
The prolongation operator relates coordinate representations in the coarse
basis $\romBasisC$ to coordinate representations in the fine basis
$\romBasisF$. Indeed, if we have a coordinate representation $\romStateVecC$
of data $\romBasisC \romStateVecC$ in the coarse basis $\romBasisC$, then
\begin{equation}
\romBasisC \romStateVecC = (\romBasisF \prolong) \romStateVecC = \romBasisF (\prolong \romStateVecC) \,.
\end{equation}
Hence, $\prolong \romStateVecC$ provides the coordinate representation of this
data in the fine basis $\romBasisF$.

Now, consider the context of refinement. Suppose we have computed a
(coarse) ROM solution $\romBasisC \romStateVecC$ satisfying
\begin{equation} \label{eq:coarse_def}
\left(\romBasisC\right)^T \gResidual\left(\romBasisC \romStateVecC\right) =
	\zero \,.
\end{equation}
If we perform uniform refinement of all frontiers
and solve the ROM corresponding to the resulting fine basis
$\romBasisF$, we obtain a higher fidelity
ROM solution $\romBasisF \romStateVecF$ satisfying
\begin{equation} \label{eq:fine_def}
\left(\romBasisF\right)^T \gResidual \left(\romBasisF \romStateVecF\right) =
	\zero \,.
\end{equation}
However, we would like to avoid computations that scale with the dimension of
the fully refined basis.

To achieve this---yet still glean information about the unknown refined
solution $\romBasisF \romStateVecF$---we apply dual-weighted-residual error
estimation. We begin by assuming the residual
$\gResidual$ is twice continuously differentiable and approximate it using 
a first-order Taylor-series expansion about the
known coarse solution $\romBasisC \romStateVecC$, i.e.,
\begin{equation}
\gResidual\left(\romBasisF \romStateVecF\right) = \gResidual(\romBasisC
	\romStateVecC) + \frac{\partial \gResidual}{\partial \stateVec}
	\left(\romBasisC \romStateVecC\right) \romBasisF \left(\romStateVecF -
	\prolong \romStateVecC\right) +O(\|
	\romStateVecF -
	\prolong \romStateVecC
	\|^2)\,,
\end{equation}
where we have used Eq.\ (\ref{eq:prolong}) to relate the coarse and refined
bases. Left multiplying the above by $\romBasisF$ and using
Eq.\ (\ref{eq:fine_def}) yields
\begin{equation}
\zero = (\romBasisF)^T \gResidual\left(\romBasisF \romStateVecF\right) = (\romBasisF)^T \gResidual(\romBasisC \romStateVecC) + (\romBasisF)^T \frac{\partial \gResidual}{\partial \stateVec} \left(\romBasisC \romStateVecC\right) \romBasisF \left(\romStateVecF - \prolong \romStateVecC\right) 
	+O(\|
	\romStateVecF -
	\prolong \romStateVecC
	\|^2)
	\,.
\end{equation}
Solving for the error $\romStateVecF - \prolong \romStateVecC$ gives the Newton approximation
\begin{equation} \label{eq:newton}
\left(\romStateVecF - \prolong \romStateVecC\right) = - \left[(\romBasisF)^T\frac{\partial \gResidual}{\partial \stateVec} \left(\romBasisC \romStateVecC\right) \romBasisF \right]^{-1} \left(\romBasisF\right)^T\gResidual(\romBasisC \romStateVecC)
+O(\|
	\romStateVecF -
	\prolong \romStateVecC
	\|^2)\,.
\end{equation}
Unfortunately, computing this Newton approximation requires a 
$\frontierglobalF \times \frontierglobalF$ linear-system solve; this is precisely what we
aim to avoid. Thus, we instead consider the dual. As above,
we assume the quantity-of-interest functional 
$\interestFunc$ 
(see Eq.\
(\ref{interest_fom}))
is twice continuously differentiable and
perform a Taylor
expansion about the coarse solution $\romStateVecC$, i.e.,
\begin{equation} \label{eq:interest_taylor}
\interestFunc(\romBasisF \romStateVecF) = \interestFunc(\romBasisC \romStateVecC) +\frac{\partial \interestFunc}{\partial \stateVec} \left(\romBasisC \romStateVecC\right) \romBasisF \left(\romStateVecF - \prolong \romStateVecC\right) 
	+O(\|
	\romStateVecF -
	\prolong \romStateVecC
	\|^2)
	\,.
\end{equation}
Substituting Eq.\ (\ref{eq:newton}) into Eq.\ (\ref{eq:interest_taylor}) yields 
\begin{equation} \label{eq:interest_func_adjoint}
\interestFunc(\romBasisF \romStateVecF) - \interestFunc(\romBasisC
	\romStateVecC) = - \left(\adjointF\right)^T (\romBasisF)^T \gResidual(\romBasisC \romStateVecC) 
	+O(\|
	\romStateVecF -
	\prolong \romStateVecC
	\|^2)
	\,,
\end{equation}
where $\adjointF \in \mathbb{R}^{\frontierglobalF}$ is the fine adjoint satisfying
\begin{equation} \label{eq:fine_adjoint}
\left[(\romBasisF)^T \frac{\partial \gResidual}{\partial \stateVec}
	(\romBasisC \romStateVecC)^T \romBasisF\right]^T \adjointF = (\romBasisF)^T \frac{\partial \interestFunc}{\partial \stateVec} \left(\romBasisC \romStateVecC\right)^T \,.
\end{equation}
It may not seem that we have made any progress, as computing the
adjoint $\adjointF$ in satisfying (\ref{eq:fine_adjoint}) still requires a
$\frontierglobalF \times \frontierglobalF$ linear-system solve. However, the advantage of
adopting this viewpoint is that there is a natural way to approximate the
adjoint $\adjointF$ in an efficient manner, namely as the prolongation of 
the coarse adjoint $\adjointC\in \mathbb{R}^{\frontierglobalC}$, i.e.,
\begin{equation} \label{eq:prolong_adjoint}
\prolongadjoint \equiv \prolong \adjointC \,.
\end{equation}
where the coarse adjoint $\adjointC$ satisfies
\begin{equation} \label{eq:coarse_adjoint}
\left[(\romBasisC)^T \frac{\partial \gResidual}{\partial \stateVec}
	(\romBasisC \romStateVecC)^T \romBasisC \right]^T \adjointC = (\romBasisC)^T
	\frac{\partial \interestFunc}{\partial \stateVec} \left(\romBasisC
	\romStateVecC\right)^T \,.
\end{equation}
Critically, computing the coarse adjoint $\adjointC$ requires only a
$\frontierglobalC \times \frontierglobalC$ linear-system solve. Replacing the 
fine adjoint $\adjointF$ with its approximation 
$\prolongadjoint$ in 
Eq.\
(\ref{eq:interest_func_adjoint}) yields
\begin{equation}\label{eq:taylorApprox}
\interestFunc\left(\romBasisF \romStateVecF \right) - \interestFunc\left(\romBasisC \romStateVecC \right) \approx - \left(\prolongadjoint\right)^T (\romBasisF)^T \gResidual\left(\romBasisC \romStateVecC\right) \,.
\end{equation}
Finally, we can bound the right hand side of Eq.\ \eqref{eq:taylorApprox} by
\begin{equation}
\left| \left(\prolongadjoint\right)^T (\romBasisF)^T \gResidual
	\left(\romBasisC \romStateVecC \right) \right| \leq \sum_{\childSpace \in
	\frontierglobalF} \errorIndF_\childSpace \,.
\end{equation}
Here, the error indicators $\errorIndF_\childSpace \in \mathbb{R}_{\geq 0}$
for $\childSpace \in \frontierglobalF$ are the absolute values of the summands in
the inner product $\left(\prolongadjoint\right)^T \left[(\romBasisF)^T
\gResidual\left(\romBasisC \romStateVecC\right)\right]$, i.e.,
\begin{equation} \label{eq:error_ind}
\errorIndF_\childSpace \equiv \left|\left[\prolongadjoint\right]_\childSpace \left(\romBasisCol_\childSpace^h\right)^T \gResidual\left(\romBasisC \romStateVecC\right)\right| \,,
\end{equation}
where $\romBasisCol_\childSpace^h$ denotes the $\childSpace$-column of
$\romBasisF$. These error indicators ascribe an approximate error
heuristic to every element in the full refinement $\frontierglobalF$. Moreover,
these error indicators can be pulled back to the global coarse frontier
$\frontierglobalC$ via the global ancestor map $\ancestor$, i.e.,
\begin{equation} \label{eq:splitCoarseErr}
\errorIndC_{\parentSpace} \equiv \sum_{\childSpace \in
	\psi^{-1}(\parentSpace)} \errorIndF_{\childSpace}, \qquad \parentSpace \in
	\frontierglobalC \,.
\end{equation}
That is, the error indicator for $\parentSpace \in \frontierglobalC$ comprises the
sum of error indicators of its children. In matrix form, this corresponds to
\begin{equation} \label{eq:coarse_err_prolong} \errorIndC = \errorIndF
	\prolong \,.  \end{equation} The key to our refinement algorithm is to
	refine only the spaces $\parentSpace \in \frontierglobalC$ for which the
	corresponding error indicator $\errorIndC_{\parentSpace}$ is large. In our
	implementation, we refine those spaces $\parentSpace \in \frontierglobalC$ such
	that the corresponding error indicator is greater than the average of all
	error indicators. Algorithm \ref{alg:err} provides the full procedure for
	computing these error indicators.

\begin{algorithm}
\caption{Computation of Error Indicators} \label{alg:err}
\hspace*{\algorithmicindent} \textbf{Input}: The current coarse basis $\romBasisC$, the current frontiers $\frontier_1, \ldots, \frontier_{\romSize_0}$ \\
\hspace*{\algorithmicindent} \textbf{Output}: The fine error indicators $\errorIndF$.
\begin{algorithmic}[1]
\Procedure{ComputeErrorIndicators}{$\romBasisC$, $\frontier_1, \ldots, \frontier_{\initRomSize}$}
\State $\frontierglobalC \gets \bigsqcup_i \frontier_i$
\State $\frontierglobalF \gets \bigsqcup_i \refine{\frontier}_i$
\State $\prolong \gets \textsc{ComputeProlongationOperator}\left(\frontierglobalC, \frontierglobalF\right)$ \Comment{Compute prolongation operator using Eq.\ (\ref{eq:prolong_def})}
	\State\label{step:coarseAdjoint} $\adjointC \gets \left[(\romBasisC)^T \frac{\partial \gResidual}{\partial \stateVec} (\romBasisC \romStateVecC)^T \romBasisC \right]^{-T} \left[ (\romBasisC)^T \frac{\partial \interestFunc}{\partial \stateVec} \left(\romBasisC \romStateVecC\right)^T \right]$ \Comment{Compute the coarse adjoint using Eq.\ (\ref{eq:coarse_adjoint})}
\State $\prolongadjoint \gets \prolong \adjointC$ \Comment{Prolongate coarse adjoint to fine coordinate space.}
\State $\errorIndF \gets \ve{0} \in \R{\frontierglobalF}$
\For{$\childSpace \in \frontierglobalF\setminus\treeLeaves$} 
	\State \label{step:errorIndicators}$\errorIndF_\childSpace \gets  \left|\left[\prolongadjoint\right]_\childSpace \left(\romBasisCol_\childSpace^h\right)^T \gResidual\left(\romBasisC \romStateVecC\right)\right| $\Comment{Compute the error indicators using Eq.\ (\ref{eq:error_ind})}
\EndFor
\State \Return $\errorIndF$
\EndProcedure
\end{algorithmic}
\end{algorithm}

\subsection{Refinement algorithm}\label{sec:refineAlg}

With the preliminaries of frontier refinement and dual-weighted-residual error
indicators now established, we return to the objective of this work: adaptive
basis refinement.

Algorithms
\ref{alg:refine_frontiers} and \ref{alg:refine_proc}
report
the proposed refinement algorithm, which takes the following
approach: at a given time instance, the method first solves the ROM equations
to within a prescribed tolerance $\romTol$. Next, an error indicator is
applied to the ROM solution to assess its accuracy; here, we take the error
indicator to be the norm of the FOM residual evaluated at the ROM solution. If
this error indicator is larger than a prescribed tolerance $\tol$, then the
algorithm performs basis refinement. This is repeated until either the ROM
solution satisfies the error-indicator tolerance, or the ROM has
converged to the FOM.

\begin{algorithm}
\caption{Computation of Refined Frontiers} \label{alg:refine_frontiers}
\hspace*{\algorithmicindent} \textbf{Input}: The $\Rn$-refinement tree $\tree$, the fine error indicators $\errorIndF$, the current frontiers $\frontier_1, \ldots, \frontier_{\romSize_0}$.  \\
\hspace*{\algorithmicindent} \textbf{Output}: A new set of frontiers $\frontier_1', \ldots, \frontier_{\initRomSize}'$ refined according to the input error indicators.
\begin{algorithmic}[1]
\Procedure{RefineFrontiers}{$\tree, \errorIndF, \frontier_1, \ldots, \frontier_{\initRomSize}$}
\State $\frontierglobalC \gets \bigsqcup_i \frontier_i$
\State $\frontierglobalF \gets \bigsqcup_i \refine{\frontier}_i$
\State $\ancestor \gets \textsc{GetGlobalAncestorMap}\left(\frontierglobalF, \frontierglobalC\right)$ \Comment{Compute the map in Eq.\ (\ref{eq:global_ancestor_map}) sending every space to its ancestor.}
\State $\prolong \gets \textsc{ComputeProlongationOperator}\left(\frontierglobalC, \frontierglobalF\right)$ \Comment{Compute prolongation operator using Eq.\ (\ref{eq:prolong_def}).}
\State $\errorIndC = \errorIndF \prolong$ \Comment{Compute coarse error indicators using Eq.\ (\ref{eq:coarse_err_prolong}).}
\State $\eta \gets \frac{1}{|\frontierglobalC|} \sum_{\parentSpace \in \frontierglobalC} \errorIndC_\parentSpace$ \Comment{Compute the average of the coarse error indicators.}
\State $S \gets \{ \parentSpace \in \frontierglobalC \mid \errorIndC_\parentSpace \geq \eta \}$ \Comment{Select the spaces in $\frontierglobalC$ whose coarse error indicator is greater than average.}
	\For{$i \in \nat{\initRomSize}$} \Comment{For each frontier $\frontier_i$}
	\State $S_i \gets \frontier_i \cap S$ \Comment{Extract the elements of $S$
	that came from $\frontier_i$.}
	\State $\frontier_i' \gets \pRefine{\frontier}{S_i}{\tree}$ \Comment{Refine the frontier $\frontier_i$ at these spaces.}
\EndFor
\State \Return $(\frontier_1', \ldots, \frontier_{\initRomSize}')$ \Comment{Return the refined frontiers.}
\EndProcedure
\end{algorithmic}
\end{algorithm}

\begin{algorithm}
\caption{Refinement Algorithm} \label{alg:refine_proc}
\hspace*{\algorithmicindent} \textbf{Input}: $\Rn$-refinement tree
	$\tree$, initial basis $\initBasis$, current frontiers $\frontier_1, \ldots, \frontier_{\initRomSize}$,
	reference solution $\trialBasisOffset$, residual function $\gResidual$,
	ROM-residual tolerance $\romTol$, and FOM-residual 
	tolerance $\tol$. \\
\hspace*{\algorithmicindent} \textbf{Output}: A new set of frontiers $\frontier_1', \ldots, \frontier_{\initRomSize}'$ refined according to the input error indicators.
\begin{algorithmic}[1]
\Procedure{SolveModel}{$\tree, \initBasis, \frontier_1, \ldots, \frontier_{\initRomSize}, \trialBasisOffset, \romTol, \tol$}
\While{\textsc{True}} \Comment{Refine the basis until the specified full-order tolerance is met.}
	\State\label{step:sieve} $\romBasisC \gets \left[\begin{array}{ccc}
\matsieve{\initBasisCol_1}{\frontier_1} & \cdots & \matsieve{\initBasisCol_{\romSize_0}}{\frontier_{\initRomSize}}  \end{array}\right]$ \Comment{Retrieve the current coarse model basis.}
	\State $\romStateVec\gets \textsc{SolveROM}(\gResidual, \romBasisC, \trialBasisOffset, \romTol)$ \Comment{Solve the system $(\romBasisC)^T \gResidual \left(\trialBasisOffset + \romBasisC \romStateVec\right) = 0$ from Eq.\ (\ref{galerkin_rom}).}
	\State $\stateVec \gets \romBasisC \romStateVec$ \Comment{Lift the result to the full-order model.}
	\If{$\|\gResidual(\stateVec)\|_2 < \tol$} \Comment{Check if the full-order residual is within the specified tolerance.}
		\State \textbf{break} \Comment{If the specified tolerance is satisfied, stop refinement.}
	\EndIf
	\State\label{step:computeErrorIndicators} $\errorIndF \gets \textsc{ComputeErrorIndicators}(\romBasisC, \frontier_1, \ldots, \frontier_{\initRomSize})$ \Comment{Compute the error indicators in Eq.\ (\ref{eq:error_ind}).}
	\State\label{step:refineFrontiers} $(\frontier_1, \ldots, \frontier_{\initRomSize}) \gets \textsc{RefineFrontiers}(\tree, \errorIndF, \frontier_1, \ldots, \frontier_{\initRomSize})$ \Comment{Use error indicators to selectively refine frontiers.}
\EndWhile
\State \Return $(\frontier_1, \ldots, \frontier_{\initRomSize}, \romStateVec)$.
\EndProcedure
\end{algorithmic}
\end{algorithm}

\subsection{Resolving ill-conditioning and ensuring linear
independence}\label{sec:illCond}

To complete the presentation of the refinement algorithm, we must address two
outstanding problems:
\begin{enumerate}
\item The refinement algorithm does not formally ensure that the matrix
	$\romBasis$ is indeed a basis, i.e., that $\romBasis$ has full column rank
		and thus belongs to the non-compact Stiefel manifold. 
\item The refinement algorithm does not ensure the matrix $\romBasis$ is well
	conditioned, even if it has full column rank. This occurs because every
		vector-space sieve reduces the $\ell^2$-norm of some columns of
		$\romBasis$, as 
\begin{equation}
\|\phiVec\|_2^2 = \sum_{\phiVec_i \in \sieve{\phiVec}{\decomp}}
	\|\phiVec_i\|_2^2 \,.
\end{equation}
		Therefore, recursive unbalanced basis refinement will cause the
		$\ell^2$-norms of some columns of $\romBasis$ to shrink, which could lead
		to poor conditioning.
\end{enumerate}

To counteract the first issue, we follow the approach of the original ROM
$h$-refinement method
\cite{carlberg2015adaptive} and \emph{deactivate} redundant vectors of
$\romBasis\equiv[\romBasisCol_1\ \cdots\ \romBasisCol_\romSize]$ by using a column-pivoted QR factorization. We address the second issue 
by scaling the remaining basis vectors to ensure the basis is
well-conditioned. This amounts to computing a diagonal scaling matrix
$\mat{\Sigma}_* \in \R{\romSize_* \times \romSize_*}$ and a selection matrix
$\mat{P} \in \{0, 1\}^{\romSize \times \romSize_*}$, and defining
\begin{equation}\label{eq:romBasisStiefelDef}
\romBasisStiefel \equiv \romBasis \mat{P} \mat{\Sigma}_*
\end{equation}
such the basis $\romBasisStiefel \in \stiefel{\stateVecSize\times\romSize_*}$ contains a subset
of the (scaled) columns
of $\romBasis$.  To
compute $\mat{P}$ and $\mat{\Sigma}_*$, we first define a general diagonal
rescaling matrix $\mat{\Sigma} \in \R{\romSize \times \romSize}$ with diagonal
entries
$
\mat{\Sigma}_{ii} \equiv {1}/{\|\romBasisCol_i\|_2}$.
This addresses the second issue above. However, to
address the first issue, we must cull the redundant columns of $\romBasis
\mat{\Sigma}$ to ensure linear independence to within some tolerance. We
accomplish this via a column-pivoted QR decomposition
\begin{equation}
\mat{Q} \mat{R} \mat{\Pi} = \romBasis \mat{\Sigma} \,.
\end{equation}
Denoting by $\epsilon_\text{QR}$ the desired tolerance for linear
independence, we select the columns of $\romBasis$ whose diagonal $\mat{R}$-factors
are greater than $\epsilon_\text{QR}$; we denote the associated cutoff by $\compSize
\defeq
\min_i \{ \mat{R}_{ii} < \epsilon_\text{QR} \}$. The selection
operator $\mat{P}$ then
corresponds to the first $s$ columns of the pivoting matrix, i.e.,
\begin{equation}
	\mat{P} \equiv \mat{\Pi}_{:, \nat{s}} \,.
\end{equation}
Likewise, to preserve the rescaling factors for the preserved columns, we set
\begin{equation}
\mat{\Sigma}_* = \mat{P}^T \mat{\Sigma} \mat{P} \,.
\end{equation}
These choices for $\mat{P}$ and $\mat{\Sigma}_*$ yield a basis
$\romBasisStiefel$ defined by Eq.~\eqref{eq:romBasisStiefelDef} that is both
linearly independent and well conditioned according to the threshold
$\epsilon_\text{QR}$. 

In the context of algorithm \ref{alg:refine_proc}, we perform this excision
procedure after performing frontier refinement. We then mark the frontier
nodes corresponding to the excised columns of $\romBasis$ as inactive, after
which point the excised basis vectors are 
effectively ignored and can no longer be refined. These modifications can be
incorporated in the refinement algorithm \ref{alg:refine_proc} with only
minimal changes.

\subsection{Proof of monotone convergence}

To conclude this section, we demonstrate that the proposed basis-refinement
algorithm ensures the ROM converges to the FOM, and that the refined bases
produce a monotone sequence of embedded subspaces.

\begin{theorem}[Convergence to the full-order model] \label{thm:convergence} If for every leaf
	$\leafSpace\in \treeLeaves$ of the refinement tree $\tree$, there exists an initial ROM
	basis vector $\initBasisCol_{i(\leafSpace)}$ such that
$\proj_{\leafSpace}(\initBasisCol_{i(\leafSpace)}) \neq 0$, then one of the following must occur:
\begin{enumerate}
\item The refinement algorithm computes a solution satisfying the FOM
	equations to within tolerance $\tol$. \label{item:res1}
\item The range of the refined basis converges to $\Rn$. \label{item:res2}
\end{enumerate}
\end{theorem}

\begin{proof}
Consider the event where $\frontier_i = \treeLeaves$,
	$i=1,\ldots,\initBasisSize$. In this event, for each
	$\leafSpace \in \treeLeaves$, the projected vector
	$\proj_{\leafSpace}(\phiVec_{i(\leafSpace)})\in\leafSpace$ is nonzero by assumption,
	 and $\leafSpace$ has dimension $1$, so
	$\proj_{\leafSpace}(\phiVec_{i(\leafSpace)})$ spans $\leafSpace$ and
\begin{equation}
	\leafSpace \subset \vecspan(
		\proj_{\leafSpace}(\phiVec_{i(\leafSpace)})) \subset \vecspan
	\left( \sieve{\phiVec_{i(\leafSpace)}}{\frontier_{i(\leafSpace)}}
	\right) \subset \vecspan \left( \bigcup_{i = 1}^{\initBasisSize}
	\sieve{\phiVec_{i}}{\frontier_{i}}\right) = \range(\romBasis)\,.
\end{equation}
Summing over $\leafSpace \in \treeLeaves$ then gives us
\begin{equation}
\Rn = \sum_{\leafSpace \in \treeLeaves} \leafSpace \subset \range(\romBasis) \,.
\end{equation}
Thus, if the event 
	$\frontier_i = \treeLeaves$,
	$i=1,\ldots,\initBasisSize$
	occurs, then the event (\ref{item:res2}) in the theorem statement has
	occurred.

To conclude, we note that the event (\ref{item:res1}) in the theorem statement
	is precisely the termination condition of algorithm \ref{alg:refine_proc}.
	Therefore, we claim that either the algorithm terminates or we
	have 
	$\frontier_i = \treeLeaves$,
	$i=1,\ldots,\initBasisSize$ at some iteration. If at a given iteration the event $\frontier_i = \treeLeaves$,
	$i=1,\ldots,\initBasisSize$ has not yet occurred and the
	algorithm has not yet terminated, then we always assign leaves with an error indicator of zero, and refine all nodes with error indicators larger than average, a space in one of the frontiers
	$\frontier_1, ..., \frontier_{\initBasisSize}$ must be selected for refinement. Note that a frontier $\frontier_i$ can be
	refined if and only
	if $|\frontier_i| < \stateVecSize$, since otherwise every space in
	$\frontier_i$ has dimension $1$. Moreover, if $|\frontier_i| =
	\stateVecSize$, then, since $\treeLeaves \preceq \frontier_i$ by proposition
	\ref{prop:character_frontiers}, we must have $\frontier_i = \treeLeaves$.
	Therefore, since a frontier cannot have size larger than $\stateVecSize$,
	and a refinement always strictly increases the size of at least one of the
	frontiers $\frontier_i$, eventually the event $\frontier_i = \treeLeaves$,
	$i=1,\ldots,\initBasisSize$ must occur, which concludes the proof.
\end{proof}

\begin{theorem}[Monotonicity] \label{thm:monotonicity}
The ranges of progressively refined bases produced by algorithm
	\ref{alg:refine_proc} form a monotone sequence of embedded subspaces.
\end{theorem}

\begin{proof}
Let $\frontier_1^{(i)}, ..., \frontier_{\initBasisSize}^{(i)}$ denote the
	frontiers in algorithm \ref{alg:refine_proc}, and let $\romBasis^{(i)}$
	denote the corresponding reduced basis at iteration $i$. Because $\frontier_1^{(i)}, ..., \frontier_{\initBasisSize}^{(i)}$ are produced from $\frontier_1^{(i - 1)}, ..., \frontier_{\initBasisSize}^{(i - 1)}$ via the refinement operators $\pRefine{\cdot}{\cdot}{\tree}$ and $\pRefineD{\cdot}{\cdot}$, we have (see remark \ref{rem:dom}),
\begin{equation}
\frontier_j^{(i)} \preceq \frontier_j^{(i - 1)} \,.
\end{equation}
Corollary (\ref{cor:ancestor_span}) then gives us
\begin{equation}
	\vecspan(\sieve{\phiVec}{\frontier_j^{(i - 1)}}) \subset
	\vecspan(\sieve{\phiVec}{\frontier_j^{(i)}}) \,.
\end{equation}
Summing over $j$ gives us the desired result,
\begin{equation}
\range\left(\romBasis^{(i - 1)}\right) \subset \range\left(\romBasis^{(i)}\right) \,.
\end{equation}
\end{proof}
We have therefore verified that our method exhibits the properties that we
desire in an adaptive basis-refinement method.

\section{Refinement-tree construction} \label{sec:tree_construction}

Thus far, we have assumed that the $\Rn$-refinement tree $\tree$ is provided
as an algorithm input without prescribing its construction. However, its
construction is clearly central to the method's peformance, as the tree
$\tree$ encodes the basis-refinement mechanism. To this end, we
propose two tree-construction approaches:

\begin{enumerate}
\item \textbf{Manual}: Some applications admit a natural decomposition
	mechanism.  For example, if one desires that vector spaces in the tree
		$\tree$ correspond to subdomains of the spatial domain, one could perform
		a recursive partitioning of the spatial domain to generate the tree
		$\tree$. In other situations, perhaps the spatial domain should be split
		until a certain resolution, at which point splitting within each subdomain
		proceeds in frequency space. Clearly, there is a substantial amount of
		flexibility in designing the tree $\tree$ for a particular problem.
		However, it is often unclear how the tree $\tree$ should be designed for
		good performance; this motivates the need for an automated data-driven approach.
\item \textbf{Data-driven}: In the absence of an obvious way to to manually
	design the tree $\tree$, we propose to employ a data-driven method that
		comprises an extension of the tree-construction method proposed in the
		original $h$-refinement work \cite{carlberg2015adaptive}, which is based
		on recursive $k$-means clustering. 
\end{enumerate}

We now briefly summarize the proposed data-driven tree-construction method.
We assume we are provided with two inputs:
\begin{enumerate}
	\item A \textit{snapshot matrix} $\treeSnaps \in \R{\stateVecSize \times \snapCount}$
	whose columns correspond to the FOM solution at a given time and parameter
		instance. Such snapshots are often used for the construction of  the
		original ROM basis $\initBasis$, e.g., in the case of POD. 
	\item An orthogonal \textit{leaf basis} $\leafBasisCol_1, \ldots,
	\leafBasisCol_\stateVecSize$ of $\Rn$, which forms the leaves of the tree,
		i.e., $\treeLeaves = \vecspan(\{\leafBasisCol_i\}_{i=1}^n)$. The
		leaf basis is determined by the user, and the optimal choice is
		highly problem dependent. The method proposed in the original
		$h$-refinement paper \cite{carlberg2015adaptive} corresponds to selecting a leaf
		basis of
		$\leafBasisCol_i = \canonicalVec{i}$, $i=1,\ldots,n$, where
		$\canonicalVec{i}$ denotes the $i$th canonical (Kronecker) unit vector.
\end{enumerate} 

To generate the tree, we recursively cluster the leaf basis vectors
$\leafBasisCol_i$ based on correlations observed in the snapshot data
$\treeSnaps$. To accomplish this, we first represent the snapshot data in
leaf-basis
coordinates, i.e., we compute
\begin{equation}
\treeSnapsTrans \defeq \leafBasis^T \treeSnaps \,,
\end{equation}
where 
$\leafBasis
\equiv
[\leafBasisCol_1\,\cdots\, \leafBasisCol_\stateVecSize]
$
and 
$\treeSnapsTrans$ denotes the transformed snapshot matrix; in particular, the
$i$th row of $\treeSnapsTrans$ represents snapshots of the $i$th transformed
degree of freedom.
We then construct the tree by following the heuristic principle that
the transformed degrees of freedom that exhibit strong correlation or
anti-correlation with each other should be grouped together in the tree
$\tree$.  The rationale behind this heuristic arises from the observation
that if the transformed degree of freedom corresponding to $\leafBasisCol_i$
is always a fixed scalar multiple of the transformed degree of freedom corresponding
to $\leafBasisCol_j$, then those degrees of freedom can be coupled and
represented by a single basis vector without sacrificing accuracy. In
contrast, if those transformed degrees are uncorrelated, then enforcing
their coupling can lead to significant accuracy loss.  Thus, the algorithm
attempts to keep $\leafBasisCol_i$ and $\leafBasisCol_j$ together in the same
refinement-tree node if their respective degrees of freedom exhibit strong
correlation or anti-correlation in the training data.


To formalize the algorithm, we denote by $\dof_i\in\RR{\snapCount}$ the snapshot data
corresponding to $i$th transformed degree of freedom associated with
$\leafBasisCol_i$, i.e.,
\begin{equation}
	\dof_i \equiv (\treeSnapsTrans_{i, :})^T \,.
\end{equation}
To apply $k$-means clustering to achieve our goal, we would like to apply a
transformation such that transformed degrees of freedom that are highly
correlated/anti-correlated will have snapshots that are nearby in
$\RR{\snapCount}$. Following the original $h$-refinement method, we accomplish
this by first normalizing each snapshot $\dof_i$ and negating it if its first
entry is negative,
i.e.,
\begin{equation}\label{eq:kmeansPre}
\dofTrans_i \equiv \begin{cases}
\dof_i / \|\dof_i\|_2 & (\dof_i)_1 \geq 0 \\
-\dof_i / \|\dof_i\|_2 & (\dof_i)_1 < 0\,.
\end{cases}
\end{equation}
We propose to apply recursive $k$-means clustering to the transformed
snapshots $\dofTrans_i$, $i=1,\ldots,\stateVecSize$ until each cluster
contains a single snapshot. This approach constructs the refinement tree in a
level-order manner from the root node to the leaf nodes, and
groups transformed degrees of freedom according to their observed correlation
and anti-correlation. Each cluster defines a vertex in the refinement tree
$\tree$ according to the span of the leaf basis vectors contained in the
cluster.  Algorithms \ref{alg:tree_comp} and \ref{alg:rec_tree_comp} provide
pseudo-code implementations of this procedure.

\begin{algorithm}[h]
\caption{Data-driven refinement tree computation} \label{alg:tree_comp}
\hspace*{\algorithmicindent} \textbf{Input}: The snapshot data $\treeSnaps$, the leaf basis $\leafBasis$ of the tree, the desired number of children $k$ of each vertex in the tree. \\
\hspace*{\algorithmicindent} \textbf{Output}: A $\Rn$-refinement tree $\tree$.
\begin{algorithmic}[1]
\Procedure{GenerateRefinementTree}{$\treeSnaps$, $\leafBasis$, $k$}
\State $\treeSnapsTrans \gets \leafBasis^T \treeSnaps$ \Comment{Transform the
	snapshot data $\snaps$ into the basis given by $\leafBasis$.}
\State $\dof_i \gets (\treeSnapsTrans_{i, :})^T$
	\For{$i \in \nat{\stateVecSize}$} \Comment{Transform $\dof_i$'s so that
	Euclidean distance is inversely proportional to correlation/anti-correlation.}
\If{$(\dof_i)_1 \geq 0$}
	\State $\dofTrans_i \gets \dof_i / \|\dof_i\|_2$
\Else
	\State $\dofTrans_i \gets -\dof_i / \|\dof_i\|_2$
\EndIf
\EndFor
\State $\mat{D} \gets \left[\begin{array}{ccc} \dofTrans_1 & \cdots & \dofTrans_\stateVecSize \end{array}\right]$
\State $S \gets \{ 1,\ldots, n \}$
\State $\tree \gets \textsc{BuildTreeRecursive}(\mat{D}, S, \leafBasis, k)$ \Comment{Assemble the tree by performing recursive $k$-means clustering.}
\State \Return $\tree$
\EndProcedure
\end{algorithmic}
\end{algorithm}

\begin{algorithm}[h]
\caption{Recursive tree construction via $k$-means} \label{alg:rec_tree_comp}
	\hspace*{\algorithmicindent} \textbf{Input}: The matrix
	$\mat{D}\equiv[\dofTrans_1\ \cdots\ \dofTrans_\stateVecSize]$, a set $S$ of
	indices of the columns of $\mat{D}$ that span the current vector space, the orthogonal leaf basis $\leafBasis$, and the desired number of children $k$ of each vertex in the tree. \\
\hspace*{\algorithmicindent} \textbf{Output}: An $\Rn$-refinement tree $\tree$.
\begin{algorithmic}[1]
\Procedure{BuildTreeRecursive}{$\mat{D}$, $S$, $\leafBasis$, $k$}
\State $\leafBasis_\text{active} \gets \leafBasis_{:, S}$
	\State $\rootSpace \gets \vecspan(\{\leafBasis_\text{active}\})$
	\State $\tree \gets (\{ \rootSpace \}, \emptyset)$
\If{$|S| \neq 1$}
\State $\mat{D}_\text{active} \gets \mat{D}_{:, S}$
\State $(S_1, \ldots, S_k) \gets \textsc{KMeans}(\mat{D}_\text{active})$
\Comment{Returns $k$ index sets denoting
	columns of $\mat{D}$ within each cluster.}
	\For{$i \in \nat{k}$}
	\If{$S_i \neq \emptyset$}
		\State $T_C \gets \textsc{BuildTreeRecursive}(\mat{D}, S_i, \leafBasis, k)$
		\State $T \gets \textsc{Graft}(T, \rootSpace, T_C)$
		\Comment{Grafts sub-tree $T_C$ onto tree $T$ such that the root of $T_C$ is a child of
		$\rootSpace$.}
	\EndIf
\EndFor
\EndIf
\State \Return $\tree$
\EndProcedure
\end{algorithmic}
\end{algorithm}

\section{Online basis compression} \label{sec:online_basis_compression}

When implemented directly within a time-integration loop, Algorithm
\ref{alg:refine_proc} produces a sequence of reduced bases of monotonically
increasing dimension; indeed, the dimension of the refined basis will increase
monotonically until the basis spans $\RR{\stateVecSize}$.  The original
$h$-refinement method \cite{carlberg2015adaptive} controlled the refined-basis
dimension by simply resetting the refined basis $\romBasis$ to the original
basis $\initBasis$ after a prescribed number of time steps. However, as
mentioned in the introduction, this effectively discards all information
gained during refinement.  To address this, we now present a novel online
basis-compression method that comprises the second key contribution of this
work. 

The proposed method operates as follows:  when either the refined-basis
dimension exceeds a specified threshold or a prescribed number of time steps
has elapsed, the method performs a compression of the refined basis
$\romBasis$ via an efficient online POD of snapshot data generated by the
refined ROM since the previous compression. The method then uses this POD to
enrich the original reduced basis $\initBasis$ and significantly reduce the
dimension of the refined basis.  Critically, we supply an algorithm that
performs this POD while incurring an operation count that depends only on the
refined ROM dimension $\romSize$ and not on the FOM dimension $\stateVecSize$.

\subsection{Compression via metric-corrected POD}\label{sec:compressionmetric}

We begin by establishing the setting of the proposed algorithm. We suppose
that we are given an initial basis
$\initBasis\equiv[\initBasisCol_1\, \cdots\, \initBasisCol_{\romSize_0}]$
and a refined basis
\begin{equation} \label{eq:rom_basis}
\romBasis = \left[\begin{array}{cccc} \matsieve{\initBasisCol_1}{\frontier_1}
& \cdots & \matsieve{\initBasisCol_{\romSize_0}}{\frontier_{\romSize_0}}
\end{array}\right] \,.
\end{equation}
We would like to reset the dimension $\romSize$ of the refined
basis $\romBasis$ to something comparable to the dimension
$\romSize_0$ of the original basis $\initBasis$. 
We assume that we are provided online snapshot data corresponding to $\compSnaps$
solutions of the refined ROM, denoted by
\begin{equation} \label{eq:snaps}
\snaps = \romBasis \romSnaps\in 
	\mathbb{R}^{\stateVecSize \times \compSnaps}
	\,,
\end{equation}
where 
$\romSnaps \in \mathbb{R}^{\frontierglobal \times \nat{\compSnaps}}$ denotes the representation of the
snapshot data in the coordinates of the refined basis.  
In practice, we take $\snaps$ to be solutions of the refined ROM at
$\compSnaps$ previous (online) time steps, i.e.,

\begin{equation}
\snaps = \left[\begin{array}{ccccc} \stateVec^{\iter} & \stateVec^{\iter - 1}
& \cdots & \stateVec^{\iter - \compSnaps+1} \end{array}\right] \,.
\end{equation}
If the dimension 
$\romSize$ of the refined basis exceeds that of the original basis
$\romSize_0$, then these snapshots cannot in general be represented using 
the original basis $\initBasis$. Hence, compression of these snapshots will
preserve solution components not present in the initial basis. However, we
require the algorithm to be online efficient such that its
operation count does not depend on
$\stateVecSize$; as a result, the method can operate only on the coordinate representation
$\romSnaps$. 

\subsubsection{Naive approach}

In the context of $\Rn$-refinement trees, the proposed basis-compression
approach entails overwriting the initial basis $\initBasis$ and resetting the
frontiers $\frontier_i$. The new basis $\initBasis$ should capture the
additional information contained in the online snapshot data $\snaps$. Because this framework
requires the ability to distinguish the original version of $\initBasis$ from
the current version of $\initBasis$, we denote the original reduced basis by
$\initBasisZ$ and the basis produced after the $\compTime$th compression by
$\initBasisIter\in\stiefel{\stateVecSize \times \romSize^{(\compTime)}_0}$. The compression procedure outlined in the subsequent
sections employs $\snaps^{(\compTime)}$ to produce $\initBasisIter$, where
$\snaps^{(\compTime)}$ denotes the online snapshot data available during the
$\compTime$th compression.
However, for notational simplicity, we simply write $\snaps$ instead of
	$\snaps^{(\compTime)}$, as $\snaps$ is always used in the context of the
	$\compTime$th basis compression. Likewise, all variables introduced
	in the next two sections are associated with the scope of the $r$th basis
	compression.

Because the original basis $\initBasisZ$ typically comprises the compression of
a large amount of training data, the proposed method always retains
the original basis $\initBasisZ$ in the compressed
basis $\initBasisIter$, i.e., 
\begin{equation} \label{eq:initBasis}
\initBasisIter = \left[\begin{array}{cc} \initBasisZ & \compVecs \end{array}\right]\,,
\end{equation}
where $\compVecs\in\RR{n\times \compSize}$ comprises $\compSize$ enrichment
vectors computed
from the online snapshot data $\snaps$.

 Due to the imposed form of the compressed basis \eqref{eq:initBasis}, the
 enrichment vectors $\compVecs$ should represent the compression of the
 components of the online
 snapshot data orthogonal to the range of $\initBasisZ$. Thus, we compute
 $\compVecs$ from the POD of the projected snapshot
 matrix
\begin{equation} \label{eq:proj_snaps}
\projSnaps \defeq \snaps - \initBasisZ (\initBasisZ)^T \snaps \,.
\end{equation}
That is, we compute the singular value decomposition
\begin{equation} \label{eq:fom_svd}
\projSnaps = \mat{U} \mat{\Sigma} \mat{V}^T \,,
\end{equation}
and subsequently set
\begin{equation}
\compVecs = \left[\begin{array}{cccc} {\ve{u}}_1 &  
\cdots & {\ve{u}}_s \end{array}\right] \,,
\end{equation}
where 
${\mat{U}}\equiv[{\ve{u}}_1\ \cdots\ {\ve{u}}_r]$ and
$\compSize$ is selected using a
singular-value threshold.
We then reset the frontiers $\frontier_i$ to
\begin{equation}
\frontier_i \gets \{ \mathbb{R}^\stateVecSize \}, \quad i=1,\ldots,\romSize_0
+ \compSize\,.
\end{equation}

While this idea is simple, we cannot explicitly perform the
above operations because each incurs an $\stateVecSize$-dependent operation
count, which
precludes online efficiency. Fortunately, because 
$\ve{u}_i\in\range(\romBasis)$ for all $i\in\nat{\compSize}$, there exists a
representation 
$\romCompVecs\in\RR{\frontierglobal \times \nat{\compSize}}$
of the enrichment vectors in the coordinates of the refined
basis $\romBasis$ such that
\begin{equation}\label{eq:enrichVecDecomp}
\compVecs = \romBasis \romCompVecs \,.
\end{equation}
Thus, we can achieve an online-efficient basis-compression algorithm by
computing the enrichment-vector representation $\romCompVecs$ from the
snapshot-data representation $\romSnaps$ without resolving anything in 
$\mathbb{R}^{\stateVecSize}$. We now describe this approach.

\subsubsection{Metric-corrected coordinate representation approach}

The first step is to compute the representation
$\romProjSnaps\in\RR{\frontierglobal \times \nat{\compSnaps}}$ of the projected data
$\projSnaps$ in the coordinates of basis $\romBasis$ such that
\begin{equation}
\projSnaps = \romBasis \romProjSnaps \,.
\end{equation}
Recall from Eq.\ (\ref{eq:initBasis}) that
$\initBasisZ$ is the given by the first $\romSize^{(0)}_0$ columns of
$\initBasisIterLast$; thus, the columns of 
$\romBasis$ associated with frontiers
$\frontier_1,\ldots,\frontier_{\romSize^{(0)}_0}$ were sieved from
the columns of $\initBasisZ$. As a result, we can write
\begin{equation}
\initBasisColZ_i = \sum_{\romBasisCol \in
	\sieve{\initBasisColZ_i
	}{\frontier_i}} \romBasisCol,\quad i=1,\ldots,\romSize^{(0)}_0\,.
\end{equation}
%
This implies 
\begin{equation}\label{eq:phi_prolong}
\initBasisZ = \romBasis \romInitBasisZ \,,
\end{equation}
where $\romInitBasisZ\in \{0,1\}^{\frontierglobal \times
\nat{\romSize^{(0)}_0}}$ associates with
the prolongation 
from $\initBasisZ$ to $\romBasis$.
Now, 
substituting Eqs.~\eqref{eq:phi_prolong} and \eqref{eq:snaps}
in Eq.~\eqref{eq:proj_snaps}
yields
\begin{equation}\label{eq:projSnapsDecomp}
\projSnaps = \romBasis \romSnaps - \romBasis \romInitBasisZ (\romBasis
	\romInitBasisZ )^T \romBasis \romSnaps  = \romBasis 
	\romProjSnaps\,,
\end{equation}
where
\begin{equation}
\romProjSnaps \defeq \romSnaps - \romInitBasisZ (\romInitBasisZ)^T \left(\romBasis^T \romBasis\right) \romSnaps \,
\end{equation}
denotes the (desired) coordinate representation of the projected data
$\projSnaps$, and
the matrix $\romBasis^T \romBasis \in \mathbb{R}^{\frontierglobal \times
\frontierglobal}$
is the induced metric on the range of $\romBasis$ in canonical coordinates.
Note that although the matrix $\romBasis$ has $\stateVecSize$ rows, the
dimension
of the matrix $\romBasis^T \romBasis$ is independent of $\stateVecSize$
and hence can be applied or factorized in an online-efficient manner.
Moreover, while a naive computation of the metric $\romBasis^T \romBasis$
would incur an $\stateVecSize$-dependent operation count, it is
possible to ensure an $\stateVecSize$-independent
operation count by performing offline precomputations
and by efficiently traversing of the refinement tree $\tree$. For now, 
we assume that the metric
$\romBasis^T \romBasis$ is provided and postpone discussion of its efficient
computation until 
section \ref{sec:metric}.

Because the enrichment vectors $\compVecs$ correspond to the first $\compSize$ left singular
vectors of the projected snapshot matrix $\projSnaps$, they provide a solution
to the optimization problem
\begin{equation}\label{eq:optProbOne}
\begin{aligned}
	& \underset{\compVecsDummy\in \mathbb{R}^{\stateVecSize \times
	\compSize}}{\text{minimize}}
& & \left\| \projSnaps - \compVecsDummy \compVecsDummy^T \projSnaps \right\|_2^2 \\
& \text{subject to}
& & \compVecsDummy^T \compVecsDummy = \mat{I}\,.
\end{aligned}
\end{equation}
Substituting Eq.~\eqref{eq:projSnapsDecomp} 
into Problem \eqref{eq:optProbOne}, we notice from Eq.\eqref{eq:enrichVecDecomp} that 
computing $\compVecs$ as a solution to \eqref{eq:optProbOne}
is equivalent to computing $\romCompVecs$ as a solution to 
\begin{equation}\label{eq:optProbTwo}
\begin{aligned}
	& \underset{\romCompVecsDummy\in \mathbb{R}^{\frontierglobal \times
	\nat{\compSize}}}{\text{minimize}}
& & \left\| \romBasis \romProjSnaps - \romBasis \romCompVecsDummy
	\romCompVecsDummy^T \romBasis^T \romBasis \romProjSnaps \right\|_2^2 \\
& \text{subject to}
& & \romCompVecsDummy^T \romBasis^T \romBasis \romCompVecsDummy = \mat{I}\,.
\end{aligned}
\end{equation}
Because the matrix $\romBasis^T \romBasis$ is symmetric positive semidefinite, there
always exists a symmetric factorization
\begin{equation} \label{eq:corrector}
\romBasis^T \romBasis = \corrector^T \corrector
\end{equation}
with $\corrector \in \mathbb{R}^{\frontierglobal \times \frontierglobal}$, which
can be computed using the eigenvalue decomposition or Cholesky factorization
of $\romBasis^T \romBasis$, for example.

Using Eq.~\eqref{eq:corrector} and the relation
$\|\romBasis \mat{A}\|_2 = \|\corrector \mat{A}\|_2$, we can write Problem
\ref{eq:optProbTwo} equivalently as
\begin{equation}\label{eq:optProbThree}
\begin{aligned}
	& \underset{\romCompVecsDummy\in \mathbb{R}^{\frontierglobal \times
	\nat{\compSize}}}{\text{minimize}}
& & \left\| \corrector \romProjSnaps - \corrector \romCompVecsDummy
	\romCompVecsDummy^T \corrector^T \corrector \romProjSnaps \right\|_2^2 \\
& \text{subject to}
& & \romCompVecsDummy^T \corrector^T \corrector \romCompVecsDummy = \mat{I}\,.
\end{aligned}
\end{equation}
Computing $\romCompVecs$ as a solution to \eqref{eq:optProbThree} is
equivalent to computing $\transComp = \corrector \romCompVecs$ as the solution
to 
\begin{equation*}
\begin{aligned}
	& \underset{\transCompDummy \in \mathbb{R}^{\frontierglobal \times
	\nat{\compSize}}}{\text{minimize}}
& & \left\| \corrector \romProjSnaps - \transCompDummy \transCompDummy^T \corrector \romProjSnaps \right\|_2^2 \\
& \text{subject to}
& & \transCompDummy^T \transCompDummy = \mat{I}\,,
\end{aligned}
\end{equation*}
which we recognize as equivalent to performing a POD on
the transformed coordinate data $\corrector \romProjSnaps$. Its solution is
given by first computing the singular value decomposition 
\begin{equation} \label{eq:rom_svd}
\corrector \romProjSnaps = \tilde{\mat{U}} \tilde{\mat{\Sigma}} \tilde{\mat{V}}^T \,,
\end{equation}
and then setting
\begin{equation}
\transComp = \left[\begin{array}{cccc} \tilde{\ve{u}}_1  &
\cdots & \tilde{\ve{u}}_\compSize \end{array}\right] \,,
\end{equation}
where $\tilde{\mat{U}}\equiv[\tilde{\ve{u}}_1\ \cdots\
\tilde{\ve{u}}_\compSnaps]$.
Moreover, because the transformation $\corrector$ induces the correct metric
on the data $\romProjSnaps$, the singular values of the decomposition in Eq.\
(\ref{eq:rom_svd}) are identical to the singular values in Eq.\
(\ref{eq:fom_svd}), and the same singular-value threshold can be used to
select the dimension $\compSize$.

 We are left with the issue of computing $\romCompVecs$
 from $\transComp$. In principle, we can achieve this by directly computing
\begin{equation} \label{eq:ill_back_trans}
\romCompVecs = \corrector^{-1} \transComp \,.
\end{equation}
Unfortunately, there is no assurance that 
the metric $\romBasis^T \romBasis$ is well conditioned; indeed, we often
observe it to be
ill conditioned in practice, in which case the matrix
$\corrector$ inherits 
this ill conditioning. Thus, 
computing $\romCompVecs$ 
via Eq.\ (\ref{eq:ill_back_trans}) 
directly is not practical.
Fortunately, 
we can avoid
this issue
as long as the singular-value threshold
is not overly aggressive. Rearranging Eq.\ (\ref{eq:rom_svd}) gives
\begin{equation}
\corrector^{-1} \tilde{\mat{U}} = \romProjSnaps \tilde{\mat{V}} \tilde{\mat{\Sigma}}^{-1} \,.
\end{equation}
Here, the ill conditioning of the matrix $\corrector^{-1}$ results in ill
conditioning of the matrix $\tilde{\mat{\Sigma}}^{-1}$. However, since we 
aim to compute only the first $\compSize$ columns of $\corrector^{-1} \tilde{\mat{U}}$,
we can 
compute
$\romCompVecs$
using only the well-conditioned
part of
$\tilde{\mat{\Sigma}}^{-1}$ via
\begin{equation} \label{eq:trick}
\romCompVecs = \corrector^{-1} \transComp = \romProjSnaps \tilde{\mat{V}}_{:,
	\nat{\compSize}} \tilde{\mat{\Sigma}}_{\nat{\compSize},
	\nat{\compSize}}^{-1} \,.
\end{equation}

Thus, by computing the singular value decomposition of the `metric-corrected'
data $\corrector \romProjSnaps$ according to Eq.~\eqref{eq:rom_svd} and
subsequently using Eq.~\eqref{eq:trick} to compute $\romCompVecs$, we can
effectively perform a POD of the data $\snaps$ while ensuring an
$\stateVecSize$-independent operation count.  Algorithm \ref{alg:pod}
reports this procedure.

\begin{algorithm}
\caption{Metric-corrected coordinate POD} \label{alg:pod}
\hspace*{\algorithmicindent} \textbf{Input}: The $\mathbb{R}^n$-refinement
	tree $\tree$, the current global frontier union $\frontierglobal \equiv
	\bigsqcup_i \frontier_i$, the coordinate representation $\romSnaps$ of the
	input data $\snaps$, and a singular-value threshold $\epsilon$. \\
\hspace*{\algorithmicindent} \textbf{Output}: The coordinate representation
	$\romCompVecs$ of the dominant left singular vectors $\compVecs$ of $\snaps
	- \initBasisZ (\initBasisZ)^T \snaps$.
\begin{algorithmic}[1]
\Procedure{MetricCorrectedPOD}{$\tree, \frontierglobal, \romSnaps, \epsilon$}
   \State $\romInitBasis \gets
	 \textsc{GetProlongationOperator}(\frontierglobal)$
	 \Comment{See Eq.\ (\ref{eq:phi_prolong}) for details.}
	 \State $\romInitBasisZ \gets \romInitBasis_{:, \nat{m}}$
	 \State\label{step:computeMetric} $\metric \gets
	 \textsc{ComputeMetric}(\tree, \frontierglobal)$
	 \Comment{See 
	 algorithm (\ref{alg:metric}.})
   \State $\romProjSnaps \gets \romSnaps - \romInitBasisZ (\romInitBasisZ)^T \metric \romSnaps$ \Comment{Project out the original basis $\initBasis$}
	 \State\label{step:factorize} $\corrector \gets \textsc{FactorizeMetric}(\metric)$
	 \Comment{Compute $\corrector \in \mathbb{R}^{\frontierglobal \times
	 \frontierglobal}$
	 such that $\corrector^T \corrector = \metric$. See Eq.\ (\ref{eq:corrector}) for details.}
   \State $\tilde{\mat{Y}}_{\perp} \gets \corrector \romProjSnaps$ \Comment{Induce the desired metric $\metric$ onto the data}
   \State\label{step:SVD} $(\tilde{\mat{U}}, \tilde{\mat{\Sigma}}, \tilde{\mat{V}}) \gets \textsc{SVD}(\tilde{\mat{Y}}_{\perp})$
   \State $\sigma \gets \|\corrector \romSnaps\|_2$ \Comment{Compute the 2-norm $\|\snaps\|_2 = \|\romBasis \romSnaps\|_2 = \|\corrector \romSnaps\|_2$}
	 \State $s \gets \max\{i\,|\, \tilde{\sigma_i} \geq \epsilon \sigma\}$
	 \Comment{Select the maximum index of the set of singular values above the
	 threshold
	 $\epsilon\sigma$}
	 \State $\romCompVecs \gets \romProjSnaps \tilde{\mat{V}}_{:, \nat{s}}
	 (\tilde{\mat{\Sigma}}_{\nat{s}, \nat{s}})^{-1}$ \Comment{Use Eq.\ (\ref{eq:trick}) to avoid ill-conditioning issues.}
      \State \Return $\romCompVecs$
\EndProcedure
\end{algorithmic}
\end{algorithm}

To ensure this algorithm remains online efficient, we must ensure that
computing the metric 
 $\romBasis^T \romBasis$
in Step \ref{step:computeMetric} incurs an
$\stateVecSize$-independent operation count. We now provide a method to
accomplish this.

\subsection{Computing the metric} \label{sec:metric}
We first establish a few notational conveniences. To begin, note that
the orthogonal projection $\proj_{\parentSpace}(\ve{w})$ of any vector $\ve{w}
\in \rootSpace$ onto a subspace $\parentSpace$ can be represented as a linear
operator, which we denote by 
$\projector_\parentSpace$ and define as
\begin{equation}
\projector_\parentSpace \ve{w} \equiv \proj_{\parentSpace}(\ve{w}) \,.
\end{equation}
Note that orthogonal projectors are 
idempotent (i.e.,
$\projector_\parentSpace^2 = \projector_\parentSpace$)
and
self-adjoint (i.e., 
${\projector_\parentSpace}^T = \projector_\parentSpace$).
Moreover, if $\childSpace$ is a subspace of $\parentSpace$, then
$\projector_{\childSpace} \projector_{\parentSpace} = \projector_{\parentSpace}  \projector_{\childSpace} = \projector_{\childSpace}
$,
and if $\childSpace \perp \parentSpace$, then
$\projector_{\childSpace} \projector_{\parentSpace} =
\projector_{\parentSpace}  \projector_{\childSpace} = \mat{0}$.
Further, if $\decomp$ is an orthogonal decomposition of $\parentSpace$, then
\begin{equation} \label{eq:decomp_add}
\projector_\parentSpace = \sum_{\childSpace \in \decomp} \projector_{\childSpace} \,.
\end{equation}
Note that Eq.~\eqref{eq:decomp_add} implies that if $\frontier$ is a frontier
in a $\rootSpace$-refinement tree $\tree$, then
$
\mat{I} = \projector_\rootSpace = \sum_{\childSpace \in \frontier} \projector_\childSpace
$.
By construction, the every column in the basis $\romBasis$ is related to a
vector in the basis $\initBasis \equiv \initBasisIterLast$ through a
projection as
\begin{equation}
\romBasisCol_{\parentSpace} = \projector_{\parentSpace} \initBasisCol_{\globalancestor(\parentSpace)} \,,
\end{equation}
where $\romBasisCol_{\parentSpace}$ denotes 
the $\parentSpace$th column of $\romBasis$ and
\begin{equation} \label{eq:globalancestor}
	\globalancestor : \frontierglobal \longrightarrow \nat{\romSize_0} 
\end{equation}
maps the subspace $\parentSpace\in\frontierglobal$ to
the index of the column of the original basis $\initBasis$ from which
$\romBasisCol_{\parentSpace}$ was sieved such that
$\globalancestor(\parentSpace) = i$ for $\parentSpace \in \frontier_i$.

Now, note that element $(\parentSpace,\childSpace)$ of the metric $\romBasis^T
\romBasis$ 
is given by
\begin{equation}\label{eq:UWinnerprod}
\left[\romBasis^T \romBasis\right]_{\parentSpace,\childSpace}={\romBasisCol_{\parentSpace}}^T \romBasisCol_{\childSpace} = {\initBasisCol_{\globalancestor(\parentSpace)}}^T {\projector_{\parentSpace}}^T {\projector_{\childSpace}} \initBasisCol_{\globalancestor(\childSpace)} = {\initBasisCol_{\globalancestor(\parentSpace)}}^T \left({\projector_{\parentSpace}} {\projector_{\childSpace}}\right) \initBasisCol_{\globalancestor(\childSpace)} \,.
\end{equation}
Moreover, since $\parentSpace$ and $\childSpace$ are derived from an
$\mathbb{R}^{\stateVecSize}$-refinement tree $\tree$, corollary
(\ref{incomparable_prop}) implies that either $\parentSpace \subset
\childSpace$, $\childSpace \subset \parentSpace$, or $\childSpace \perp
\parentSpace$, which in turn implies
\begin{equation}\label{eq:PuPw}
{\projector_{\parentSpace}} {\projector_{\childSpace}} = \begin{cases} \projector_{\parentSpace} & \parentSpace \subset \childSpace \\
\projector_{\childSpace} & \childSpace \subset \parentSpace \\
0 & \text{otherwise}. \end{cases}
\end{equation}
Eqs.~\eqref{eq:UWinnerprod} and \eqref{eq:PuPw} yield
\begin{equation}
{\romBasisCol_{\parentSpace}}^T \romBasisCol_{\childSpace} = \begin{cases} {\initBasisCol_{\globalancestor(\parentSpace)}}^T \projector_{\parentSpace} \initBasisCol_{\globalancestor(\childSpace)}  & \parentSpace \subset \childSpace \\
{\initBasisCol_{\globalancestor(\parentSpace)}}^T \projector_{\childSpace} \initBasisCol_{\globalancestor(\childSpace)} & \childSpace \subset \parentSpace \\
0 & \text{otherwise}, \end{cases}
\end{equation}
which can be written equivalently as
\begin{equation} \label{eq:metric}
\left[\romBasis^T \romBasis\right]_{\parentSpace,\childSpace} = \begin{cases} \left[\initBasis^T \projector_{\parentSpace} \initBasis\right]_{\globalancestor(\parentSpace), \globalancestor(\childSpace)} & \parentSpace \text{ is descendant from } \childSpace \text{ in } \tree \\
\left[\initBasis^T \projector_{\childSpace} \initBasis\right]_{\globalancestor(\parentSpace), \globalancestor(\childSpace)} & \childSpace \text{ is descendant from } \parentSpace \text{ in } \tree \\
0 & \text{otherwise}. \end{cases}
\end{equation}
This expression illuminates two important points:
\begin{enumerate}
\item The metric $\romBasis^T \romBasis$ is sparse. Moreover, the
	sparsity pattern can be efficiently computed from Eq.\ (\ref{eq:metric})
		in $O(\romSize^2 \min(\log \stateVecSize, \romSize))$ operations: for each pair $(\parentSpace,
		\childSpace)\in\frontierglobal\times\frontierglobal$, we
		can compute the common ancestor of $\parentSpace$ and $\childSpace$ in
		$O(\min(\log \stateVecSize, \romSize))$, and
		the entry is zero if the ancestor is neither
		$\parentSpace$ nor $\childSpace$.
\item The matrices $\initBasis^T \projector_{\parentSpace} \initBasis$, which
	we refer to as the \textbf{projected metrics}, have dimension
		$\romSize_0 \times \romSize_0$. Thus, since $\romSize_0 \ll
		\stateVecSize$, we can store a significant number of them in memory.
		Moreover, these projected metrics are additive: if $\decomp$ is an
		orthogonal decomposition of $\parentSpace$, then multiplying Eq.\
		(\ref{eq:decomp_add}) on the left and right by $\initBasis^T$ and 
		$\initBasis$, respectively, yields
\begin{equation}
\initBasis^T \projector_{\parentSpace} \initBasis = \sum_{\childSpace \in \decomp} \initBasis^T \projector_{\childSpace} \initBasis \,.
\end{equation}
In particular, for every non-leaf node, the projected metric at that node is the sum of the projected metrics at its children,
\begin{equation} \label{eq:proj_metric_rec}
\initBasis^T \projector_{\parentSpace} \initBasis = \sum_{\childSpace \in
	\tRefine{\parentSpace}{T}} \initBasis^T \projector_{\childSpace} \initBasis
	\,.
\end{equation}
Hence, projected metrics can be computed recursively. Indeed, if the projected
		metrics are supplied for any frontier in the tree $\tree$, then the
		projected metrics for all vertices in $\tree$ above the frontier can be computed by recursively applying Eq.\ (\ref{eq:proj_metric_rec}).
\end{enumerate}

In light of these two considerations, our algorithm for computing the metric
simply precomputes the projected metrics $\initBasis^T
\projector_{\parentSpace} \initBasis$ for every vertex $\parentSpace$ in the
tree $\tree$. We store these precomputed projected metrics in ``projected-metric attributes.''

\begin{definition}[projected-metric attribute]
We ascribe to every vertex $\parentSpace$ of the
	$\mathbb{R}^\stateVecSize$-refinement tree $\tree$ a \textbf{projected
	metric attribute} $\projMetric_\tree(\parentSpace)$. Initially, all of the
	projected-metric attributes are set to the projected metrics of those
	subspaces, i.e.,
\begin{equation}
\projMetric_\tree(\parentSpace) \gets (\initBasisZ)^T
	\projector_{\parentSpace} \initBasisZ\,.
\end{equation}
\end{definition}
The culmination of all of the above machinery yields algorithm
(\ref{alg:metric}). Some technical considerations prevent us from giving the
full implementation of the procedure $\textsc{GetProjectedMetric}$ as of yet,
but for now, think of $\textsc{GetProjectedMetric}(\tree,
\childSpace)$ as simply returning the projected-metric attribute
$\projMetric_\tree(\childSpace)$ stored at $\childSpace$. The procedure
$\textsc{GetGlobalAncestorMap}(\frontierglobal)$ is simple to implement with
proper bookkeeping, as it returns
the mapping from any element of the global frontier $\parentSpace
\in \frontierglobal$ to the index $i$ of the original basis vector $\initBasisCol_i$
from which it was sieved.

\begin{algorithm}
\caption{Metric Computation} \label{alg:metric}
\hspace*{\algorithmicindent} \textbf{Input}: The $\mathbb{R}^n$-refinement
	tree $\tree$, and the current global frontier $\frontierglobal$. \\
\hspace*{\algorithmicindent} \textbf{Output}: The metric $\metric = \romBasis^T \romBasis$.
\begin{algorithmic}[1]
\Procedure{ComputeMetric}{$\tree, \frontierglobal$}
   \State $\metric \gets \mat{0} \in \mathbb{R}^{\frontierglobal \times
	 \frontierglobal}$
   \State $\globalancestor \gets \textsc{GetGlobalAncestorMap}(\frontierglobal)$
   \For{$\parentSpace, \childSpace \in \frontierglobal$}
   		\State $\mathbb{A} \gets \textsc{GetCommonAncestor}(\tree, \parentSpace, \childSpace)$
   		\If{$\mathbb{A} = \parentSpace$} \Comment{In this case, $\childSpace \subset \parentSpace$}
   			\State $\metric^{\proj} \gets \textsc{GetProjectedMetric}(\tree, \childSpace)$ \Comment{Retrieve $\initBasis^T \projector_\childSpace \initBasis$. See full implementation in algorithm \ref{alg:get_proj_metric}.}
        		\State $\metric_{\parentSpace,\childSpace} \gets \metric^{\proj}_{\globalancestor(\parentSpace), \globalancestor(\childSpace)}$ \Comment{Compute the $\parentSpace, \childSpace$ entry of $\metric$. See Eq.\ (\ref{eq:metric})}
        \ElsIf{$\mathbb{A} = \childSpace$} \Comment{In this case, $\parentSpace \subset \childSpace$}
        \State $\metric^{\proj} \gets \textsc{GetProjectedMetric}(\tree, \parentSpace)$ \Comment{Retrieve $\initBasis^T \projector_\parentSpace \initBasis$. See full implementation in algorithm \ref{alg:get_proj_metric}.}
        		\State $\metric_{\parentSpace,\childSpace} \gets \metric^{\proj}_{\globalancestor(\parentSpace), \globalancestor(\childSpace)}$ \Comment{Compute the $\parentSpace, \childSpace$ entry of $\metric$. See Eq.\ (\ref{eq:metric})}
        \EndIf
      \EndFor
   \State \textbf{return} $\metric$
\EndProcedure
\end{algorithmic}
\end{algorithm}

\subsection{The meet of frontiers}\label{sec:meetFrontiers}

Because the original basis vectors $\initBasisCol_i$ are sieved
through distinct frontiers $\frontier_i$, comparing the resulting columns of 
$\romBasis$ is difficult. In the coming section, we will require a way of
comparing the vectors of $\romBasis$ using a
single frontier of the refinement tree $\tree$. Thus, in this section, we
describe how to find a
frontier $\bigwedge_{i = 1}^m \frontier_i$ in the tree $\tree$ that satisfies
$\bigwedge_{i = 1}^m \frontier_i \preceq \frontier_i$ for all $i \in
\nat{\romSize_0}$, but is still as coarse as possible. These considerations
motivate the following definition.
\begin{definition}[frontier meet] \label{def:meet1}
The \textbf{meet} (i.e., largest lower bound) of a collection of frontiers
	$\frontier_1, \ldots, \frontier_m$ of a $\rootSpace$-refinement tree $\tree$
	is denoted by $\bigwedge_{i = 1}^m \frontier_i$ and is defined by the
	following
	properties:
\begin{enumerate}
\item $\bigwedge_{i = 1}^m \frontier_i$ is a frontier.
\item $\bigwedge_{i = 1}^m \frontier_i$ is a lower bound for $\{ \frontier_1,
	\ldots, \frontier_m \}$. That is, $\bigwedge_{i = 1}^m \frontier_i \preceq
		\frontier_j$ for all $j \in \nat{m}$. Or in other words, all frontiers $\frontier_j$ can be refined to $\bigwedge_{i = 1}^m \frontier_i$.
\item $\bigwedge_{i = 1}^m \frontier_i$ is coarser than all other lower bounds
	for $\{ \frontier_1, \ldots, \frontier_m \}$. That is, if $\frontierH \preceq
		\frontier_j$ for all $j \in \nat{m}$, then $\frontierH \preceq \bigwedge_{i = 1}^m \frontier_i$.
\end{enumerate}
\end{definition}

However, this definition is nonconstructive and so the uniqueness and existence of the meet is left in question. Hence, we will adopt an alternate constructive definition and prove in the appendix that the two are equivalent.

\begin{definition}[frontier meet] \label{def:meet2}
The \textbf{meet} (i.e., largest lower bound) of a collection of frontiers $\frontier_1, \ldots \frontier_m$ of a $\rootSpace$-refinement tree $\tree$ is denoted and defined as
\begin{equation}
\bigwedge_{i = 1}^m \frontier_i = \{ \parentSpace_1 \cap \parentSpace_2 \cap \ldots \cap \parentSpace_m \mid \parentSpace_j \in \frontier_j, \, \parentSpace_1 \cap \parentSpace_2 \cap \ldots \cap \parentSpace_m \neq 0 \} \,.
\end{equation}
\end{definition}
\noindent That is, the meet is the set of all nontrivial intersections between the elements of each of the frontiers $\frontier_i$. 
To gain a more workable characterization of the meet of frontiers, we
introduce the following propositions.

\begin{restatable}{proposition}{propmeet} \label{prop:meet}
The meet $\bigwedge_i \frontier_i$ of a collection of frontiers
	$\frontier_1,\ldots, \frontier_m$ of a $\rootSpace$-refinement tree $\tree
	\equiv (\treeNodes, \treeEdges)$ is also a frontier.
\end{restatable}

\begin{restatable}{proposition}{propmeetsub} \label{prop:meetsub}
Every element of the meet $\bigwedge_i \frontier_i$ is an element of some $\frontier_i$, i.e., $\bigwedge_i \frontier_i \subset \bigcup_i \frontier_i$.
\end{restatable}

The above proposition (\ref{prop:meetsub}) can be used to prove that the constructive definition (\ref{def:meet2}) and the prescriptive definition (\ref{def:meet1}) are in fact equivalent.

\begin{restatable}{proposition}{proplargestlowerbound} \label{prop:lower_bound}
The meet $\bigwedge_i \frontier_i$ is the largest lower bound of the
	$\frontier_i$'s. That is, $\bigwedge_i \frontier_i \preceq \frontier_j$ for
	all $j \in\nat{m}$ and for any frontier $\frontierH$ such that $\frontierH
	\preceq \frontier_i$ for all $i \in \nat{m}$, we have that $\frontierH\preceq \bigwedge_i \frontier_i$.
\end{restatable}

Finally, we provide a characterization of the meet which can be used for efficient computation.

\begin{restatable}{proposition}{propnodescendants} \label{prop:meet}
The meet $\bigwedge_i \frontier_i$ is the subset of all elements $\parentSpace
	\in \bigcup_i \frontier_i$ that have no descendants in $\left(\bigcup_i \frontier_i\right) \setminus \parentSpace$.
\end{restatable}

In light of the above, one can compute the meet $\bigwedge_i \frontier_i$ of a
collection of frontiers by performing an upward flood-fill of the tree $\tree$
starting at the vertices in $\frontier_i$. Afterwards, we can exact the meet
$\bigwedge_i \frontier_i$ by removing all of the vertices in $\bigcup_i
\frontier_i$ that were marked during this flood-fill process. 

\begin{algorithm}
\caption{Meet Computation via Flood-Fill} \label{alg:meet}
\hspace*{\algorithmicindent} \textbf{Input}: A collection of frontiers $\frontier_1, \ldots, \frontier_m$ in a $\rootSpace$-refinement tree $\tree$. \\
\hspace*{\algorithmicindent} \textbf{Output}: The meet $\bigwedge_i \frontier_i$ of the frontiers $\frontier_1, \ldots, \frontier_m$.
\begin{algorithmic}[1]
\Procedure{ComputeMeet}{$\tree, \frontier_1, \ldots, \frontier_m$}
   \State $M \gets \emptyset$ \Comment{The set of \textbf{marked} vertices.}
   \State $U \gets \bigcup_i \frontier_i$ \Comment{The union of all frontiers.}
   \For{$\parentSpace \in U$} \Comment{Iterate through all elements of the union $\bigcup_i \frontier_i$.}
   \State $v \gets \parentSpace$
	\While{$v \neq \rootSpace$} 
		\State $v \gets \textsc{Parent}(v)$ \Comment{Traverse up the tree}
		\If{$v \in M$}
			\State \textbf{break} \Comment{If this vertex was already marked, there is no need to continue up the tree any further.}
		\Else
			\State $M \gets M \cup v$ \Comment{Otherwise, mark this vertex and continue on.}
		\EndIf
	\EndWhile
   \EndFor
   \State \textbf{return} $U \setminus M$
\EndProcedure
\end{algorithmic}
\end{algorithm}

Algorithm (\ref{alg:meet}) outlines the full procedure. A quick inspection of
the algorithm reveals that, at termination, the set of marked vertices $M$
will contain all strict ancestors in $\tree$ of vertices in $\bigcup_i
\frontier_i$. Taking the set difference $\bigcup_i \frontier_i \setminus M$
thus gives the desired result by proposition (\ref{prop:meet}).

\subsection{Updating the projected metrics} \label{sec:update_proj_metrics}

A fundamental problem that we have not yet resolved arises from the fact that 
every time the basis is updated from 
$\initBasisIterLast$ to $\initBasisIter$
during the $r$th compression,
the projected-metric attributes
$\projMetric_\tree(\parentSpace)$ 
at each vertex in $\tree$ become invalid, as they are precomputed to be 
equal to
$(\initBasisIterLast)^T
\projector_\parentSpace \initBasisIterLast$.
To
address this issue, we derive a relationship between the projected metrics
$(\initBasisIter)^T \projector_\parentSpace \initBasisIter$  and $(\initBasisIterLast)^T
\projector_\parentSpace \initBasisIterLast$. Let
$\frontierBefore_1, \ldots, \frontierBefore_{\initBasisSizeBefore}$ denote the
frontiers before the $r$th compression, which are associated with initial
basis 
$\initBasisIterLast$
and refined basis
$\romBasis$, and define the frontier meet as
\begin{equation}
\frontierMeet \equiv \bigwedge_i \frontierBefore_i\,.
\end{equation}
We now establish a useful structural property of frontiers.
\begin{restatable}{proposition}{proponbelowabove} \label{prop:onbelowabove}
If $\frontier$ is a frontier of a $\rootSpace$-refinement tree $\tree \equiv (\treeNodes, \treeEdges)$, then for every $\parentSpace \in \treeNodes$, exactly one of the following is true:
\begin{enumerate}
\item $\parentSpace$ is \textbf{on} the frontier $\frontier$, i.e.,
	$\parentSpace \in \frontier$,
\item $\parentSpace$ is \textbf{below} the frontier $\frontier$, i.e., there
	exists $\childSpace \in \frontier$ such that $\parentSpace \subsetneq
		\childSpace$, or
\item  $\parentSpace$ is \textbf{above} the frontier $\frontier$, i.e., there exists $\childSpace \in \frontier$ such that $\childSpace \subsetneq \parentSpace$.
\end{enumerate} 
\end{restatable}

We now decompose the task of relating the projected metrics
$(\initBasisIter)^T \projector_\parentSpace \initBasisIter$ and
$(\initBasisIterLast)^T \projector_\parentSpace \initBasisIterLast$ into the
three cases of proposition (\ref{prop:onbelowabove}).

\begin{enumerate}
	\item Suppose $\parentSpace$ is \textbf{on} the frontier meet $\frontierMeet$, i.e.
		$\parentSpace \in \frontierMeet$. By proposition (\ref{prop:ancestor}) and
		the fact that $\frontierMeet \preceq \frontierBefore_i$ by definition
		(\ref{def:meet1}), there exists a unique ancestor map from
		$\frontierMeet$ to each of the frontiers $\frontierBefore_i$, which we
		denote by
	\begin{equation}
	\ancestor_i : \frontierMeet \longrightarrow \frontierBefore_i \,.
	\end{equation}
	To understand the structure of the new projected metrics $(\initBasisIter)^T
		\projector_\parentSpace \initBasisIter$, we first consider the matrix
		$\projector_\parentSpace \initBasisIter$. Note that
	\begin{equation}
	\projector_\parentSpace \initBasisIter = \projector_\parentSpace \romBasis
		\romInitBasisIter\,,
	\end{equation}
		where $\romInitBasisIter\in \RR{\bigsqcup_i \frontierBefore_i \times
\nat{\romSize^{(\compTime)}_0}}$ denotes the representation of 
	$\initBasisIter$	in the coordinates of basis $\romBasis$.
		The $i$th column of this matrix is given by
		\begin{equation}\label{eq:projParentDecomp}
	\projector_\parentSpace \initBasisColIter_i = \projector_\parentSpace
	\romBasis \romInitBasisColIter_i = \sum_{\childSpace \in \bigsqcup_j
	\frontierBefore_j } \projector_{\parentSpace} \romBasisCol_{\childSpace}
	\romInitBasisIter_{\childSpace, i} = \sum_{\childSpace \in \bigsqcup_j
	\frontierBefore_j }  \projector_{\parentSpace} \left(\projector_\childSpace
	\initBasisColIterBefore_{\globalancestor(\childSpace)}\right)
	\romInitBasisIter_{\childSpace, i}\,,
\end{equation}
where we recall 
		from Eq.\ (\ref{eq:globalancestor})
		that $\globalancestor : \bigsqcup_i \frontierBefore_i
		\longrightarrow \nat{\initBasisSizeBefore}$, and we have used the fact that
		$\romBasisCol_{\childSpace}$ is the projection of the basis vector
		$\initBasisColIterBefore_{\globalancestor(\childSpace)}$ into the space
		$\childSpace$. We can write Eq.~\eqref{eq:projParentDecomp} equivalently
		as
\begin{equation} \label{eq:projector_new}
\projector_\parentSpace \initBasisColIter_i = \sum_j \sum_{\childSpace \in
	\frontierBefore_j } \left(\projector_{\parentSpace}
	\projector_\childSpace\right)
	\initBasisColIterBefore_{j}
	\romInitBasisIter_{\childSpace, i} \,,
\end{equation}
where we have used $\globalancestor(\childSpace) = j$ for
		$\childSpace\in\frontierBefore_j$.
The advantage of choosing $\parentSpace$ to be on the frontier $\frontierMeet$
		becomes clear in considering the projector product
		$\projector_{\parentSpace} \projector_{\childSpace}$. Because
		$\frontierMeet \preceq \frontierBefore_j$ for all $j$, 
		$\ancestor_j(\parentSpace) \in \frontierBefore_j$ is the unique vertex in
		$\frontierBefore_j$ with the property that $\parentSpace \subset
		\ancestor_j(\parentSpace)$ and $\parentSpace \perp \childSpace$ for all
		other $\childSpace \in \frontierBefore_j$. Therefore, for $\childSpace
		\in \frontierBefore_j$, we have
		\begin{equation}\label{eq:vertexOnMeetresult}
\projector_{\parentSpace} \projector_\childSpace = \begin{cases}
\projector_{\parentSpace} & \childSpace = \ancestor_j(\parentSpace) \\
\zero & \text{otherwise}.
\end{cases}
\end{equation}
Substituting Eq.~\eqref{eq:vertexOnMeetresult} in Eq.~\eqref{eq:projector_new}
yields
\begin{equation}
\projector_\parentSpace \initBasisColIter_i = \sum_j  \projector_{\parentSpace} \initBasisColIterBefore_{j} \romInitBasisIter_{\ancestor_j(\parentSpace), i} \,.
\end{equation}
In matrix form, this becomes
\begin{equation} \label{eq:proj_coeffs}
\projector_\parentSpace \initBasisIter = \projector_\parentSpace \initBasisIterLast \compressionCoeffs{\parentSpace} \,,
\end{equation}
where $\compressionCoeffs{\parentSpace} \in \R{\initRomSizeIterLast \times \initRomSizeIter}$ is the matrix with entries
\begin{equation} \label{eq:conj_matrix}
\compressionCoeffs{\parentSpace}_{j,i} \equiv \romInitBasisIter_{\ancestor_j(\parentSpace), i}  \,.
\end{equation}
Now, we can derive the following relationship between the projected metrics
$(\initBasisIter)^T \projector_\parentSpace \initBasisIter$ and
$(\initBasisIterLast)^T \projector_\parentSpace \initBasisIterLast$:
\begin{equation} \label{eq:update_eq}
\begin{split}
(\initBasisIter)^T \projector_\parentSpace \initBasisIter =
	(\projector_\parentSpace \initBasisIter)^T (\projector_\parentSpace
	\initBasisIter) = (\projector_\parentSpace \initBasisIterLast
	\compressionCoeffs{\parentSpace})^T ( \projector_\parentSpace
	\initBasisIterLast \compressionCoeffs{\parentSpace}) \\
	=
	(\compressionCoeffs{\parentSpace})^T \left[ (\initBasisIterLast)^T
	\projector_\parentSpace \initBasisIterLast\right] \compressionCoeffs{\parentSpace} \,.
\end{split}
\end{equation}
where we have used the idempotent and self-adjoint properties of projectors.
Therefore, we see that, \textbf{for a space $\parentSpace$ on the frontier
meet $\frontierMeet$}, \textit{the projected metrics before and after compression are related by conjugation by the matrix $\compressionCoeffs{\parentSpace}$ in Eq.\ (\ref{eq:conj_matrix}).} This concludes the first case.

\item Suppose $\parentSpace$ is \textbf{below} the frontier meet $\frontierMeet$, i.e., there exists a $\childSpace \in \frontierMeet$ such that $\parentSpace \subsetneq \childSpace$.  In this case, the properties of projectors give us
\begin{equation} \label{eq:projector_sqeeze}
\projector_\parentSpace \projector_\childSpace = \projector_\parentSpace \,.
\end{equation}
Using this property, and the idempotent and self-adjoint nature of projectors,
we have
\begin{equation}\label{eq:projector_sqeeze_two}
\projector_\childSpace^T \projector_\parentSpace \projector_\childSpace = \projector_\parentSpace \,.
\end{equation}
Hence, we can write
\begin{equation} \label{eq:expand_proj_metric}
(\initBasisIter)^T \projector_\parentSpace \initBasisIter = (\projector_\childSpace \initBasisIter)^T \projector_\parentSpace (\projector_\childSpace \initBasisIter) \,.
\end{equation}
Noting that $\childSpace \in \frontierMeet$, we can use results from case
(1). Specifically,  we have from Eq.\ (\ref{eq:proj_coeffs}) that
\begin{equation}\label{eq:expand_proj_metric_two}
\projector_\childSpace \initBasisIter = \projector_\childSpace \initBasisIterLast \compressionCoeffs{\childSpace} \,.
\end{equation}
Substituting Eqs.~\eqref{eq:expand_proj_metric_two} and
\eqref{eq:projector_sqeeze_two} into
Eq.~\ref{eq:expand_proj_metric} yields
\begin{equation}
(\initBasisIter)^T \projector_\parentSpace \initBasisIter = (\projector_\childSpace
	\initBasisIterLast \compressionCoeffs{\childSpace})^T
	\projector_\parentSpace (\projector_\childSpace \initBasisIterLast
	\compressionCoeffs{\childSpace}) = (\compressionCoeffs{\childSpace})^T
	\left[(\initBasisIterLast)^T \projector_\parentSpace
	\initBasisIterLast\right] \compressionCoeffs{\childSpace} \,.
\end{equation}
Therefore, we see that, \textbf{for a space $\parentSpace$ below the frontier
meet
$\frontierMeet$}, \textit{the projected metrics before and after compression
are related by conjugation by the matrix $\compressionCoeffs{\childSpace}$
defined in Eq.\ (\ref{eq:conj_matrix}) for the unique space $\childSpace \in
\frontierMeet$ such that $\parentSpace \subsetneq \childSpace$}. This
concludes the second case.

\item Suppose $\parentSpace$ is \textbf{above} the frontier meet $\frontierMeet$, i.e., there
	exists a $\childSpace \in \frontierMeet$ such that $\childSpace \subsetneq
	\parentSpace$. In this case, there is no easy conjugation expression that relates the projected metrics before and after compression, to our knowledge. However, if the new projected metrics on the frontier $\frontierMeet$ have been computed, then all of the projected metrics for vertices above the frontier $\frontierMeet$ can be computed by recursively applying the property in Eq.\ (\ref{eq:proj_metric_rec}), namely that
\begin{equation}
(\initBasisIter)^T \projector_{\parentSpace} \initBasisIter = \sum_{\childSpace \in \tRefine{\parentSpace}{T}} (\initBasisIter)^T \projector_{\childSpace} \initBasisIter \,.
\end{equation}
The base case for the recursion corresponds to case (1), where $\parentSpace \in
\frontierMeet$. 
\end{enumerate}

These three cases present a practical algorithm for updating the
projected-metric attributes stored at each vertex of the refinement tree after
compression. Roughly, this algorithm comprises four steps:

\begin{enumerate}
\item Compute the frontier meet $\frontierMeet \equiv \bigwedge_i \frontierBefore_i$ using algorithm \ref{alg:meet}.
\item Compute all the projected metrics \textbf{on} the frontier meet
	$\frontierMeet$ via Eqs.~(\ref{eq:conj_matrix}) and (\ref{eq:update_eq}).
		This requires computing a conjugation matrix
		$\compressionCoeffs{\parentSpace}$ for each vertex $\parentSpace \in
		\frontierMeet$ and then conjugating the projected-metric attribute at that
		vertex, i.e.,
\begin{equation}
\projMetric_\tree(\parentSpace) \gets (\compressionCoeffs{\parentSpace})^T \, (\projMetric_\tree(\parentSpace)) \, \compressionCoeffs{\parentSpace} \,.
\end{equation}
\item For all vertices \textbf{above} the frontier meet $\frontierMeet$, we
	recursively compute the projected-metric attribute for a vertex as the sum
		of the projected-metric attributes for that vertex's children, i.e.,
\begin{equation}
\projMetric_\tree(\parentSpace) \gets \sum_{\childSpace \in \tRefine{\parentSpace}{T}} \projMetric_\tree(\childSpace) \,.
\end{equation}
\item For all vertices \textbf{below} the frontier meet $\frontierMeet$, we
	apply the conjugation matrix $\compressionCoeffs{\parentSpace}$ computed at
		their ancestor $\parentSpace \in \frontierMeet$, i.e.,
\begin{equation} \label{eq:recursive}
\projMetric_\tree(\childSpace) \gets (\compressionCoeffs{\parentSpace})^T  \,
	(\projMetric_\tree(\childSpace))\,\compressionCoeffs{\parentSpace}\,.
\end{equation}
\end{enumerate}

Steps (1), (2), and (3) all incur an operation count that depends only on
$\romSize$ (the basis dimension before compression) and $\initRomSizeIterLast$
(the dimension of the basis $\initBasisIterLast$).  On the other hand,
performing step (4) explicitly would require visiting all $O(\stateVecSize)$
vertices beneath the frontier $\frontierMeet$, thereby incurring an
unacceptable computational cost that would preclude significant savings over
simply computing the metric $\romBasis^T \romBasis$ outright. Fortunately, the
conjugation matrices $\compressionCoeffs{\parentSpace}$ that must be applied
to all descendants of $\parentSpace \in \frontierMeet$ in step (4) are
\textit{completely identical}. Therefore, we opt to defer the application of
these matrices to their respective metrics until the projected-metric
attributes are actually required. In the full version of our metric
computation algorithm, the majority of these computations are done
\textit{just-in-time} by storing the state of the refinement tree attributes
implicitly rather than explicitly.

In this spirit of deferring computation until necessary, we make a clear
distinction between two different types of vertices in the refinement tree
$\tree$:

\begin{enumerate}
\item \textbf{Explicit vertices}: these are vertices $\parentSpace$ of $\tree$ such that the attribute $\projMetric_\tree(\parentSpace)$ is current with the correct value of the projected metric. That is,
\begin{equation}
\projMetric_\tree(\parentSpace) = (\initBasisIter)^T \projector_\parentSpace \initBasisIter \,.
\end{equation}
\item \textbf{Implicit vertices}: these are vertices $\parentSpace$ of $\tree$
	such that the attribute $\projMetric_\tree(\parentSpace)$
		must be updated before it has the desired value. That is,
\begin{equation}
\projMetric_\tree(\parentSpace) \neq (\initBasisIter)^T \projector_\parentSpace \initBasisIter
\end{equation}
in general.
However, by virtue of the conjugacy relations derived earlier, obtaining the
		desired value is a matter of conjugating the projected-metric attribute by
		some matrix $\mat{B}$ that can be computed,
\begin{equation}
\mat{B}^T (\projMetric_\tree(\parentSpace)) \mat{B} = (\initBasisIter)^T \projector_\parentSpace \initBasisIter \,,
\end{equation}
and the projected-metric attribute can then be updated in a deferred manner when it is needed,
\begin{equation} \label{eq:proj_metric_conj}
\projMetric_\tree(\parentSpace) \gets \mat{B}^T (\projMetric_\tree(\parentSpace)) \mat{B} \,.
\end{equation}

\end{enumerate}

To aid in this deferred computation, we introduce two new attributes that are
stored at each vertex of the refinement tree $\tree$.

\begin{definition}[explicit flag attribute]
We ascribe to every vertex $\parentSpace$ of the refinement tree $\tree$ an
	\textbf{explicit flag attribute} $\explicitFlag_\tree(\parentSpace)$. The
	flag determines whether or not the vertex is viewed by our algorithm as an
	\textit{explicit} or \textit{implicit} vertex. Initially, all of the flags
	are set to false, i.e.,
\begin{equation}
\explicitFlag_\tree(\parentSpace) \gets \textsc{False} \,,
\end{equation}
with the understanding that \textit{every vertex is explicit before the first compression and the flag is used only after the first compression has occurred}.
\end{definition}

\begin{definition}[deferred conjugation attribute]
We ascribe to some vertices $\parentSpace$ in $\tree$ a \textbf{deferred conjugation attribute} $\defConj_\tree(\parentSpace)$. The understanding is that a deferred conjugation at node $\parentSpace$ must be applied to all of $\parentSpace$'s strict descendants before the projected-metric attributes at those descendants are correct. That is, for an implicit node $\childSpace$ in $\tree$, the necessary conjugation matrix $\mat{B}$ in Eq.\ (\ref{eq:proj_metric_conj}) to correct the projected-metric attribute is given by
	\begin{align} \label{eq:attribute_conj}
	\begin{split}
	\mat{B} = &(\defConj_\tree(\parentSpace_m)) (\defConj_\tree(\parentSpace_{m
	- 1})) \cdots\\
	&(\defConj_\tree(\parentSpace_1)) \,,
	\end{split}
	\end{align}
where $\parentSpace_1, \ldots, \parentSpace_m$ are the vertices with deferred
	conjugations encountered when traversing the tree $\tree$ upwards from the
	parent of $\childSpace$ to the first explicit vertex $\parentSpace_m$
	encountered (both inclusive). The deferred conjugation attributes for all
	vertices are initially set to empty, i.e.,
\begin{equation}
\defConj_\tree(\parentSpace) \gets \emptyset \,.
\end{equation}
\end{definition}

Despite the somewhat complex definition, the idea behind the deferred
conjugation attribute is quite simple: when we compute the conjugation matrix
$\compressionCoeffs{\parentSpace}$ for a space $\parentSpace$ in the meet
$\frontierMeet$, we need to conjugate the projected-metric attributes of
$\parentSpace$ and all of its descendants. While we compute the conjugation of
the projected-metric attribute of $\parentSpace$ explicitly, we defer the
computation of all of the strict descendants of $\parentSpace$, and instead
mark $\parentSpace$ as having a deferred conjugation attribute,
\begin{equation}
\defConj_\tree(\parentSpace) \gets \compressionCoeffs{\parentSpace} \,,
\end{equation}
with the understanding that this conjugation will be applied to all the strict
descendants of $\parentSpace$ on a \textit{just-in-time} basis. To illustrate
this point, suppose that after the first compression, we later want to compute
the projected metric at a descendant $\childSpace$ of $\parentSpace \in
\frontierMeet$. As indicated in Eq.\ (\ref{eq:attribute_conj}), we would
traverse the tree upward from $\childSpace$ until we reached $\parentSpace$,
find the deferred conjugation $\compressionCoeffs{\parentSpace}$ at
$\parentSpace$, apply it to the projected-metric attribute of $\childSpace$ as
necessary, set the vertex $\childSpace$ to explicit, and set the deferred
conjugation at $\childSpace$ to $\compressionCoeffs{\parentSpace}$, as all of
the descendants of $\parentSpace$ must still be updated with the deferred
conjugation $\compressionCoeffs{\parentSpace}$ before they are correct. In
this way, the deferred conjugations `trickle down' the tree $\tree$ as 
needed. Thus, this approach avoids any operations whose complexity explicitly
dependence on the FOM dimension $\stateVecSize$.

Of course, while this illustrative example is rather simple, there are still nontrivial implementation details we must address. However, the baseline invariant we maintain is that if a vertex $\parentSpace$ is explicit, then the projected-metric attribute at $\parentSpace$ is correct, i.e.,
\begin{equation}
\explicitFlag_\tree(\parentSpace) = \textsc{True} \qquad \text{implies} \qquad \projMetric_\tree(\parentSpace) = (\initBasisIter)^T \projector_\parentSpace \initBasisIter \,,
\end{equation}
where $\initBasisIter$ is the reduced basis after the most recent compression.
And conversely, if the vertex $\parentSpace$ is not explicit, then the
projected-metric attribute at $\parentSpace$ can corrected by
accumulating all necessary deferred conjugations above it in the tree,
i.e.,
\begin{equation}
\explicitFlag_\tree(\parentSpace) = \textsc{False} \qquad \text{implies} \qquad \mat{B}^T (\projMetric_\tree(\parentSpace))\mat{B}  = (\initBasisIter)^T \projector_\parentSpace \initBasisIter \,,
\end{equation}
where $\mat{B}$ is the product of all deferred conjugations above the
vertex $\parentSpace$, given by Eq.\ (\ref{eq:attribute_conj}). 

\subsection{Basis-compression algorithm}\label{sec:basiscompressionalg}

We now assemble the preliminaries prepared in the previous few sections.
First, we present an algorithm to actually compute projected metrics from the
tree.  Algorithm \ref{alg:get_proj_metric}, which follows the schematic
outline given in the previous section, provides the computational details of
this process.

\begin{algorithm}[h]
\caption{Get Projected Metric (Basic)} \label{alg:get_proj_metric}
\hspace*{\algorithmicindent} \textbf{Input}: The $\mathbb{R}^n$-refinement tree $\tree$, and the vertex $\parentSpace$ at which to compute the projected metric. \\
\hspace*{\algorithmicindent} \textbf{Output}: The projected metric $(\initBasisIter)^T \projector_\parentSpace \initBasisIter$ at $\parentSpace$.
\begin{algorithmic}[1]
\Procedure{GetProjectedMetric}{$\tree, \parentSpace$}
   \State $\metric \gets \projMetric_\tree(\parentSpace)$
   \If{\textbf{not} $\explicitFlag_\tree(\parentSpace)$} \Comment{If the vertex is implicit, we traverse upwards and apply all deferred conjugations.}
   	\State $\childSpace = \parentSpace$
   	\State $\mat{B} \gets \mat{I}$ \Comment{$\mat{I}$ is the identity with the same dimensions as $\metric$}
   	\While{\textsc{True}} \Comment{Apply all the deferred conjugations above the vertex $\parentSpace$ to $\metric$.}
   		\State $\childSpace \gets \textsc{Parent}(\childSpace)$
   		\If{$\defConj_\tree(\childSpace) \neq \emptyset$}
   			\State $\mat{B} \gets \mat{B} (\defConj_\tree(\childSpace))$ \Comment{If there is a deferred conjugation stored at this vertex, apply it.}
   		\EndIf
			\If{$\explicitFlag_\tree(\childSpace)$} \Comment{We stop when we have
			reached an explicit vertex.}
   			\State $\metric \gets \mat{B}^T \metric \mat{B}$ \Comment{Apply the accumulated deferred conjugations to $\metric$. }
   			\State \textbf{break}
   		\EndIf
   	\EndWhile
   \EndIf
   \State \textbf{return} $\metric$
\EndProcedure
\end{algorithmic}
\end{algorithm}

With the ability to compute metrics in this fashion, we now follow the
schematic details provided at the end of section \ref{sec:update_proj_metrics}
to update the projected metrics of the tree $\tree$ after a compression has
been performed. Algorithms \ref{alg:update_proj_metric} and
\ref{alg:compute_proj_metric_recursive} provide the associated algorithms.

\begin{algorithm}[h]
\caption{Update Projected Metrics (Basic)} \label{alg:update_proj_metric}
\hspace*{\algorithmicindent} \textbf{Input}: The $\mathbb{R}^n$-refinement tree $\tree$ and the frontiers $\frontierBefore_i$ used to perform compression, and the coefficients $\romInitBasisIter$. \\
\hspace*{\algorithmicindent} \textbf{State Effects}:  Updates the internal
	representation of the projected metrics of $\tree$ so that they correspond
	to the desired basis with coordinate coefficients $\romInitBasisIter$.
\begin{algorithmic}[1]
\Procedure{UpdateProjectedMetrics}{$\tree, \frontierBefore_1, \ldots,\frontierBefore_{\initRomSizeIterLast},\romInitBasisIter$}
   \State $\frontierMeet \gets \textsc{ComputeMeet}(\frontierBefore_1, \ldots,\frontierBefore_{\initRomSizeIterLast})$
	 \For {$\parentSpace \in \frontierMeet$}\Comment{Iterate through all
	 vertices on the frontier meet.}
   	\State $\projMetric_\tree(\parentSpace) \gets
		\textsc{GetProjectedMetric}(\tree, \parentSpace)$  \Comment{Compute correct
		projected metric.}
		\State $\compressionCoeffs{\parentSpace} \gets \mat{0} \in
		\R{\initRomSizeIterLast \times \initRomSizeIter}$ 
		\For {$(i, j) \in \nat{\initRomSizeIterLast} \times
		\nat{\initRomSizeIter}$} \Comment{Compute the deferred conjugation using
		Eq.\ (\ref{eq:conj_matrix}).}
   		\State $\compressionCoeffs{\parentSpace}_{i,j} \gets \romInitBasisIter_{\ancestor_i(\parentSpace), j} $
   	\EndFor
 	\If{$\defConj_\tree(\parentSpace) \neq \emptyset$} \Comment{Apply the deferred conjugation.}
 		\State $\defConj_\tree(\parentSpace) \gets \defConj_\tree(\parentSpace) \, \compressionCoeffs{\parentSpace}$
 	\Else
 		\State $\defConj_\tree(\parentSpace) \gets \compressionCoeffs{\parentSpace}$
 	\EndIf
 	\State $\explicitFlag_\tree(\parentSpace) \gets \textsc{True}$
		\Comment{Now, set this vertices on the frontier meet to explicit.}
   \EndFor
   
	\For {$\parentSpace \in \{ \childSpace \mid \explicitFlag_\tree(\childSpace)
	= \textsc{True} \} \setminus \frontierMeet$}\Comment{Reset all other vertices except those on $\frontierMeet$ to implicit.}
		\State $\explicitFlag_\tree(\childSpace) \gets \textsc{False}$
	\EndFor   
   
   \State $\rootSpace \gets \textsc{Root}(\tree)$
   \State $\textsc{ComputeProjectedMetricsRecursive}(\tree, \rootSpace)$ 
\EndProcedure
\end{algorithmic}
\end{algorithm}

\begin{algorithm}[h]
\caption{Recursive Projected Metric Computation (Basic)} \label{alg:compute_proj_metric_recursive}
\hspace*{\algorithmicindent} \textbf{Input}: The $\mathbb{R}^n$-refinement tree $\tree$ and the vertex $\parentSpace$ at which to begin the recursive computation.  \\
\hspace*{\algorithmicindent} \textbf{Output}: The correct projected metric $\metric$ at $\parentSpace$. \\
\hspace*{\algorithmicindent} \textbf{State Effects}: Computes the projected metric for $\parentSpace$ and its descendants above the frontier meet via the procedure in Eq.\ (\ref{eq:recursive}). Sets $\parentSpace$ and all of its descendants above the frontier meet to explicit and removes any deferred conjugations.
\begin{algorithmic}[1]
\Procedure{ComputeProjectedMetricsRecursive}{$\tree, \parentSpace$}
	\If{\textbf{not} $\explicitFlag_\tree(\parentSpace)$} \Comment{If this
	vertex is implicit, recursively sum the projected metrics of its descendants.}
		\State $\projMetric_\tree(\parentSpace) \gets \sum_{\childSpace \in \textsc{Children}(\parentSpace)} \textsc{ComputeProjectedMetricsRecursive}(\tree, \childSpace) $
		\State $\explicitFlag_\tree(\parentSpace) \gets \textsc{True}$
		\State $\defConj_\tree(\parentSpace) \gets \emptyset$ 
	\EndIf
	\State \Return $\projMetric_\tree(\parentSpace)$
\EndProcedure
\end{algorithmic}
\end{algorithm}

Finally, algorithm \ref{alg:basis_compression} combines all the required
components into the final online basis-compression algorithm.

\begin{algorithm}[h]
\caption{Basis Compression } \label{alg:basis_compression}
\hspace*{\algorithmicindent} \textbf{Input}: The $\mathbb{R}^n$-refinement tree $\tree$, frontiers $\frontierBefore_i$ used to perform compression, the compression snapshot data $\romSnaps$ in the ROM basis, and the compression threshold $\epsilon$. \\
\hspace*{\algorithmicindent} \textbf{Output}: The new frontiers $\frontierIter_i$, the coefficients $\romInitBasisIter$ of the compressed ROM basis in the current ROM basis. \\
\hspace*{\algorithmicindent} \textbf{State Effects}: Updates the internal state of the tree $\tree$ to correspond to the projected metrics of the new compressed basis.
\begin{algorithmic}[1]
\Procedure{CompressBasis}{$\tree, \frontierBefore_1, \ldots, \frontierBefore_{\initRomSizeIterLast}, \romSnaps, \epsilon$}
\State $\frontier \gets \bigsqcup_i \frontierBefore_i$
	\State\label{step:metricPOD} $\romCompVecs \gets \textsc{MetricCorrectedPOD}(\tree, \frontier, \romSnaps, \epsilon)$ \Comment{Compute metric corrected POD to get the enrichment vectors $\romCompVecs$}
\State $\romInitBasisIter \gets \left[\begin{array}{cc} \romInitBasisZ & \romCompVecs\end{array}\right]$ \Comment{Append the enrichment vectors $\romCompVecs$ to the original ROM vectors to get compressed basis.} 
	\State\label{step:UpdateProjectedMetrics} $\textsc{UpdateProjectedMetrics}(\tree, \frontierBefore_1, \ldots, \frontierBefore_{\initRomSizeIterLast},\romInitBasisIter)$ \Comment{Update the projected metrics for this new basis}
	\For{$i \in \nat{\initRomSizeIter}$} \Comment{All new frontiers are initialized to the root of $\tree$. Note $\initRomSizeIter$ is the column count of $\romInitBasisIter$.}
	\State $\frontierIter_i \gets \{ \textsc{Root}(\tree) \}$
\EndFor
\State \Return $(\frontierIter_1, \ldots, \frontierIter_{\initRomSizeIter}, \romInitBasisIter)$
\EndProcedure
\end{algorithmic}
\end{algorithm}

\subsubsection{A note on numerical errors and diagnosis}

Due to the limits of finite-precision arithmetic, numerical errors can
accumulate over the course of several basis compressions. In general, the
severity of these numerical errors depends on the choice of singular-value
threshold $\epsilon$ in algorithm (\ref{alg:pod}). However, as long as the
singular-value threshold $\epsilon$ is not too small (e.g., $\epsilon >
10^{-3}$), we have found that the magnitude of such numerical errors is usually
negligible. Moreover, because vectors $\compVecs$ comprise a somewhat
heuristic choice of enriching the original ROM $\initBasisZ$---and because we
never explicitly rely on orthogonality of $\compVecs$---even substantial
numerical errors in this process do not have substantial impact on the
fundamental objective of the overall method.

However, if numerical error do become large, the proposed method exposes a
natural indicator for the degree of numerical error that has accumulated in
the tree's internal representation of the projected metrics, namely, how close
the projected-metric attribute of the root of $\tree$ is to the identity
matrix.  Because the compressed basis $\romInitBasisIter$ is always orthogonal
in exact arithmetic, and the root node is always explicit, we should always
have
\begin{equation}
\projMetric_\tree(\textsc{Root}(\tree)) = \mat{I} \in \R{\initRomSizeIter\times \initRomSizeIter} \,.
\end{equation}
The extent to which this equality is violated provides a useful indicator for
the magnitude of accumulated numerical errors. If the stored value deviates
substantially from the identity, it is always possible to simply reset the
compressed basis $\romInitBasisIter$ to the original ROM basis $\initBasisZ$
and reset the internal state of the tree $\tree$ to its original
configuration.

\section{Discussion of algorithm complexity}\label{sec:complexity}

We conclude by examining the space and time complexity of the proposed
algorithm.

\begin{enumerate}
\item \textit{Offline Precomputation}: The offline stage requires the
	following precomputation steps.
\begin{enumerate}
	\item \textit{Refinement-tree construction}. We give the complexity of
		the data-driven tree construction algorithm described in section
		\ref{sec:tree_construction}. First,
		the algorithm transforms the snapshots $\snaps \in \R{\stateVecSize \times \snapCount}$ to the correct basis via 
		$\leafBasis^T\snaps$, where
		$\leafBasis\equiv[\leafBasisCol_1\,\cdots\,\leafBasisCol_\stateVecSize]$ denotes the orthogonal leaf
		basis. Naively,
		this incurs $O(\snapCount \cdot T(\stateVecSize))$ operations, where
		$T(\stateVecSize)$ denotes the cost of transforming a single
		$\stateVecSize$-vector into the desired
		leaf basis. In the worst case, this is $T(\stateVecSize) =
		O(\stateVecSize^2)$. However, there are many cases were $T(\stateVecSize)$
		is smaller. For example, 
		$T(\stateVecSize) = O(\stateVecSize
		\log \stateVecSize)$
		when the leaf basis comprises
		Fourier modes,
		$T(\stateVecSize) =
		O(\stateVecSize)$ when leaf-basis transformation is computed via a 
		fast wavelet transform, and 
		$T(\stateVecSize) = 0$
		for the Kronecker leaf basis proposed in the original $h$-refinement paper
		\cite{carlberg2015adaptive}, as transformation is
		unnecessary in this case.

Next, the algorithm preprocesses the transformed snapshots for clustering
		via Eq.~\eqref{eq:kmeansPre}, which incurs
		$O(\snapCount \stateVecSize)$ operations. Finally, the algorithm performs recursive
		$k$-means clustering to construct the refinement tree. If the
		resulting
		tree is roughly balanced, as seems to be the case in
		practice,  the complexity of this dynamic-programming procedure is
		$O(K(\snapCount, \stateVecSize) \cdot \log \stateVecSize)$, where
		$K(\snapCount, \stateVecSize)$ denotes the complexity of a $k$-means operation
		on $\stateVecSize$ vectors of length $\snapCount$. Thus, in total, the
		operation count is bounded by
\begin{equation}
\text{Operation Count}: \qquad O(\snapCount \cdot T(\stateVecSize)) + O(\snapCount\stateVecSize) + O(K(\snapCount, \stateVecSize) \cdot \log \stateVecSize) \,.
\end{equation}
For storage, we require $O(\snapCount\stateVecSize)$ for the snapshot
matrix, and $O(\stateVecSize)$ for the tree (without any additional
attributes), because the number of nodes in a tree is at most twice the
		number of leaves. Finally, we require $B(\stateVecSize)$ storage for the
leaf basis itself. The amount of storage needed for the leaf basis can
vary, and  $B(\stateVecSize) = O(\stateVecSize^2)$ in the worst case; however,
in many cases, the basis $B(\stateVecSize)$ need not
be explicitly stored at all (e.g., Fourier basis), or the basis
is sparse. Thus, the total required storage is 
\begin{equation}
\text{Storage}: \qquad O(\snapCount \stateVecSize) + B(\stateVecSize) \,.
\end{equation}
\item \textit{Precomputing the projected metrics}: the online-efficient
	compression algorithm requires precomputing the projected metrics
		$(\initBasisZ)^T \projector_\parentSpace \initBasisZ$ for all vertices
		$\parentSpace$ in the refinement tree $\tree$.  We accomplish this by
		first computing this projected metric for each leaf node
		$\leafSpace_i\equiv\vecspan(\{\leafBasisCol_i\}) \in \treeLeaves$,
		$i=1,\ldots,\stateVecSize$. In this case, we have
 \begin{equation} 
(\initBasisZ)^T \projector_\parentSpace \initBasisZ
	 =(\initBasisZ)^T\leafBasisCol_i\leafBasisCol_i^T\initBasisZ.
	  \end{equation} 
This identity, combined with the summation property from Eq.\
(\ref{eq:proj_metric_rec}) implies that, after computing $\leafBasis^T
\initBasisZ$, which incurs $O(\initRomSize \cdot T(\stateVecSize) )$
operations, one can
compute the projected metric for the leaves in $O(\stateVecSize
\initRomSize^2)$, as computing each
$\initRomSize \times \initRomSize$ matrix amounts to
computing the outer product of an $\initRomSize$-vector with itself. Then, we can compute the 
projected metrics at
the remaining tree vertices by recursively applying Eq.\
(\ref{eq:proj_metric_rec}), which incurs $O(\stateVecSize \initRomSize^2)$
operations. Likewise, these metrics consume $O(\stateVecSize \initRomSize^2)$
storage, yielding
\begin{align}
\text{Operation Count}&: \qquad O(\stateVecSize \initRomSize^2) + O(\initRomSize \cdot T(\stateVecSize)) \,, \\
\text{Storage}&: \qquad O(\stateVecSize \initRomSize^2)  \,. \\
\end{align}
\item \textit{Precomputing Vector Sieves}: 
	The proposed method requires the ability to quickly compute vector sieves;
	this occurs implicitly in step \ref{step:sieve} of algorithm
	\ref{alg:refine_proc}.
	The complexity of doing so is heavily dependent on the
	selected leaf
	basis
		$\leafBasis$. In the best case, the leaf basis
		$\leafBasis$ is sparse, with the extreme scenario corresponding to $\leafBasis =
		\mat{I}$, i.e., when the leaf basis corresponds to the standard Kronecker basis as in
		the original $h$-refinement paper
		\cite{carlberg2015adaptive}. In
		this scenario, computing sieves is straightforward: if any vertex $\parentSpace$ is
		spanned by the Kronecker vectors $\ve{e}_{i_1} ,\ldots, \ve{e}_{i_k}$, then
		projecting a vector $\ve{\phi}$ into $\parentSpace$ amounts to setting all
		of the entries of $\ve{\phi}$ to $0$ except for entries $i_1, \ldots,
		i_k$. As such, this choice for the leaf basis does not
		require any specialized strategy or additional storage. However, in the worst case (e.g., a Fourier
		basis), each of the leaf basis vectors
		is dense. Moreover, due to the Shannon Sampling Theorem, it
		is unlikely that one can accurately compute projections onto
		vector spaces spanned by high-frequency Fourier modes without observing
		most (if not all) of the entries of the vector being projected; however, this
		incurs an online cost of $\Omega(n)$. 
		
		Fortunately, this online cost can be
		avoided with offline precomputation and storage. In particular, we can
		precompute all the projections of the original basis $\initBasisZ$ into
		all of the possible vector spaces $\parentSpace$ in the tree.
		Unfortunately, this consumes $O(p_0 n^2)$ storage. However, only the
		sieves associated with tree vertices \textit{above} the current global
		frontier must be stored in fast memory. Thus, the remaining projections
		can be offloaded and retrieved as needed.
		
		One can also trade storage for online computation by computing only a
		suitable subset of the projections of $\initBasisZ$ onto the vector spaces
		$\parentSpace$ and use the additive property of the refinement tree
		$\tree$ to sum these precomputed projections appropriately to achieve any
		desired projection. Alternatively, one can switch to directly computing
		the projections when a certain depth in the tree has been reached. Still,
		we acknowledge that neither of these options may be ideal
		in practice. As is, the ability of this algorithm to efficiently handle
		general, globally supported leaf bases $\leafBasis$ remains an area for possible
		improvement. For this reason, we encourage the use of locally supported
		leaf bases $\leafBasis$ when possible, as locally supported leaf bases do
		not have this storage problem when the tree $\tree$ is constructed to encourage sparsity in the refined basis vectors.
\end{enumerate}
\item \textit{Online basis refinement}: Online basis refinement is composed of
	two steps associated with steps \ref{step:computeErrorIndicators} and
	\ref{step:refineFrontiers} in algorithm \ref{alg:refine_proc}. First, the
	algorithm computes error indicators via
		algorithm \ref{alg:err}, then the algorithm uses these error indicators to refine the
		current frontier via algorithm \ref{alg:refine_frontiers}. In
		algorithm \ref{alg:err}, the expensive operations correspond to the computation of the coarse
		adjoint $\adjointC$ in step \ref{step:coarseAdjoint} and the computation of the error indicators
		$\errorIndF$ in step \ref{step:errorIndicators}. 
		These steps require computing 
		(1) $(\romBasisC)^T \frac{\partial \gResidual}{\partial \stateVec} (\romBasisC
		\romStateVecC)^T \romBasisC $, which incurs
		$O(\stateVecSize\romSize^2)$ operations due to the sparsity of the
		Jacobian, (2) $(\romBasisC)^T \frac{\partial \interestFunc}{\partial
		\stateVec} \left(\romBasisC \romStateVecC\right)^T$, which incurs
		$O(\stateVecSize\romSize)$ operations, and (3)
		$\left|\left[\prolongadjoint\right]_\childSpace
		\left(\romBasisCol_\childSpace^h\right)^T \gResidual\left(\romBasisC
		\romStateVecC\right)\right|$, which incurs $O(\stateVecSize\romSize)$
		operations.
		The remainder of the
		computation can be performed in $O(\romSize^3)$ and is dominated by the
		matrix solve for the coarse adjoint $\adjointC$. Of course, any operations
		whose complexity scales with $\stateVecSize$ precludes online efficiency.
		While we do not resolve this bottleneck in this paper, we
		point out that this $\stateVecSize$-dependence can be eliminated using
		a hyper-reduction technique such as collocation
		\cite{astrid2008missing,legresley2006application}, gappy POD
		\cite{sirovichOrigGappy,bos2004als,astrid2008missing,CarlbergGappy,carlberg2013gnat},
		or empirical interpolation 
		\cite{barrault2004eim,chaturantabut2010nonlinear,galbally2009non,drohmannEOI}.
		Such techniques approximate the full-order
		residual $\gResidual$ via $\hat{\gResidual}(\cdot) \equiv \mat{W}
		\, \mat{P} \, \gResidual(\cdot)$, where $\mat{W} \in \R{\stateVecSize
		\times s}$ and $\mat{P} \in
		\{0, 1\}^{s \times \stateVecSize}$ comprises $s(\ll\stateVecSize)$ selected
		rows of the identity matrix. Subsequently, evaluating the projected residual $\romBasis^T
		\hat{\gResidual}$ incurs only $O(\romSize s)$ operations, assuming that
		$\romBasis^T \mat{W} \in \R{\romSize \times s}$ is
		precomputed. The precise
		integration of hyper-reduction with the proposed method consitutes an
		important direction for further work. Finally, we note that algorithm
		\ref{alg:refine_frontiers} does not require any operations that depend on
		$\stateVecSize$, as it only examines the tree $\tree$ above the
		fully-refined global frontier $\refine{(\frontierglobalC)}$. Moreover, all
		operations, excluding the sub-routine call to algorithm \ref{alg:err} and
		the sieve of the original basis $\initBasis$ through $\frontierglobalF$ can be
		performed in $O(\romSize)$. For the sieve of the basis $\initBasis$
		through $\frontierglobalF$ we refer the reader to the earlier discussion of
		precomputation of vector sieves. In principle, it can be achieved in
		$O(\romSize \initRomSize)$ if the leaf basis is sufficiently sparse, or if
		enough precomputation is done offline.
\item \textit{Online basis compression}: To facilitate the analysis, we divide
	this up into multiple steps.
\begin{enumerate}
	\item \textit{Computing the metric} $\romBasis^T \romBasis$ (step \ref{step:computeMetric} of
		algorithm \ref{alg:pod}): each entry of the
	metric requires a lookup in the projected-metric attributes of the tree
		$\tree$. Finding the correct vertex at which to perform a lookup requires
		finding the common ancestor of two nodes. Assuming the tree is relatively
		balanced, this can be done in $O(\log \stateVecSize)$ time. If the
		projected-metric attribute is stored implicitly, then it is difficult to
		place an exact operation count on retrieving the projected metric stored
		at the common ancestor, as it depends on the number of deferred
		conjugations above the common ancestor. As is, the computation could be
		potentially expensive, fortunately, there is a method to accelerate this
		computation using a path compression technique, which reuses computation
		performed across multiple lookups for projected metrics\footnote{Such a technique works by setting all encountered vertices to explicit when traversing upwards through the tree in algorithm \ref{alg:get_proj_metric}, and setting deferred conjugations appropriately on these newly explicit vertices. This allows computation to be reused over multiple calls to algorithm \ref{alg:get_proj_metric}, greatly reducing the run-time of this component of the algorithm. However, some extra bookkeeping is required to ensure that one does not apply a given deferred conjugation twice to a single vertex. }. One can also
		maintain $\romBasis^T \romBasis$ every time a basis refinement is
		performed, and update the matrix accordingly.
	\item \textit{Compressing the basis} (steps \ref{step:factorize} and \ref{step:SVD} of
		algorithm \ref{alg:pod}): Compressing the basis once $\romBasis^T
		\romBasis$ requires factorizing the metric $\romBasis^T \romBasis \in
		\R{\romSize \times \romSize}$. This can be accomplished via symmetric
		eigenvalue decomposition, SVD, or a sparse
		Cholesky decomposition, for example. In the worst case, this decomposition
		incurs $O(\romSize^3)$ operations.
		Afterwords, we must perform a POD on the matrix $\tilde{\mat{Y}}_{\perp} \in
		\R{\romSize \times \compSnaps}$, which incurs $O(\romSize^2 \compSnaps)$ or
		$O(\compSnaps^2 \romSize)$ operations if $\compSnaps$ or $\romSize$ is
		larger, respectively. The matrix multiplications performed afterward are dominated by
		the above operations.
	\item \textit{Updating the internal state of the refinement tree} (step \ref{step:UpdateProjectedMetrics} of
		algorithm \ref{alg:basis_compression}): After the
	basis has been compressed, we must update the projected-metric attributes on
		the refinement tree. This results in computations performed on vertices
		residing \textit{above} 
		frontier meet $\bigwedge_i \frontier_i$. Because the frontier meet
		is characterized by $O(\romSize)$ vertices, the number of vertices this computation affects
		is also $O(\romSize)$. Each vertex involves a matrix conjugation, which
		incurs $O(\initRomSize^2 \romSize)$, yielding a total of $O(\initRomSize^2
		\romSize^2)$ operations. 
\end{enumerate}
\end{enumerate}

\section{Numerical experiments}\label{sec:numericalExp}

We now assess the performance of the proposed method on two benchmark problems
in the context of a Galerkin reduced-order model characterized by a test basis
of the form \eqref{test_trial} with $\testTrialTrans = \identity$.

\subsection{FitzHugh--Nagumo equations} \label{sec:fitzhugh}

The FitzHugh--Nagumo equations model the activation and deactivation dynamics
of a spiking neuron. It is a prototypical example of an excitable dynamical
system and has played an important role in mathematical neuroscience. The
system models two variables, the voltage $\voltage(x, t)$ and the recovery of
voltage $\recVoltage(x, t)$, on the domain $\Omega\times \mathbb T\equiv [0,
L]\times [0,8]$. The governing system of partial differential equations is
given by
\begin{align}
\varepsilon \voltage_t(x, t) &= \varepsilon^2 \voltage_{xx}(x, t) + f(\voltage(x, t)) - \recVoltage(x, t) + c \,, \\
\recVoltage_t(x, t) &= b\voltage(x, t) - \gamma \recVoltage(x, t) + c \,,
\end{align}
with 
$
f:\voltage \mapsto \voltage(\voltage - 0.1)(1 - \voltage) \,.
$
We consider the boundary conditions
\begin{alignat}{4}
&\voltage(x, 0) = 0, \qquad &&\recVoltage(x, 0) = 0, \qquad &&&x \in \Omega \,, \\
&\voltage_x(0, t) = -i_0(t), \qquad &&\voltage_x(L, t) = 0, \qquad &&&t \geq 0 \,
\end{alignat}
with parameter values
$L = 1$, $\varepsilon = 0.015$, $b = 0.5$,
$c = 0.05$, and $\gamma = 2$.
The stimulus $i_0(t)$ corresponds to 
\begin{equation}
i_0(t) \equiv 50\,000 \, t^3 \exp(-15t)\,.
\end{equation}

To discretize the governing equations, we employ a finite-difference scheme
with $512$ points, yielding a state vector of dimension $\stateVecSize = 1024$
given by
\begin{equation} \label{eq:fhn_statevec}
\stateVec = \left[\begin{array}{cccccc} \voltage_1 & \cdots & \voltage_{\stateVecSize / 2} & \recVoltage_1 & \cdots & \recVoltage_{\stateVecSize / 2} \end{array}\right]^T \,,
\end{equation}
where $\voltage_i$ and $\recVoltage_i$ denote the computed values of
$\voltage$ and $\recVoltage$, respectively, at grid point $i$. To numerically
solve the resulting semidiscretized system,
we use the backward-Euler scheme
and a uniform time step of $\Delta t = 0.008$. We construct
the initial basis $\initBasisZ$ by performing a POD of the first $\compSnaps =
100$ time snapshots $\treeSnaps \in \R{\stateVecSize \times \compSnaps}$
(corresponding to the time interval $[0,0.8]$), and taking the first
$\initRomSize=3$ POD vectors. 

To demonstrate the merits of the proposed approach, we now show that the
ability of the method to enable general basis-refinement mechanisms and to
adaptively compress the basis can yield substantial performance improvements
over the original $h$-refinement method \cite{carlberg2015adaptive} For this
comparison, we construct two refinement trees $T_\text{DCT}$ and $T_\text{K}$,
the first of which employs the discrete cosine transform (DCT) basis as a leaf
basis, and the latter of which uses the standard Kronecker basis as a leaf
basis. Without basis compression, the latter choice is equivalent to the
original $h$-refinement method, which provides a baseline for performance
comparison.

To generate the trees $\tree_\text{DCT}$ and $\tree_\text{K}$ we first
separate
the degrees of freedom into two disjoint sets
$\dofVoltage\defeq\{1,\ldots,\stateVecSize/2\}$ and
$\dofRecVoltage\defeq\{\stateVecSize/2+1,\ldots,\stateVecSize\}$, which
correspond to the degrees of freedom associated with the 
voltage and the recovery of voltage,
respectively. 
We take the leaf bases $\leafBasis_\text{DCT}$ and
$\leafBasis_\text{K}$ of $\tree_\text{DCT}$ and $\tree_\text{K}$ respectively
to be
\begin{equation}
\begin{split}
	\leafBasis_\text{DCT} &\equiv \leafBasis_\text{DCT}^{(\voltage)} \oplus \leafBasis_\text{DCT}^{(\recVoltage)} \equiv \mat{M}^{(\stateVecSize / 2)}_\text{DCT} \oplus \mat{M}_\text{DCT}^{(\stateVecSize / 2)} \,,\\
	\leafBasis_\text{K} &\equiv \leafBasis_\text{K}^{(\voltage)} \oplus \leafBasis_\text{K}^{(\recVoltage)} \equiv \mat{I}^{(\stateVecSize /2)} \oplus \mat{I}^{(\stateVecSize /2)} = \mat{I}^{(\stateVecSize)}\,,
\end{split}
\end{equation}
where $\mat{M}^{(\stateVecSize / 2)}_\text{DCT}$ is the $\stateVecSize / 2 \times \stateVecSize / 2$ Discrete Cosine Transform II matrix. We take the direct sum of two such DCT-II matrices because the first $\mat{M}^{(\stateVecSize / 2)}_\text{DCT}$ corresponds to a DCT-II basis on the degrees of freedom $\dofVoltage$ and the second $\mat{M}^{(\stateVecSize / 2)}_\text{DCT}$ corresponds to a DCT-II basis on the degrees of freedom $\dofRecVoltage$.

To construct the $\tree_\text{DCT}$ and $\tree_\text{K}$ from these leaf
bases, we employ the procedure outlined in section \ref{sec:tree_construction}
with one small modification. In particular, because the state space
$\R{\stateVecSize}$ can be decomposed naturally into \begin{equation}
\label{eq:top_split} \R{\stateVecSize} = \parentSpace^{(\voltage)} +
\parentSpace^{(\recVoltage)} \cong \R{\stateVecSize / 2} \oplus
\R{\stateVecSize / 2} \,, \end{equation} where $\parentSpace^{(\voltage)}
\cong \R{\stateVecSize / 2}$ is the vector space associated with the degrees
of freedom in $\dofVoltage$ and $\parentSpace^{(\recVoltage)} \cong
\R{\stateVecSize / 2}$ is the vector space associated with the degrees of
freedom in $\dofRecVoltage$. Because it is natural to decouple the degrees of
freedom $\dofVoltage$ from the degrees of freedom $\dofRecVoltage$, as they
associated with different variables, we enforce that top-most decomposition in
the trees $\tree_\text{DCT}$ and $\tree_\text{K}$ is precisely the one in Eq.\
(\ref{eq:top_split}). This means that the first sieve performed on any vector
$\initBasisCol_i$ will always be to decompose it into two vectors
\begin{equation} \initBasisCol_i = \initBasisCol_i^{\voltage} +
	\initBasisCol_i^{\recVoltage} \,, \end{equation} where
	$\initBasisCol_i^{\voltage}$ and $\initBasisCol_i^{\recVoltage}$ are
	supported on $\dofVoltage$ and $\dofRecVoltage$, respectively, and are equal
	to $\initBasisCol_i$ on their respective supports. 

We construct the remainder of the trees $\tree_\text{DCT}$ and
$\tree_\text{K}$ beyond the first level by constructing
trees $\tree_\text{DCT}^{(\voltage)}$, $ \tree_\text{DCT}^{(\recVoltage)}$ and
$\tree_\text{K}^{(\voltage)}$, $ \tree_\text{K}^{(\recVoltage)}$ for
decompositions of the spaces 
$\parentSpace^{(\voltage)}$ and $\parentSpace^{(\recVoltage)}$ via the
data-driven algorithm in section \ref{sec:tree_construction}, where the input
snapshot data $\treeSnaps$ corresponds to the same
data used to generate
the basis $\initBasisZ$. To be precise, we execute the algorithm in section
\ref{sec:tree_construction} with data $\treeSnaps_{\dofVoltage, :}$ and leaf
basis $\leafBasis_\text{DCT}^{(\voltage)}$ and $\leafBasis_{T}^{(\voltage)}$
to generate $\tree^{(\voltage)}_\text{DCT}$ and $\tree_\text{K}^{(\voltage)}$,
respectively, and with data $\treeSnaps_{\dofRecVoltage, :}$ and leaf
basis $\leafBasis_\text{DCT}^{(\recVoltage)}$ and
$\leafBasis_{T}^{(\recVoltage)}$ to generate
$\tree^{(\recVoltage)}_\text{DCT}$ and $\tree_\text{K}^{(\recVoltage)}$,
respectively. We then construct $\tree_\text{DCT}$ and $\tree_\text{K}$ by
grafting $\tree_\text{DCT}^{(\voltage)}$, $\tree_\text{DCT}^{(\recVoltage)}$ and
$\tree_\text{DCT}^{(\voltage)}$, $\tree_\text{DCT}^{(\recVoltage)}$,
respectively, to separate root nodes representing $\Rn$.

Fig.\ \ref{fig:pareto_fhn} reports performance Pareto fronts that enable a
fair comparison between four variants of the proposed method arising from two
different leaf bases $\leafBasis_\text{DCT}$ and $\leafBasis_\text{K}$, and
enabling or disabling online basis compression.  In particular, for every
combination of the hyperparameters reported in table
\ref{tab:hyper_param_fhn}, we simulate the adaptive ROM and record its
relative $\ell^2$ state error and its mean basis dimension.  Then, for each of
the four methods, we plot the performance associated with Pareto-dominant
hyperparameter values, which are those for which no other hyperparameter value
produces strictly better performance in both relative $\ell^2$ error and mean
basis dimension.  We note that $N_\text{reset}$ corresponds to the frequency
with which basis compression or basis reset (to the original $\initBasisZ$) is
performed, depending on whether or not online basis compression is enabled.

Fig.\ \ref{fig:pareto_fhn} shows that both of the two major contributions of
this work, namely the ability to prescribe different refinement mechanisms and
the ability to perform basis compression, significantly improve the method's
performance relative to that of the  original $h$-refinement method. In
particular, we note that the performance benefits from these contributions
stack, i.e., using the basis $\leafBasis_\text{DCT}$ together with basis
compression yields the best performance.  We note that the ROM with the
(fixed) original basis of dimension three yielded 387\% relative error.

For illustrative purposes, fig.\ \ref{fig:printout_fhn} provides a simultaneous printout of the solutions to the full-order model, the reduced-order model with our refinement algorithm, and the base reduced-order model without any refinement. As we can see, the base reduced-order model completely fails to produce the appropriate dynamics of the full-order model, while our refinement algorithm manages to recover the appropriate dynamics with good accuracy.

In fig.\ \ref{fig:splitting_fhn}, we provide a visualization of the refinement of a ROM basis vector during simulation. We note that, in this instance, our method opts to split only dynamical variables associated with the variable $\voltage$.

We also provide an example of a basis compression step performed during
simulation in fig.\ \ref{fig:basis_compression_fhn}. This figure visualizes how our algorithm can significantly reduce refined ROM dimension while maintaining the ability to represent solutions at previous time steps.

\begin{table}
\centering
\begin{tabular}{|l|c|}
\hline
Hyperparameter & Test Values \\
\hline
Tree Topology: Number of Children ($k$) & $2, 4, 8, 12$ \\
Child Grouping\footnote{See \cite{carlberg2015adaptive} for a description of child grouping.} & true, false \\
Number of Time Steps Between Basis Resets / Compressions ($N_{reset}$) & $10, 25, 50, 75$ \\
Full-Order Model Tolerance ($\tol$) & \begin{tabular}{@{}c@{}}$0.01, 0.005, 0.002, 0.001,$ \\ $0.0005, 0.0002, 0.0001, 0.00005$ \end{tabular} \\
Reduced-Order Model Tolerance ($\romTol$) & $10^{-8}$ \\
\hline
\end{tabular}
\\[10pt]
\caption{Marginal hyper-parameter choices for the Pareto fronts reported in
	fig.\ \ref{fig:pareto_fhn}.}
\label{tab:hyper_param_fhn}
\end{table}

\begin{figure}
\centering
\includegraphics[scale=0.5]{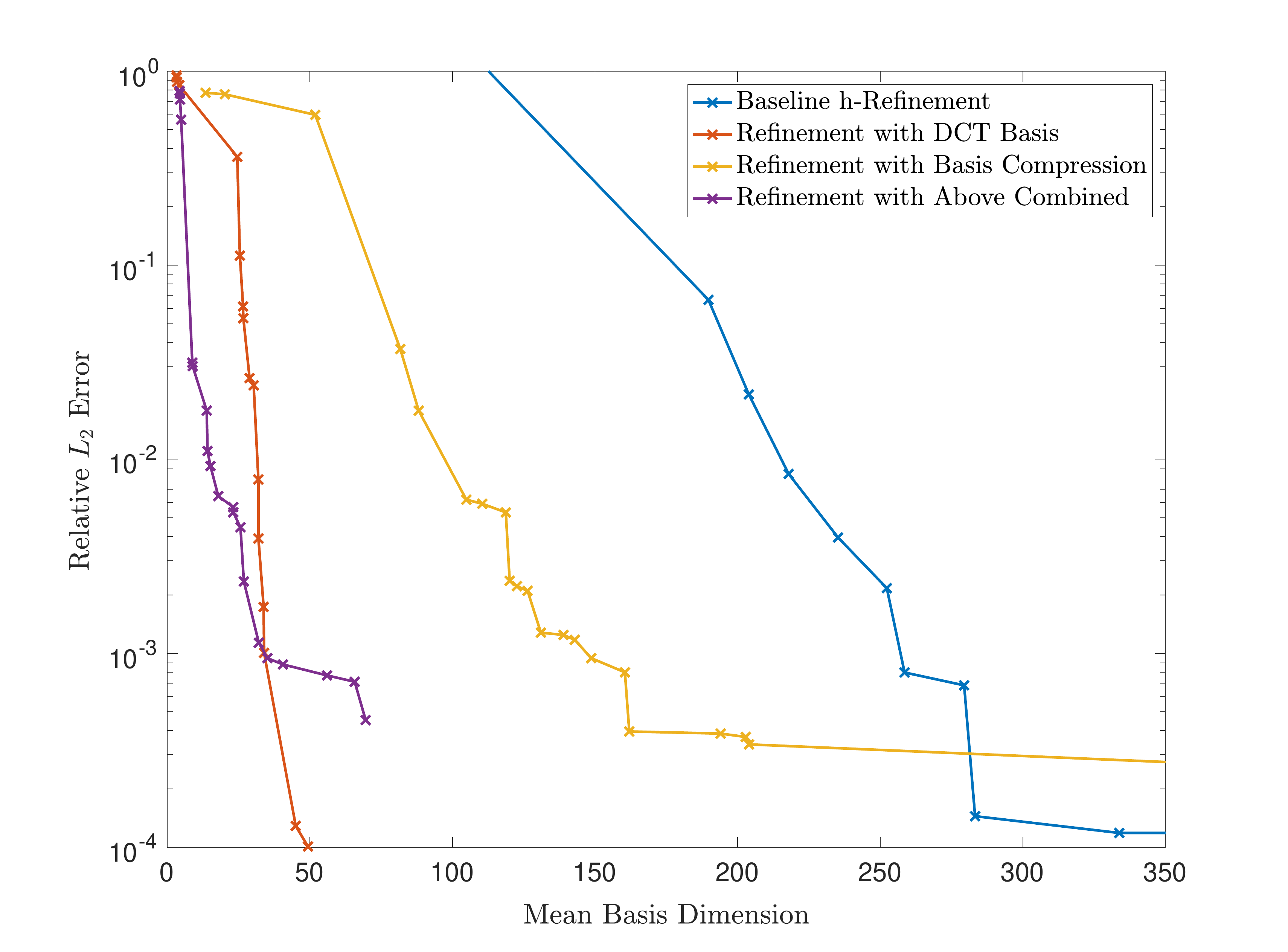}
\caption{\textbf{Pareto front comparison for FitzHugh--Nagumo system}. This figure contains Pareto fronts computed for four different versions of our method on the
	FitzHugh--Nagumo example. The method executed with leaf basis
	$\leafBasis_\text{K}$ and no basis compression, shown in blue, is identical
	to the original $h$-refinement method and serves as a baseline for
	performance comparison. The other variants correspond to the method executed
	with leaf basis $\leafBasis_\text{DCT}$ and no basis compression, shown in
	orange; the method performed with leaf basis $\leafBasis_\text{K}$ and
	basis compression, shown in yellow; and the method with
	$\leafBasis_\text{DCT}$ and basis compression, shown in purple. In this
	case, the ROM executed with the fixed initial basis $\initBasisZ$ of
	dimension $3$ yielded $387 \%$ error. The proposed method is able to reduce
	this error to arbitrarily low levels while maintaining reasonable basis
	dimensions. Furthermore, note that the major contributions of this paper,
	the ability to specify any arbitrary leaf basis, and the ability to perform
	basis compression, both lead to very significant performance increases over
	the original $h$-refinement method.} 
	\label{fig:pareto_fhn}
\end{figure}

\begin{figure}
\centering
\includegraphics[scale=0.35]{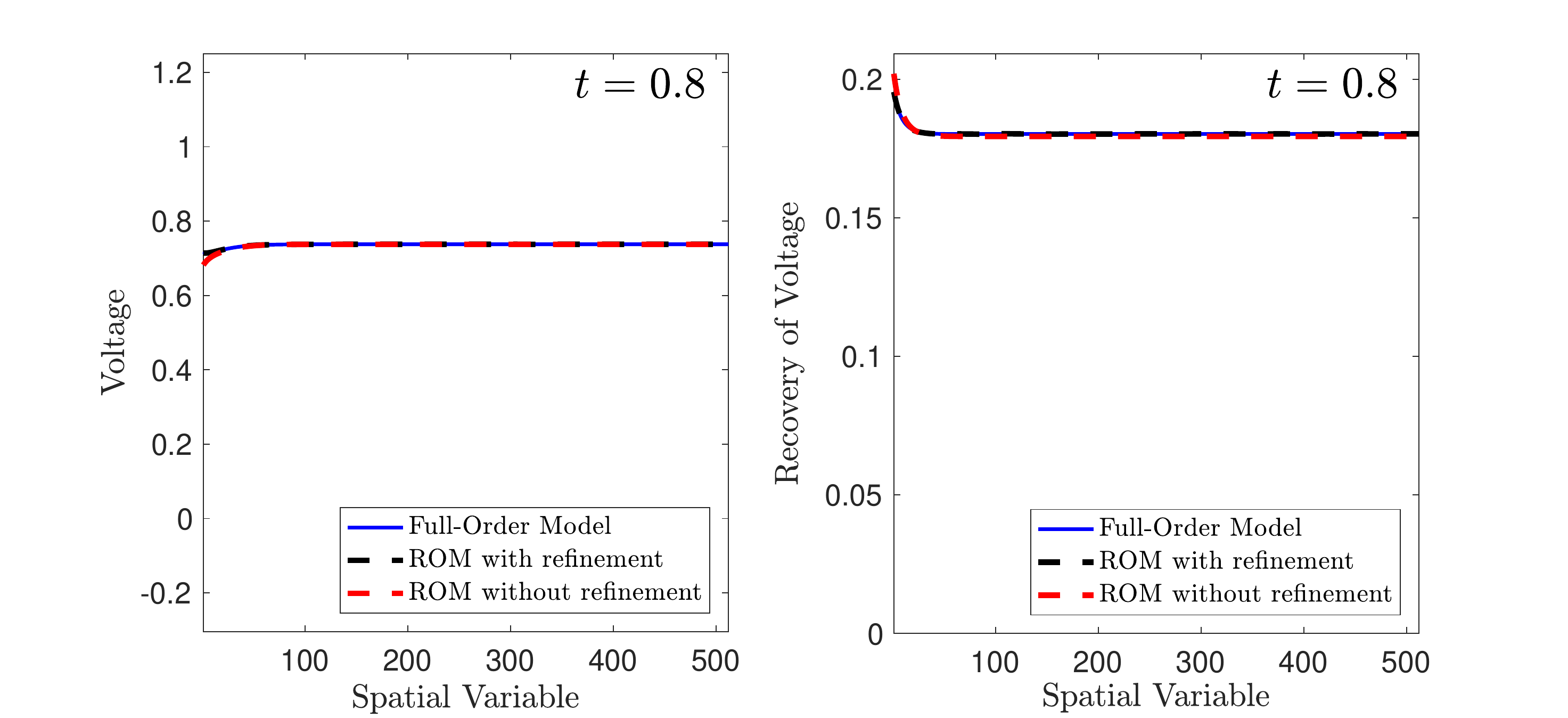}
\includegraphics[scale=0.35]{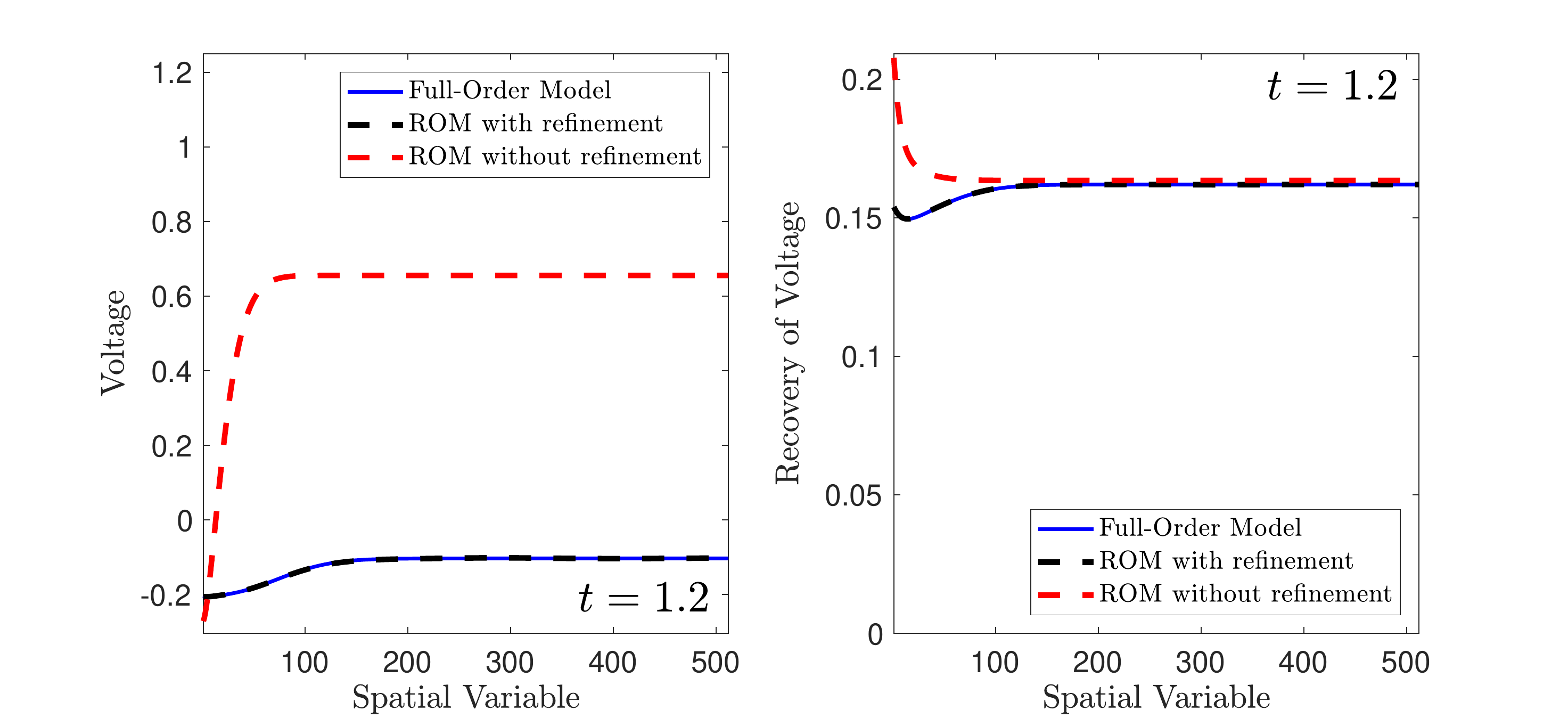}
\includegraphics[scale=0.35]{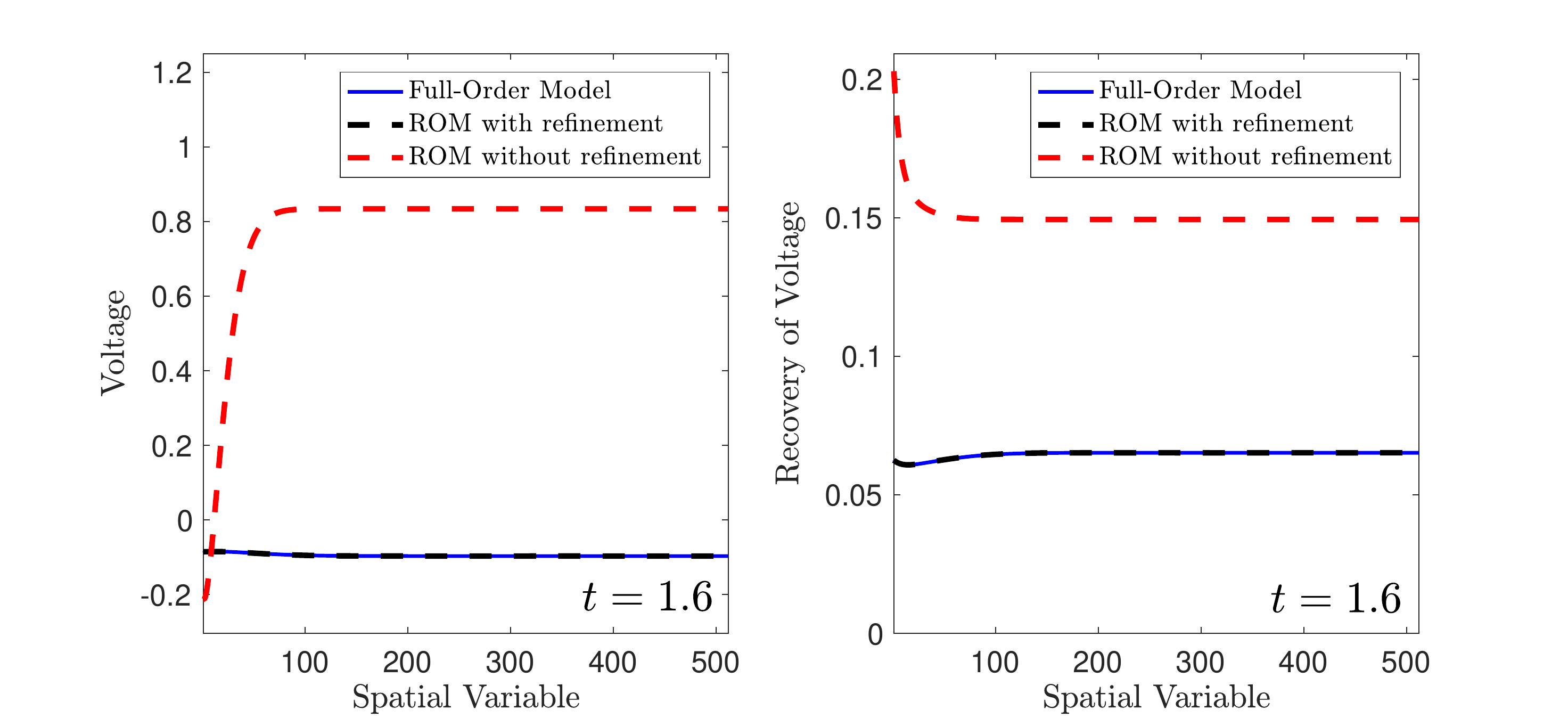}
\includegraphics[scale=0.35]{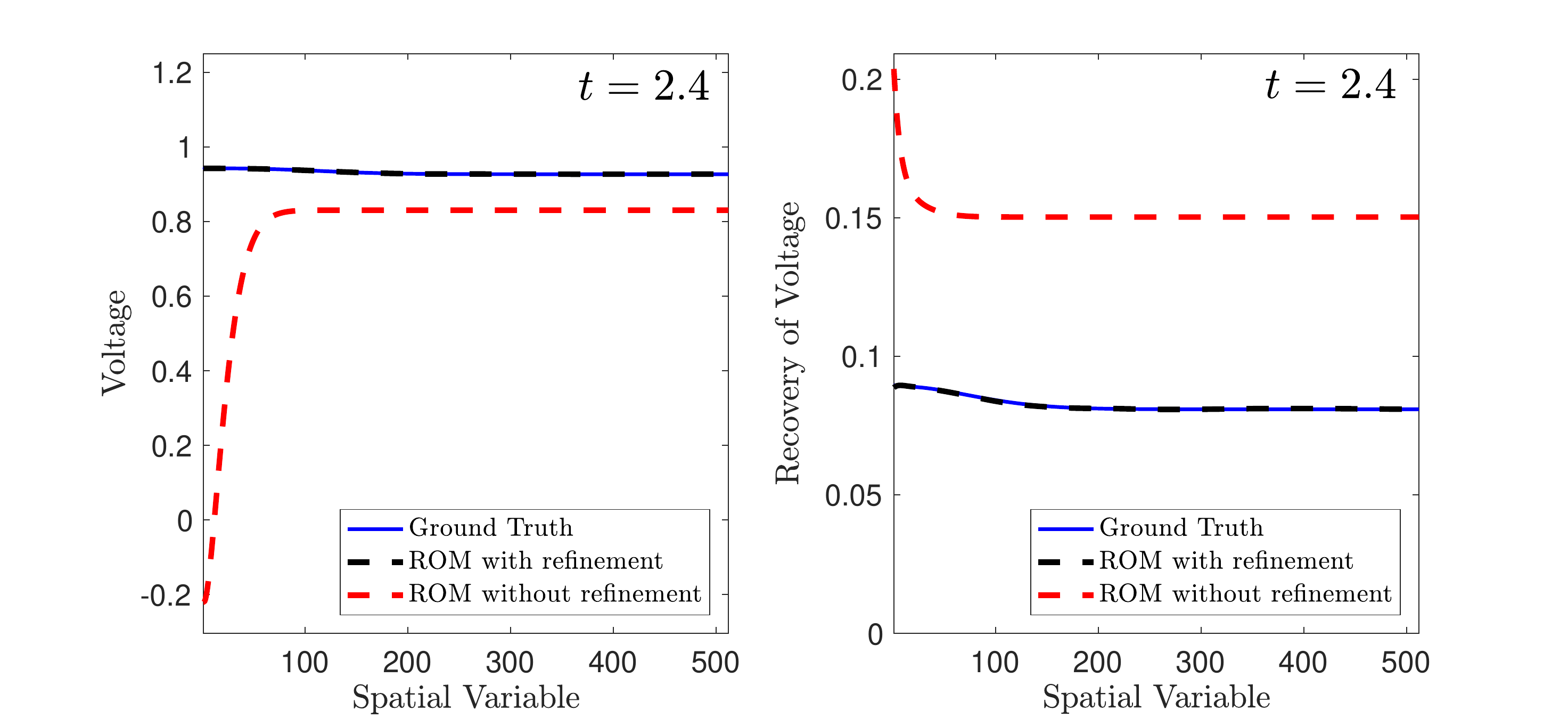}
\caption{\textbf{Comparison of FitzHugh-Nagumo Solutions}. This figure compares solutions $v(x)$ (Voltage) and $w(x)$ (Recovery of Voltage) to the FitzHugh--Nagumo
	system in section \ref{sec:fitzhugh} above at four different times $t$. The
	full-order model solution in shown in blue, while the base reduced-order
	model we use is shown in dotted red. The base reduced-order model combined
	with our refinement algorithm is shown in dotted black. In particular, we
	use a DCT refinement tree with $8$ children, child grouping enabled, a
	compression frequency of $25$ and a full-order model tolerance of $\tol =
	0.0005$. The mean basis dimension of our refined ROM is $48$ (for
	comparison, the problem dimension is $1024$). The dynamics of this
	FitzHugh--Nagumo System involve bouncing back and forth between two regions
	in phase space. These bounces (i.e., neural spikes) happen on a very short
	time scale and are very difficult for the base ROM to capture. Indeed, as we see
	above, the base ROM is incapable of resolving these spikes and hence remains
	stuck in one region of phase space for the entirety of the simulation. Our
	refinement algorithm allows this behavior to be resolved very precisely,
	with final relative $\ell^2$ error of $0.7 \%$.}
\end{figure}

\begin{figure} \label{fig:printout_fhn}
\centering
\includegraphics[scale=0.6]{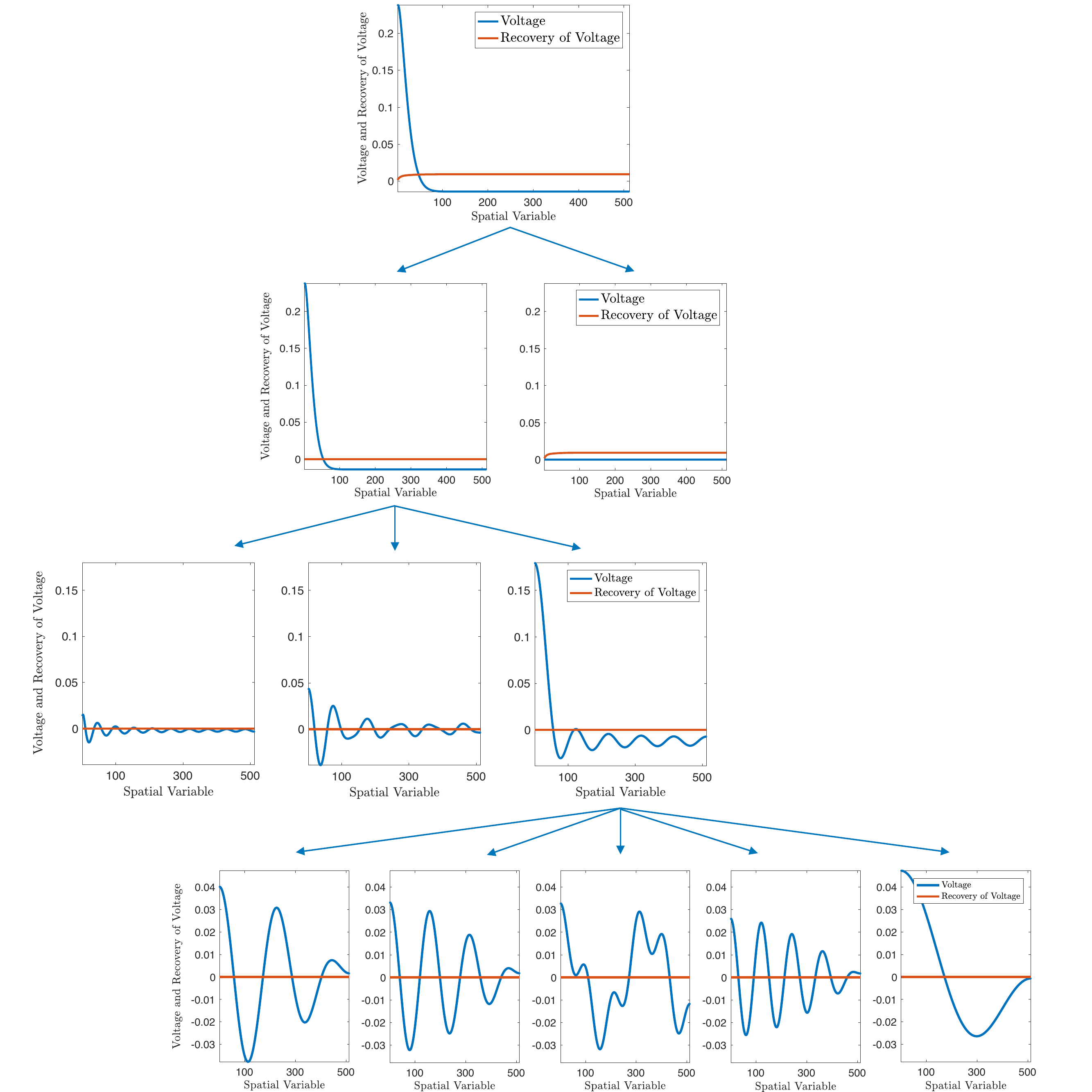}
\caption{\textbf{An illustration of the refinement of a ROM basis vector} as it is
	sieved through the refinement tree in the FitzHugh--Nagumo case study above.
	The degrees of freedom are laid out as described in section
	\ref{sec:fitzhugh}. The first split corresponds to the vector space
	decomposition $\R{\stateVecSize} = \parentSpace^{(\voltage)} +
	\parentSpace^{(\recVoltage)}$, decoupling the degrees of freedom
	corresponding to $\voltage$ and $\recVoltage$, respectively. After this
	first refinement, the next two refinements the refinement algorithm performs
	two band-pass filters to decouple the low frequencies on the degrees of
	freedom corresponding to $\voltage$. Typically, the solutions to the
	FitzHugh--Nagumo system in section \ref{sec:fitzhugh} are very smooth on the
	computational domain, with the exception of at the boundary. Hence, by
	isolating low frequencies components from the original ROM basis vector, our
	refined model is able to efficiently represent smooth data, leading to
	significantly lower average basis dimension than if one had used a
	refinement tree corresponding to the Kronecker basis. This performance
	increase is evident in fig.\ \ref{fig:pareto_fhn}.}
\end{figure}

\begin{figure} \label{fig:splitting_fhn} \label{fig:basis_compression_fhn}
\centering
\includegraphics[scale=0.5]{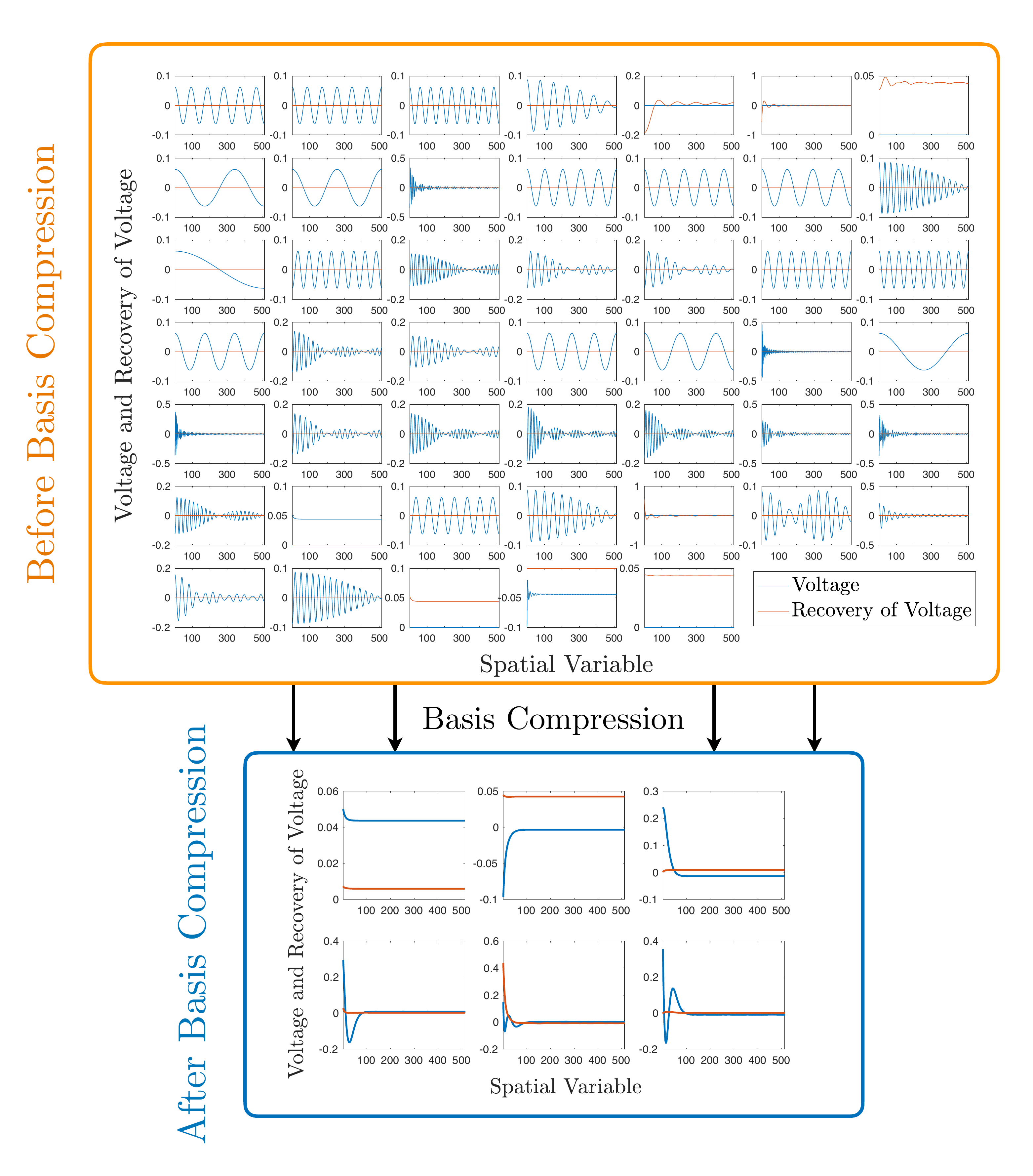}
\caption{\textbf{Visualization of basis compression}. This figure presents a visualization of our basis compression algorithm on the
	FitzHugh--Nagumo system in section \ref{sec:fitzhugh}. The compression
	algorithm allows us to reduce the dimension of our refined ROM without
	losing the ability to represent solution data at previous time steps. In
	this case, the compression algorithm determines the important fidelity added
	to the ROM via refinement is concentrated on the left-hand side of the
	domain, and the remainder of the additional representative power of the
	refined ROM can be discarded without significantly impacting the ROM's
	ability to capture solution data at previous time steps. If the underlying
	system exhibits some degree of temporal coherence, then this technique can
	lead to significant reduction in average basis dimension over simply
	resetting the ROM, as seen in fig.\ \ref{fig:pareto_fhn}.}
\end{figure}

\subsection{Nonlinear transmission line model} \label{sec:nonlinear_transmission}

The second example we consider corresponds to a nonlinear transmission line
model \cite{white2003trajectory}, which simulates the behavior of a particular
circuit consisting of resistors, capacitors, and diodes and has been used as a
benchmark problem for model-reduction techniques. The system consists of a
collection of $\stateVecSize = 100$ nodes, with the unknowns corresponding to
the voltages $v_1, \ldots, v_\stateVecSize$ at each of these nodes such that
the state vector is
$\stateVec = \left[\begin{array}{ccc} v_1 & \ldots & v_\stateVecSize
\end{array}\right]^T$.
The dynamics can be expressed as a system of nonlinear ODEs
\begin{equation}
\frac{d\stateVec}{dt} = \velocity(\stateVec) + \mat{B} \, u(t) \,,
\end{equation}
where the velocity $\velocity$ is 
\begin{equation}
	\velocity:\stateVec \mapsto \left[\begin{array}{cccccc} 
-2 	& 1 		& 			& 			& & \\
1 	& -2		& 1			&			& & \\
	& 1 		& \ddots 	& \ddots 	& & \\
	&		& \ddots 	& -2			& 1 &\\
	&		&			& 1 			& -2 & 1\\
	&		&			&			& 1 	 & -2 
	\end{array}\right] \stateVec + \left[\begin{array}{c} 
	
	2 - \exp(40x_1) - \exp(40(x_1 - x_2)) \\
	\exp(40(x_1 - x_2)) - \exp(40(x_2 - x_3)) \\
	\exp(40(x_2 - x_3)) - \exp(40(x_3 - x_4)) \\
	\vdots \\
	\exp(40(x_{n-2} - x_{n_1})) - \exp(40(x_{n - 1} - x_n)) \\
	\exp(40(x_{n - 1} - x_n)) - 1
	\end{array}\right] \,,
\end{equation}
and the input matrix is
$\mat{B} \equiv \left[\begin{array}{cccc} 1 & 0 & \cdots & 0
\end{array}\right]^T$ such that the input current
$u(t)$  enters the first node. 
The initial condition is given by
$
\left.\stateVec\right|_{t = 0} = \ve{0}$. 
To numerically integrate the governing ODE system in the time interval
$\mathbb T = [0,10]$, we employ the backward-Euler scheme
with a time step size $\Delta t = 0.01$.

As in the previous example, we again demonstrate that the proposed method's
ability to consider general refinement mechanisms and perform online basis
compression can yield substantial performance improvements over the original
$h$-refinement method. Once again, we construct two refinement trees
$\tree_\text{DCT}$ and $\tree_\text{K}$ associating with a discrete cosine
transform leaf basis and a standard Kronecker leaf basis, respectively. 

We consider a predictive scenario wherein training is executed for a training
input $u(t)=u_\text{train}(t)\defeq 1 - {t}/{50}$ and testing is executed for an online input
$u(t)=u_\text{test}(t)\defeq 1/2(\cos(2 \pi t/10) + 1)$. 
We set the initial basis $\initBasisZ$ to the first $\initRomSize=4$ POD vectors
of snapshot data $\snaps \in \R{\stateVecSize \times \compSnaps}$ with
$\compSnaps=1000$ collected at the training input.

To generate the trees $\tree_\text{DCT}$ and $\tree_\text{K}$, we use the
procedure outlined in section \ref{sec:tree_construction}, with
\begin{equation}
\begin{split}
\leafBasis_\text{DCT} &\equiv \mat{M}^{(\stateVecSize)}_\text{DCT} \,, \\
\leafBasis_K &\equiv \mat{I}^{(\stateVecSize)} \,,
\end{split}
\end{equation}
where $\mat{M}^{(\stateVecSize)}_\text{DCT}$ is the $\stateVecSize \times
\stateVecSize$ DCT-II matrix and $\mat{I}^{(\stateVecSize)}$ is the
$\stateVecSize \times \stateVecSize$ identity. We use the snapshot data
$\snaps$ as the input to the tree-construction procedure.

Fig.\ \ref{fig:pareto_nltrans} reports the resulting Pareto fronts, which
arise from varying all hyperparameters according to values in table
\ref{tab:hyper_param_nltrans}.  Once again, we observe the new contributions
of this work to yield significant performance improvements over the original
$h$-refinement method. In particular, using the DCT basis with online basis
compression clearly yields the best performance.
We note that the ROM with the
(fixed) original basis of dimension three yielded 20.2\% relative error. Just
like in the FighHugh-Nagumo example, we provide another simultaneous printout
of the solutions to the FOM, refined ROM, and unrefined ROM in fig.\ \ref{fig:printout_nltrans}. We note again that the unrefined ROM fails to fully resolve the physics of the FOM, but the refined ROM performs quite well. 

\begin{table}
\centering
\begin{tabular}{|l|c|}
\hline
Hyperparameter & Test Values \\
\hline
Tree Topology: Number of Children ($k$) & $2, 4, 8, 12$ \\
Child Grouping & true, false \\
Number of Time Steps Between Basis Resets / Compressions ($N_{reset}$) & $10, 25, 50, 75$ \\
Full-Order Model Tolerance ($\tol$) & \begin{tabular}{@{}c@{}}$0.01, 0.005, 0.002, 0.001,$ \\ $0.0005, 0.0002, 0.0001, 0.00005$ \end{tabular}
 \\
Reduced-Order Model Tolerance ($\romTol$) & $10^{-8}$ \\
\hline
\end{tabular}
\\[10pt]
\caption{Marginal hyper-parameter choices for Pareto plot shown in fig.\ \ref{fig:pareto_nltrans}.}
\label{tab:hyper_param_nltrans}
\end{table}

\begin{figure}
\centering
\includegraphics[scale=0.5]{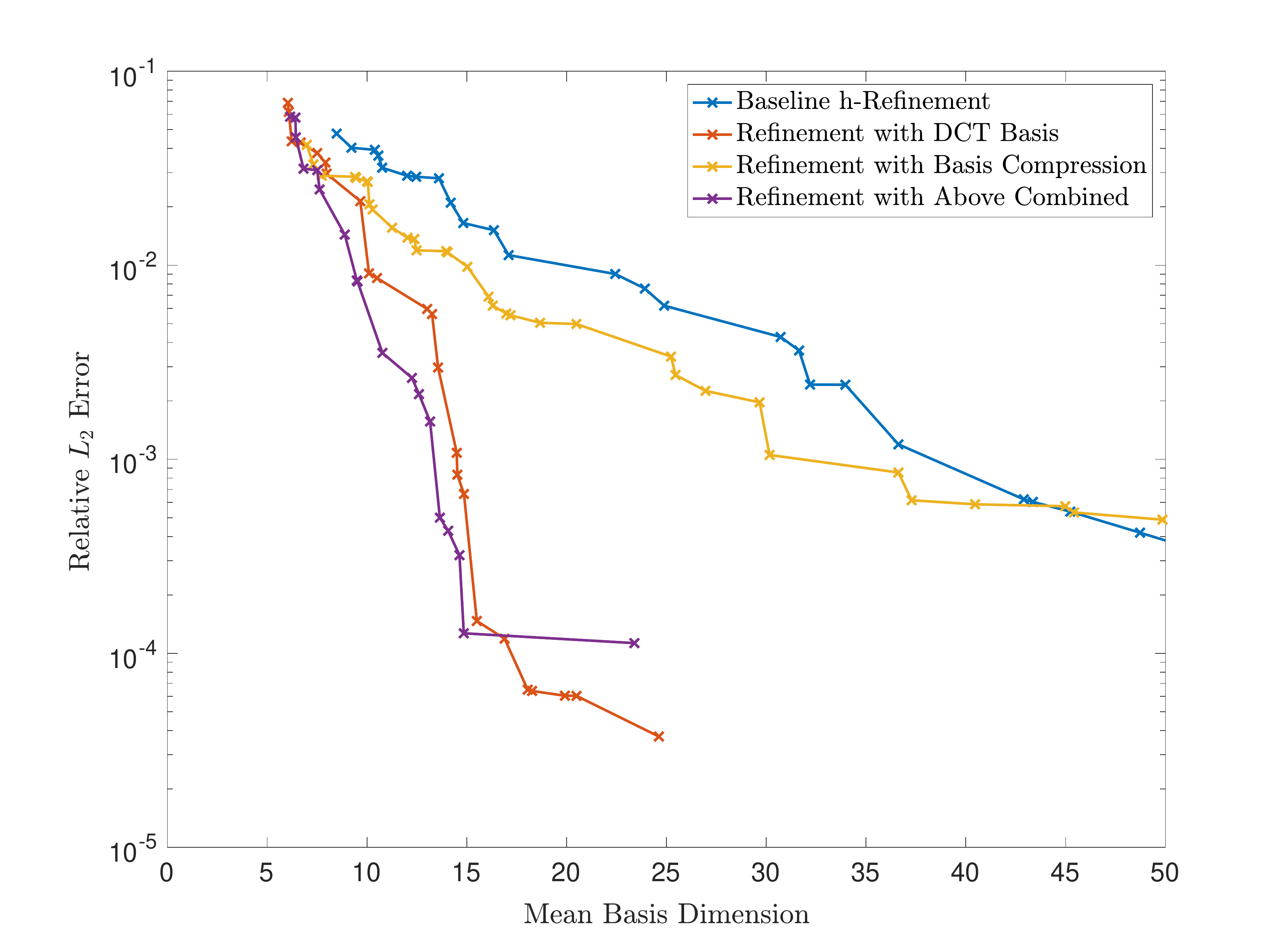}
\caption{
	\textbf{Pareto front comparison for the Nonlinear Transmission Line system}. This figure contains Pareto fronts computed for four different versions of our method on the
	Nonlinear Transmission Line example.
	The method executed with leaf basis
	$\leafBasis_\text{K}$ and no basis compression, shown in blue, is identical
	to the original $h$-refinement method and serves as a baseline for
	performance comparison. The other variants correspond to the method executed
 the method performed with leaf basis
	$\leafBasis_\text{DCT}$ and no basis compression, shown in orange; the
	method performed with leaf basis $\leafBasis_\text{K}$ and basis
	compression, shown in yellow; and the method with $\leafBasis_\text{DCT}$
	and basis compression, shown in purple. 
	In this
	case, the ROM executed with the fixed initial basis $\initBasisZ$ of
	dimension $4$ yielded $20.2 \%$ error.
	Again, the proposed method is able to reduce this error to arbitrarily low
	levels while maintaining reasonable basis dimensions. Furthermore, note that
	the major contributions of this paper, the ability to specify any arbitrary
	leaf basis, and the ability to perform basis compression, both lead to very
	significant performance increases over the original $h$-refinement method.}
\label{fig:pareto_nltrans}
\end{figure}

\begin{figure} \label{fig:printout_nltrans}
\centering
\includegraphics[scale=0.4]{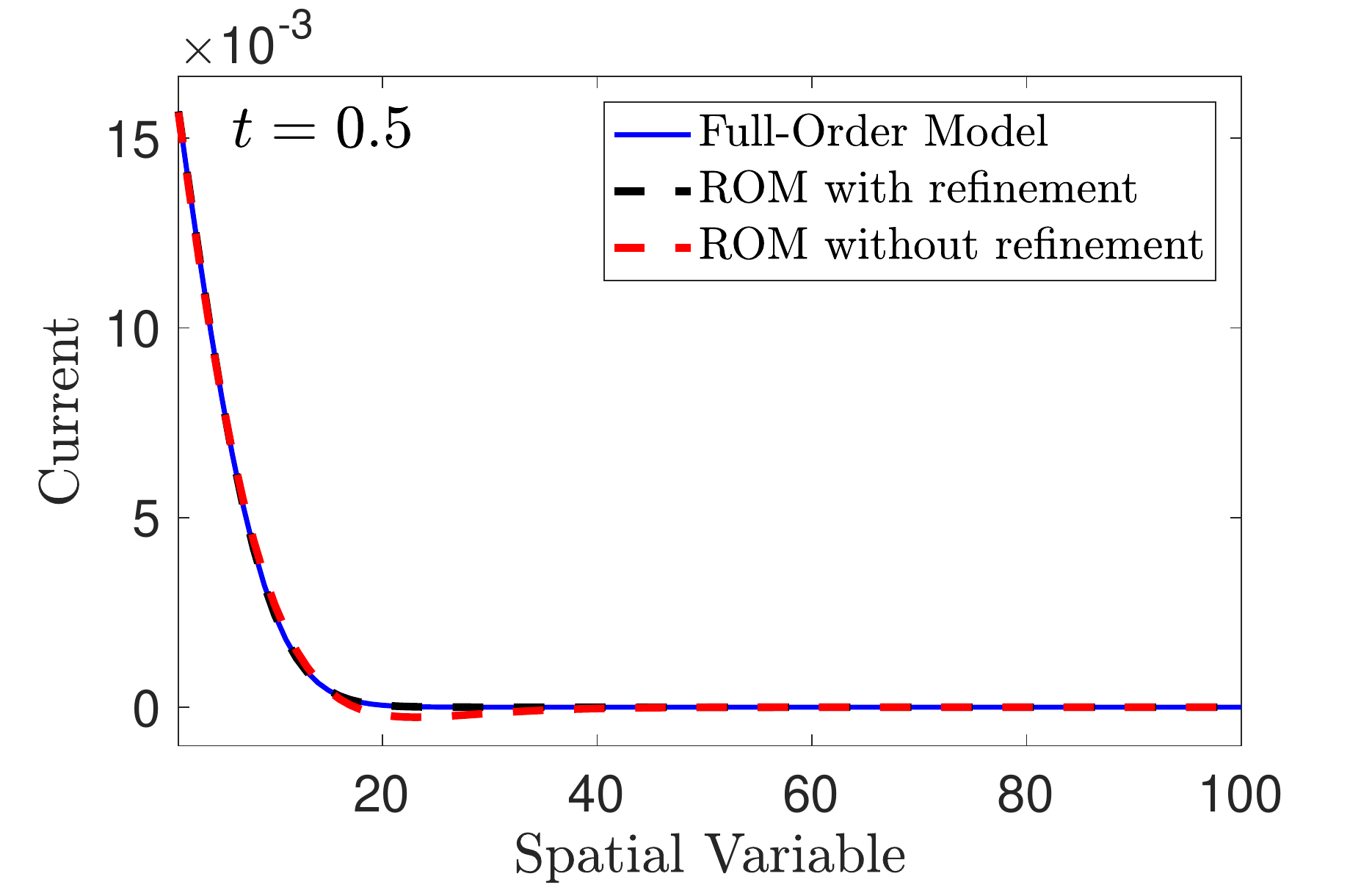}
\includegraphics[scale=0.4]{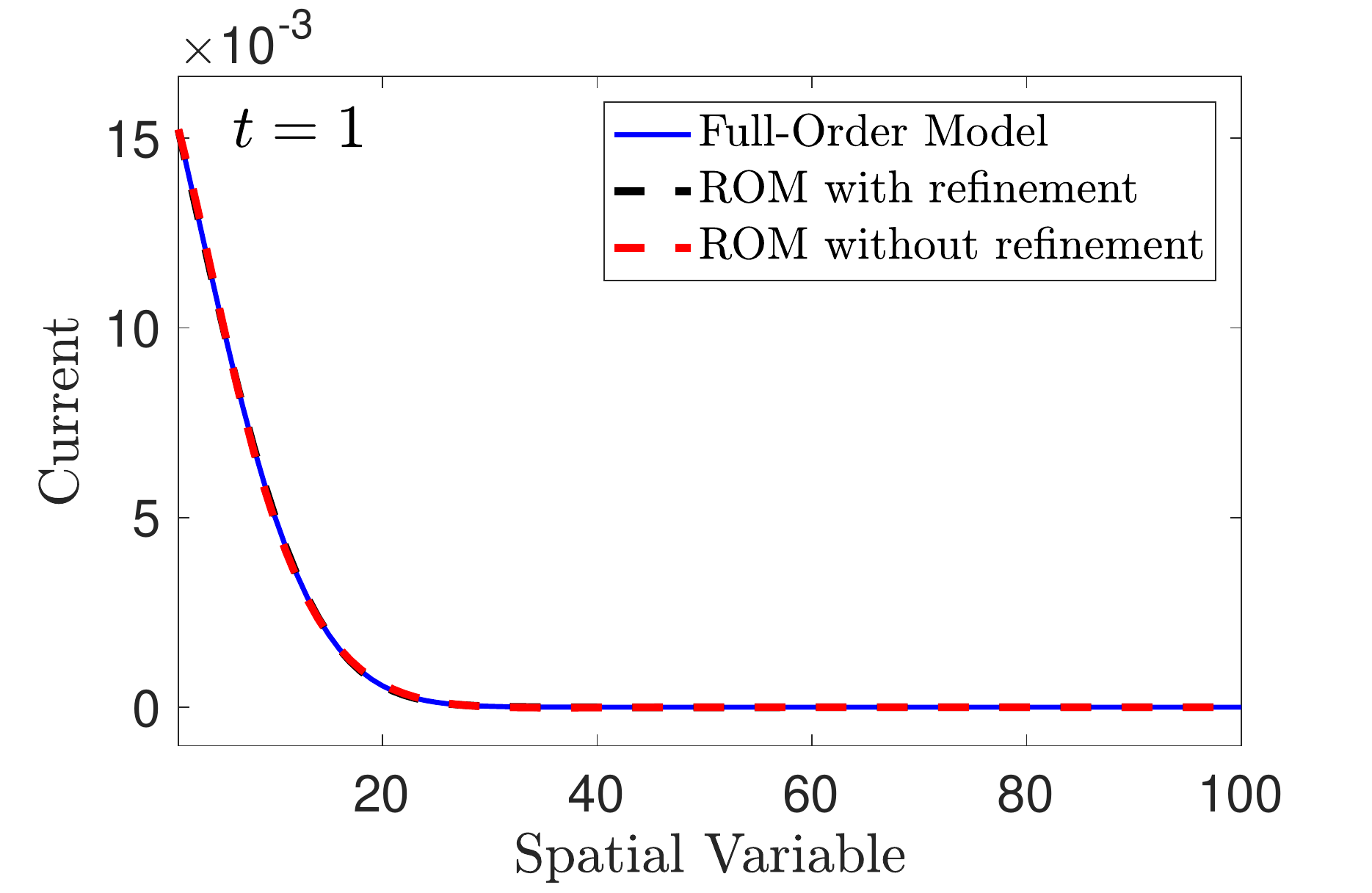}
\includegraphics[scale=0.4]{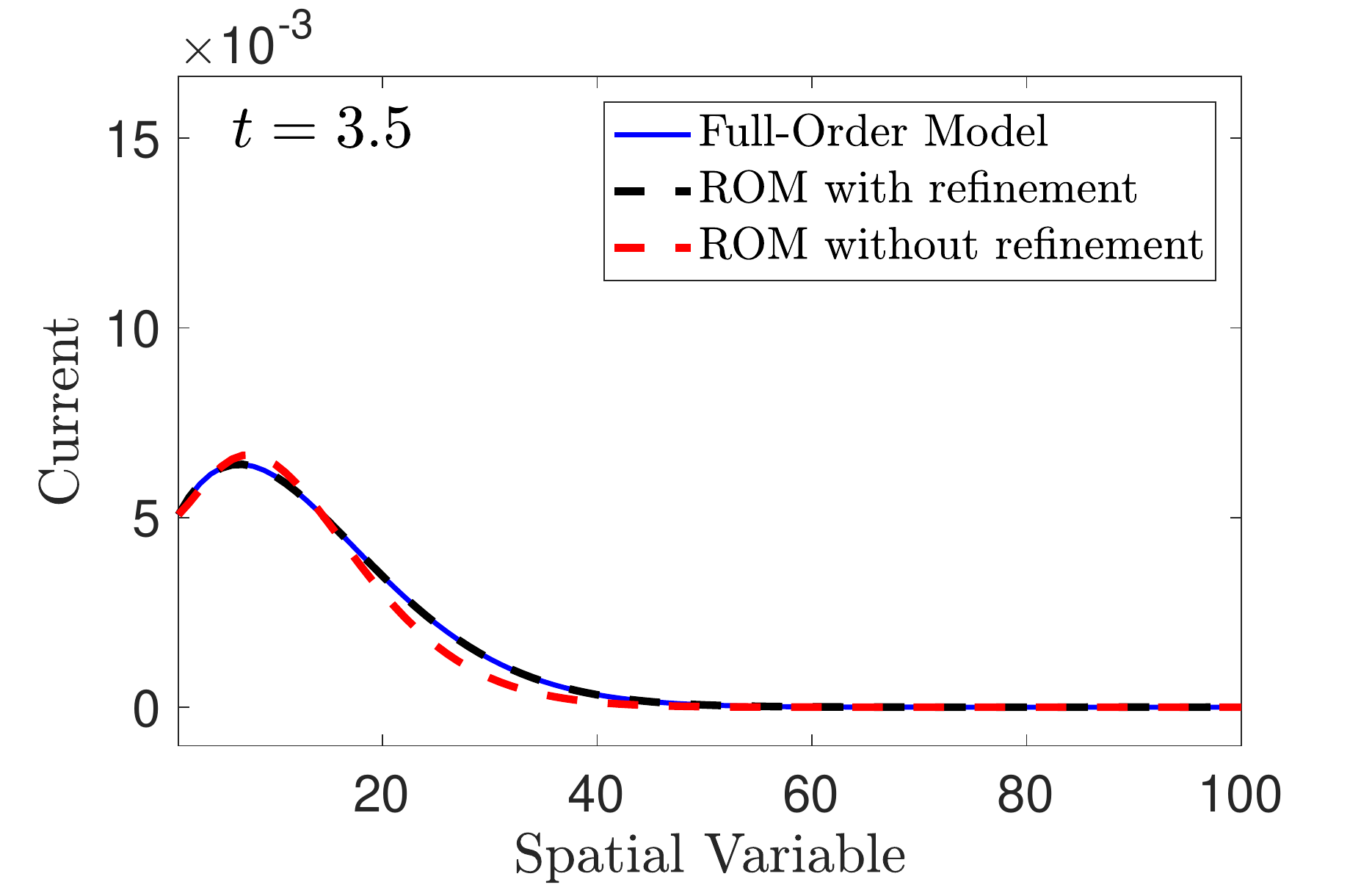}
\includegraphics[scale=0.4]{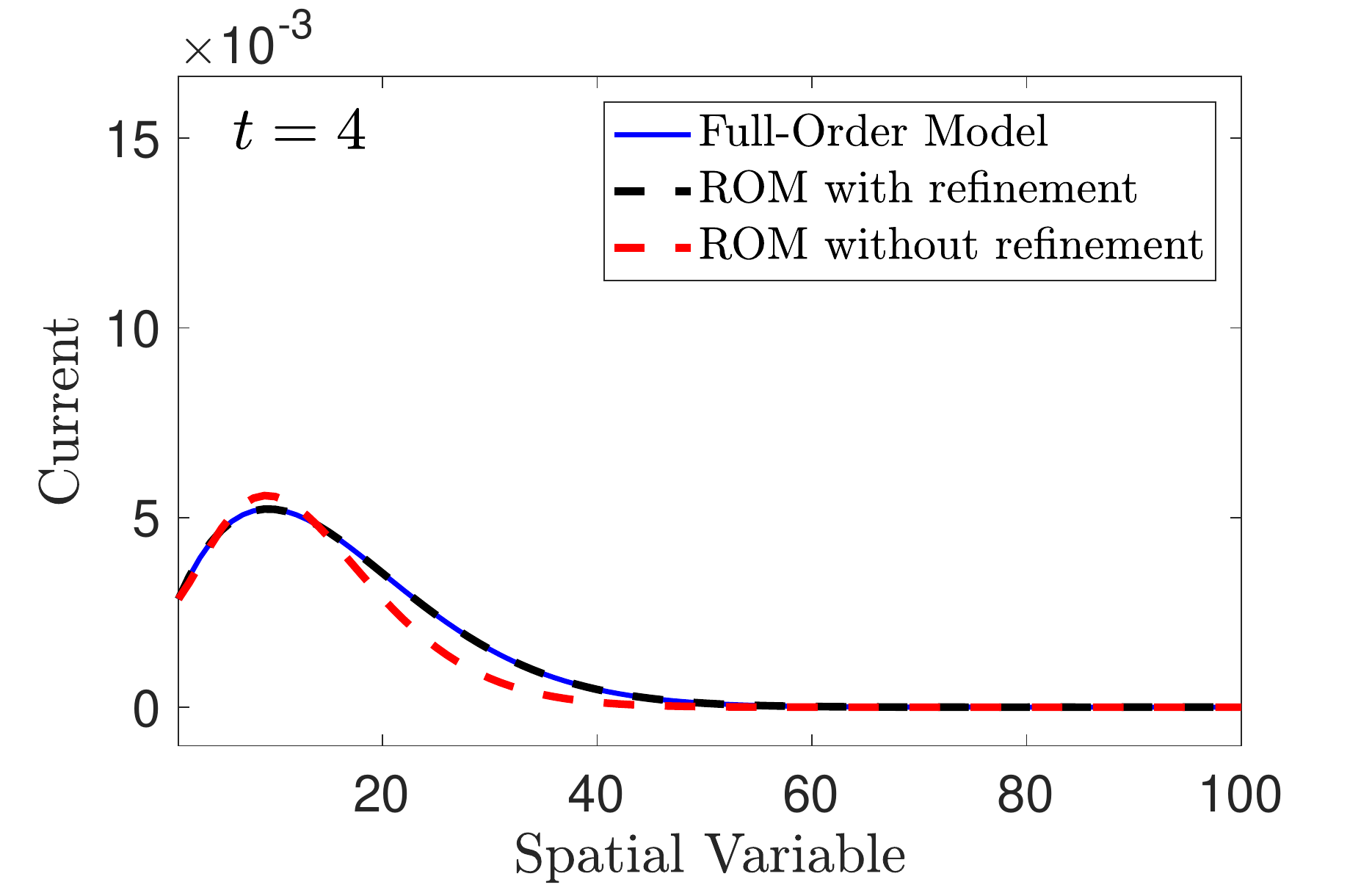}
\includegraphics[scale=0.4]{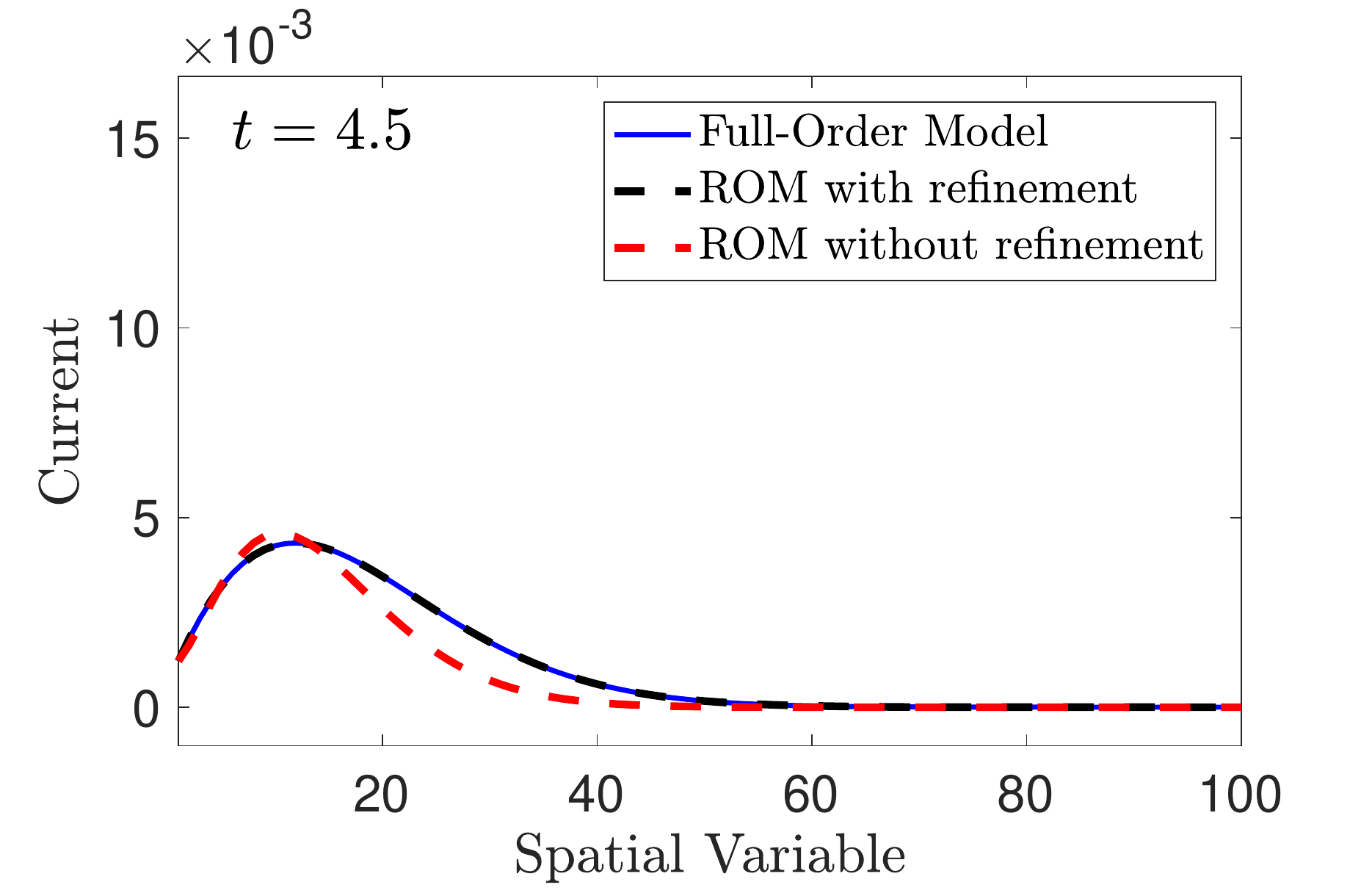}
\includegraphics[scale=0.4]{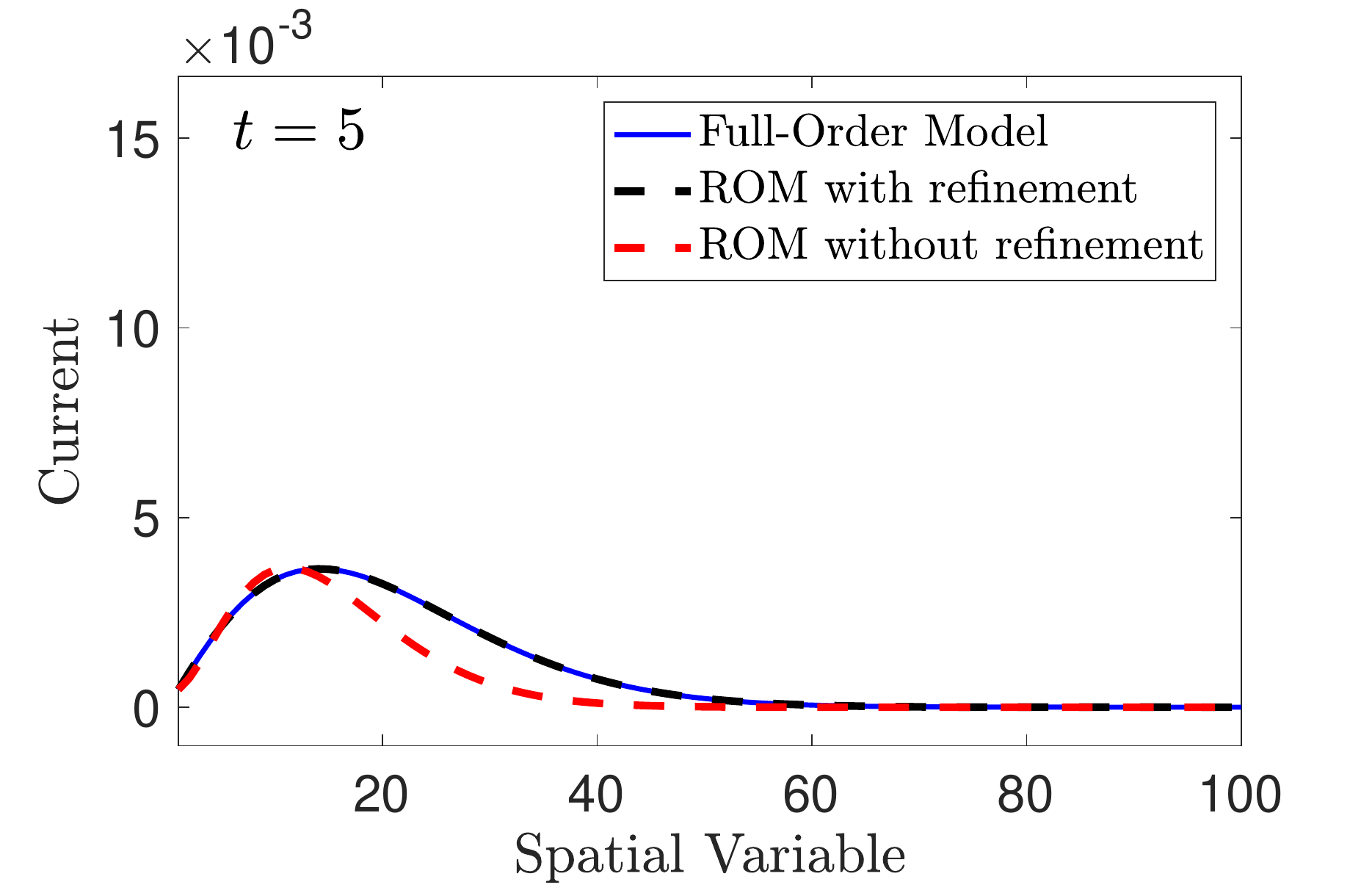}
\caption{\textbf{Comparison of Nonlinear Transmission Line Solutions}. This figure shows an illustration of the output of our algorithm on the nonlinear
	transmission line system in section \ref{sec:nonlinear_transmission} at different times $t$. The
	full-order model solution in shown in blue, while the base reduced-order
	model we use is shown in dotted red. The base reduced-order model combined
	with our refinement algorithm is shown in dotted black. In particular, we
	use a DCT refinement tree with $8$ children, child grouping enabled, a
	compression frequency of $25$ and a full-order model tolerance of $\tol =
	0.001$. The mean basis dimension of our refined ROM is $13$ (for comparison,
	the problem dimension is $100$). The dynamics of this Nonlinear Transmission
	Line System involves relaxation from a pulse at $t = 0$. We note that the
	base ROM performs well initially, but fails to resolve later times
	correctly. On the other hand, our refinement algorithm allows this behavior
	to be resolved, with final relative $\ell^2$ error of $1.9 \%$.}
\end{figure}

\section{Conclusions}\label{sec:conclusions}

This work has proposed an online adaptive basis refinement and compression
method for reduced-order models. The principal new contributions of the method
include:
\begin{enumerate}
\item A mathematical framework, presented in sections
	\ref{sec:math_framework}, \ref{sec:basis_refinement}, and
		\ref{sec:tree_construction}, that generalizes the original ROM
		$h$-refinement method \cite{carlberg2015adaptive}, as it enables a general
		basis-refinement mechanism based on recursive vector-space decompositions.
		This allows for custom tailoring of the ROM refinement mechanism to the
		particular problem, which we have demonstrated can significantly improve
		performance over the original approach.
\item A novel online basis-compression algorithm, presented in section
	\ref{sec:online_basis_compression}, which controls the dimension of the
		refined basis online while ensuring an $\stateVecSize$-independent
		operation count.  We have demonstrated that this aspect of the method
		enables additional substantial performance improvements over the original
		approach.
\end{enumerate}

The proposed approach distinguishes itself from existing approaches for online
ROM adaptivity in that it avoids any FOM solves, yet it also ensures monotone convergence to the full-order
model, as proved in theorems \ref{thm:convergence} and \ref{thm:monotonicity}.

Future research directions include integrating hyper-reduction techniques
such as collocation, empirical interpolation, or gappy POD into the proposed
framework.  This would entail adaptively
adding/removing residual sampling points to maintain well-posedness of the
refined-ROM system; one challenge here involves ensuring the 
sampling pattern is `compatible' with the refined basis, so
that one does not run into the scenario where, for example, many sample points
fall outside the supports of refined basis vectors. 
Another direction for investigation entails devising approaches to reduce the
storage requirements for leaf vector spaces with global support, either by
using an \textit{on-the-fly} computation approach (as we employ to compute the
metric $\romBasis^T \romBasis$) or perhaps via a sampling approach.
Finally, future work entails applying the proposed method to a truly
large-scale industrial problem.

\appendix 

\section{Modified splitting algorithm} \label{sec:mod_split}

In this appendix, we give a pseudocode integration of the conditioning technique we presented in section \ref{sec:illCond} into the refinement algorithm \ref{alg:refine_proc}. To perform this integration, we must maintain a list $\mathcal{I} \subset \bigsqcup_i \frontier_i$ of inactive vertices in the frontiers $\frontier_1, ..., \frontier_{\initBasisSize}$. After refinement, whenever a vertex is detected by the ill-conditioning check in \ref{sec:illCond}, we place it into the list $\mathcal{I}$. Afterwards, during the frontier refinement algorithm, specifically during the computation of the full refinement $\frontierglobalF$ of the global frontier $\frontierglobalC \equiv \bigsqcup_i \frontier_i$ we deliberately exclude all of the vertices in $\mathcal{I}$ from taking part in refinement. Furthermore, we also implement basis vector rescaling, which maintains a scale factor for every active frontier element in $\bigsqcup_i \frontier_i \setminus \mathcal{I}$ , as described in section \ref{sec:illCond}. 

Below, in algorithms \ref{alg:mod_err}, \ref{alg:mod_refine_frontiers}, \ref{alg:mod_refine_proc}, we give the modified versions of the algorithms \ref{alg:err}, \ref{alg:refine_frontiers}, \ref{alg:refine_proc} respectively as described in section \ref{sec:illCond}. Note that in algorithm \ref{alg:mod_refine_proc}, the implementation of $\textsc{HandleFullInactiveVertices}$ is deferred to the user. This routine is triggered when all available refinement options have been systematically deactivated by the conditioning technique. If the value of the cutoff $\epsilon_{\text{QR}}$ used for the conditioning technique is too high, it is possible that we may disable all available refinement options before we converge to a solution that is within the desired tolerance. In order to address this, there are a number of things the user can do to reactivate vertices which have been deactivated if this corner case is reached. two possibilities include:
\begin{enumerate}
\item Simply resetting $\mathcal{I}$ to empty and halving the value of $\epsilon_{\text{QR}}$.
\item Performing a column-pivoted QR decomposition of $\left[\begin{array}{ccc}
\matsieve{\initBasisCol_1}{\frontier_1} & \cdots & \matsieve{\initBasisCol_{\romSize_0}}{\frontier_{\initRomSize}}  \end{array}\right]$ to determine which elements of $\mathcal{I}$ to reactivate. Optionally, one can modify $\epsilon_{\text{QR}}$ by examining the diagonal values of the above QR decomposition.
\end{enumerate}
The specific implementation may depend on the use-case, and hence is left up to the user. 

\begin{algorithm}
\caption{Modified computation of Error Indicators} \label{alg:mod_err}
\hspace*{\algorithmicindent} \textbf{Input}: The current coarse basis $\romBasisC_*$, the current frontiers $\frontier_1, \ldots, \frontier_{\romSize_0}$, the set of inactive vertices $\mathcal{I}$. \\
\hspace*{\algorithmicindent} \textbf{Output}: The fine error indicators $\errorIndF$.
\begin{algorithmic}[1]
\Procedure{ComputeErrorIndicators}{$\romBasisC_*$, $\frontier_1, \ldots, \frontier_{\initRomSize}, \mathcal{I}$}
\State $\frontierglobalC \gets \left(\bigsqcup_i \frontier_i \right)\setminus \mathcal{I}$
\State $\frontierglobalF \gets \bigsqcup_i \left(\bigcup_{\parentSpace \in \frontier_i \setminus \mathcal{I}} \tRefine{\parentSpace}{\tree}\right) $ \Comment{To compute the refined global frontier $\frontierglobalF$, we refine all spaces in $\frontierglobalC$ except those marked as inactive.}
\State $\prolong \gets \textsc{ComputeProlongationOperator}\left(\frontierglobalC, \frontierglobalF\right)$ \Comment{Compute prolongation operator using Eq.\ (\ref{eq:prolong_def})}
	\State\label{step:mod_coarseAdjoint} $\adjointC \gets \left[(\romBasisC_*)^T \frac{\partial \gResidual}{\partial \stateVec} (\romBasisC_* \romStateVecC)^T \romBasisC_* \right]^{-T} \left[ (\romBasisC_*)^T \frac{\partial \interestFunc}{\partial \stateVec} \left(\romBasisC_* \romStateVecC\right)^T \right]$ \Comment{Compute the coarse adjoint using Eq.\ (\ref{eq:coarse_adjoint})}
\State $\prolongadjoint \gets \prolong \adjointC$ \Comment{Prolongate coarse adjoint to fine coordinate space.}
\State $\errorIndF \gets \ve{0} \in \R{\frontierglobalF}$
\For{$\childSpace \in \frontierglobalF\setminus\treeLeaves$} 
	\State \label{step:mod_errorIndicators}$\errorIndF_\childSpace \gets  \left|\left[\prolongadjoint\right]_\childSpace \left(\romBasisCol_\childSpace^h\right)^T \gResidual\left(\romBasisC \romStateVecC\right)\right| $\Comment{Compute the error indicators using Eq.\ (\ref{eq:error_ind})}
\EndFor
\State \Return $\errorIndF$
\EndProcedure
\end{algorithmic}
\end{algorithm}

\begin{algorithm}
\caption{Modified computation of Refined Frontiers} \label{alg:mod_refine_frontiers}
\hspace*{\algorithmicindent} \textbf{Input}: The $\Rn$-refinement tree $\tree$, the fine error indicators $\errorIndF$, the current frontiers $\frontier_1, \ldots, \frontier_{\romSize_0}$, the set of inactive vertices $\mathcal{I}$.  \\
\hspace*{\algorithmicindent} \textbf{Output}: A new set of frontiers $\frontier_1', \ldots, \frontier_{\initRomSize}'$ refined according to the input error indicators.
\begin{algorithmic}[1]
\Procedure{RefineFrontiers}{$\tree, \errorIndF, \frontier_1, \ldots, \frontier_{\initRomSize}$}
\State $\frontierglobalC \gets \left(\bigsqcup_i \frontier_i \right)\setminus \mathcal{I}$
\State $\frontierglobalF \gets \bigsqcup_i \left(\bigcup_{\parentSpace \in \frontier_i \setminus \mathcal{I}} \tRefine{\parentSpace}{\tree}\right) $ \Comment{To compute the refined global frontier $\frontierglobalF$, we refine all spaces in $\frontierglobalC$.}
\State $\ancestor \gets \textsc{GetGlobalAncestorMap}\left(\frontierglobalF, \frontierglobalC\right)$ \Comment{Compute the map in Eq.\ (\ref{eq:global_ancestor_map}) sending every space to its ancestor.}
\State $\prolong \gets \textsc{ComputeProlongationOperator}\left(\frontierglobalC, \frontierglobalF\right)$ \Comment{Compute prolongation operator using Eq.\ (\ref{eq:prolong_def}).}
\State $\errorIndC = \errorIndF \prolong$ \Comment{Compute coarse error indicators using Eq.\ (\ref{eq:coarse_err_prolong}).}
\State $\eta \gets \frac{1}{|\frontierglobalC|} \sum_{\parentSpace \in \frontierglobalC} \errorIndC_\parentSpace$ \Comment{Compute the average of the coarse error indicators.}
\State $S \gets \{ \parentSpace \in \frontierglobalC \mid \errorIndC_\parentSpace \geq \eta \}$ \Comment{Select the spaces in $\frontierglobalC$ whose coarse error indicator is greater than average.}
	\For{$i \in \nat{\initRomSize}$} \Comment{For each frontier $\frontier_i$}
	\State $S_i \gets \frontier_i \cap S$ \Comment{Extract the elements of $S$
	that came from $\frontier_i$.}
	\State $\frontier_i' \gets \pRefine{\frontier}{S_i}{\tree}$ \Comment{Refine the frontier $\frontier_i$ at these spaces.}
\EndFor
\State \Return $(\frontier_1', \ldots, \frontier_{\initRomSize}')$ \Comment{Return the refined frontiers.}
\EndProcedure
\end{algorithmic}
\end{algorithm}

\begin{algorithm}
\caption{Refinement Algorithm} \label{alg:mod_refine_proc}
\hspace*{\algorithmicindent} \textbf{Input}: $\Rn$-refinement tree
	$\tree$, initial basis $\initBasis$, current frontiers $\frontier_1, \ldots, \frontier_{\initRomSize}$,
	reference solution $\trialBasisOffset$, residual function $\gResidual$,
	ROM-residual tolerance $\romTol$, and FOM-residual 
	tolerance $\tol$, the set of inactive vertices $\mathcal{I}$ (initially empty), a cutoff $\epsilon_\text{QR}$ for the conditioning technique. \\
\hspace*{\algorithmicindent} \textbf{Output}: A new set of frontiers $\frontier_1', \ldots, \frontier_{\initRomSize}'$ refined according to the input error indicators, the new set of inactive vertices $\mathcal{I}'$.
\begin{algorithmic}[1]
\Procedure{SolveModel}{$\tree, \initBasis, \frontier_1, \ldots, \frontier_{\initRomSize}, \trialBasisOffset, \romTol, \tol, \mathcal{I}, \epsilon_\text{QR}$}
\While{\textsc{True}} \Comment{Refine the basis until the specified full-order tolerance is met.}
\If{$\left(\bigsqcup_i \frontier_i \right) \setminus \treeLeaves \subset \mathcal{I}$} \Comment{If there are no viable refinement options available because all of them have been deactivated.}
	\State $\textsc{HandleFullInactiveVertices}$ \Comment{Reactivate inactive vertices, specific implementation is left to the user. See above for suggestions.}
\EndIf
	\State\label{step:mod_sieve} $\romBasisC \gets \left[\begin{array}{ccc}
\matsieve{\initBasisCol_1}{\frontier_1} & \cdots & \matsieve{\initBasisCol_{\romSize_0}}{\frontier_{\initRomSize}}  \end{array}\right]$ \Comment{Retrieve the current coarse model basis.}
\State $\romBasisC_* \gets \romBasisC_{:, \left(\bigsqcup_i \frontier_i \right)\setminus \mathcal{I}}$ \Comment{Remove columns of $\romBasisC$ which correspond to vertices in $\mathcal{I}$}
\State $\sigma^H_* \gets \left\{ 1/ \|(\romBasisC_*)_{:, \parentSpace}\|_2 \mid  \parentSpace \in \left(\bigsqcup_i \frontier_i \right)\setminus \mathcal{I}\right\}$ \Comment{Get inverse norms of all columns in $\romBasisC$}
\State $\mat{\Sigma}^H_* \gets \text{diag}(\sigma^H_*)$
\State $\romBasisC_* \gets  \romBasisC\mat{\Sigma}_*$ \Comment{Rescale the columns of $\romBasisC$}
	\State $\romStateVec\gets \textsc{SolveROM}(\gResidual, \romBasisC_*, \trialBasisOffset, \romTol)$ \Comment{Solve the system $(\romBasisC)^T \gResidual \left(\trialBasisOffset + \romBasisC_* \romStateVec\right) = 0$ from Eq.\ (\ref{galerkin_rom}).}
	\State $\stateVec \gets \romBasisC_* \romStateVec$ \Comment{Lift the result to the full-order model.}
	\If{$\|\gResidual(\stateVec)\|_2 < \tol$} \Comment{Check if the full-order residual is within the specified tolerance.}
		\State \textbf{break} \Comment{If the specified tolerance is satisfied, stop refinement.}
	\EndIf
	\State\label{step:mod_computeErrorIndicators} $\errorIndF \gets \textsc{ComputeErrorIndicators}(\romBasisC_*, \frontier_1, \ldots, \frontier_{\initRomSize})$ \Comment{Compute the error indicators in Eq.\ (\ref{eq:error_ind}).}
	\State\label{step:mod_refineFrontiers} $(\frontier_1, \ldots, \frontier_{\initRomSize}) \gets \textsc{RefineFrontiers}(\tree, \errorIndF, \frontier_1, \ldots, \frontier_{\initRomSize})$ \Comment{Use error indicators to selectively refine frontiers.}
	\State\label{step:mod_sieve} $\romBasisF \gets \left[\begin{array}{ccc}
\matsieve{\initBasisCol_1}{\frontier_1} & \cdots & \matsieve{\initBasisCol_{\romSize_0}}{\frontier_{\initRomSize}}  \end{array}\right]$ \Comment{Retrieve the current refined model basis.}
\State $\romBasisF \gets \romBasisF_{:, \left(\bigsqcup_i \frontier_i \right)\setminus \mathcal{I}}$ \Comment{Remove columns of $\romBasisF$ which correspond to vertices in $\mathcal{I}$}
\State $\sigma^h \gets \left\{ 1/ \|\romBasisF_{:, \parentSpace}\|_2 \mid  \parentSpace \in \left(\bigsqcup_i \frontier_i \right)\setminus \mathcal{I}\right\}$ \Comment{Get inverse norms of all columns in $\romBasisF$}
\State $\mat{\Sigma}^h \gets \text{diag}(\sigma^h)$
\State $\romBasisF \gets  \romBasisF \mat{\Sigma}^h$ \Comment{Rescale the columns of $\romBasisF$}
\State $(\mat{Q}, \mat{R}, \pi) \gets \textsc{CPQR}(\romBasisF)$ \Comment{Take column-pivoted QR factorization of $\romBasisF$; $\pi$ denotes the resulting column permutation.}
\State $k \gets \max_i \{ i \mid \mat{R}_{ii} \geq \epsilon_\text{QR} \}$ \Comment{Find the selection cutoff $k$ for the columns of $\romBasisF$.}
\State $\mathcal{I}_{\text{new}} \gets \pi_{k+1:\text{end}}$  \Comment{The spaces corresponding to the columns which occur after the cutoff $k$ must now be marked as inactive.}
\State $\mathcal{I} \gets \mathcal{I} \cup \mathcal{I}_{\text{new}}$\Comment{Add the above spaces to the set of inactive spaces.}
\EndWhile
\State \Return $(\frontier_1, \ldots, \frontier_{\initRomSize}, \romStateVec, \mathcal{I})$. 
\EndProcedure
\end{algorithmic}
\end{algorithm}

\section{Child grouping} \label{sec:child_grouping}

To better curtail the ROM dimension increase of a refinement, we may not want to split a parent node completely into all of its children, but rather into groups of children, as was done in the previous reduced-order model $h$-refinement paper \cite{carlberg2015adaptive}. These considerations motivate the definition of a generalized frontier,

\begin{definition}[generalized frontier]
A \textbf{generalized frontier} $\genFrontier$ (i.e., a frontier with child grouping) of a refinement tree $\tree$ is an orthogonal decomposition of $\mathbb{R}^\stateVecSize$ such that there exists a frontier $\frontier$ of $\tree$ satisfying
\begin{equation} \label{eq:genfrontier}
\frontier \preceq \genFrontier \,,
\end{equation}
with the property that the ancestor map $\ancestor : \frontier \longrightarrow \genFrontier$ satisfies
\begin{equation} \label{eq:genfrontier2}
\ancestor(\parentSpace) \subset \parent_\tree(\parentSpace), \qquad \forall \parentSpace \in \frontier \,,
\end{equation}
where $\parent_\tree(\parentSpace)$ denotes the parent of $\parentSpace$ in $\tree$. 
\end{definition}
The concept of a generalized frontier provides a useful range of resolutions between the coarse frontier $\frontier$ and the refinement $\refine{\frontier}$. In practice, a generalized frontier $\genFrontier$ can be represented on a computer by storing the frontier $\frontier$ in Eq.\ (\ref{eq:genfrontier}) together with the induced ancestor map $\ancestor : \frontier \longrightarrow \genFrontier$. Note that this representation $(\frontier, \ancestor)$ of a generalized frontier may not be unique without additional constraints, but every generalized frontier can be represented as such. To guarantee uniqueness of the frontier $\frontier$, we use the convention that
\begin{equation} \label{eq:canonFrontier}
(\frontier \cap \genFrontier) \setminus \treeLeaves = \emptyset \,.
\end{equation}
We give this $\frontier$ above a definition:
\begin{definition}[Full refinement of a generalized frontier]
The \textbf{full refinement} of a generalized frontier $\genFrontier$ is any frontier $\frontier$ which satisfies Eqs.\ (\ref{eq:genfrontier}), (\ref{eq:genfrontier2}), and (\ref{eq:canonFrontier}). We denote the full refinement of a generalized frontier $\genFrontier$ as $\refine{(\genFrontier)}$.
\end{definition}

\begin{restatable}[Uniqueness of full refinements of generalized frontiers]{proposition}{propcharacterfrontiers} \label{prop:character_frontiers}
The full refinement of a generalized frontier exists and is unique.
\end{restatable}

This definition now allows for straightforward generalization of the refinement algorithms \ref{alg:err} and \ref{alg:refine_proc} to generalized frontiers by simply replacing the frontiers $\frontier_1, ..., \frontier_k$ with generalized frontiers $\genFrontier_1, ..., \genFrontier_k$. On the other hand, to make use of the additional freedom that generalized frontiers give us in performing refinement, we modify algorithm \ref{alg:refine_frontiers}, which actually performs the frontier refinement. This modification is a direct extension of the \emph{child grouping} mechanism from the original ROM $h$-refinement paper \cite{carlberg2015adaptive}. 

To begin, note that $\genFrontier$ can be thought of as a special grouping of the spaces in $\refine{(\genFrontier)}$. From Eq.\ (\ref{eq:genfrontier}) above, every space in $\genFrontier$ is a sum of a group of spaces in $\refine{(\genFrontier)}$, with the property, imparted from Eq.\ (\ref{eq:genfrontier2}), that groupped spaces must have the same parent in the tree $\tree$. The generalized frontier structure is therefore directly analogous to \emph{child grouping} from ROM $h$-refinement. 

Just like our refinement algorithm \ref{alg:refine_proc} does not necessarily refine a frontier $\frontier$ to its full refinement $\refine{\frontier}$, we also do not necessarily have to split the spaces $\parentSpace \in \genFrontier$ we select for refinement into all of their individual consituents $\ancestor^{-1}(\parentSpace)$ in $\refine{(\genFrontier)}$. Instead, we opt to use the dual-weighted error residual indicators $\errorIndF$ in Eq.\ (\ref{eq:error_ind}) to determine a more conservative refinement.

This is done as follows: suppose we select $\parentSpace \in \genFrontier$ for refinement, i.e., the coarse error indicator $\errorIndC_\parentSpace$ is larger than the average of the coarse error indicators for the spaces in $\genFrontier$. The coarse error indicator can then be decomposed in terms of the fine error indicators for constituent spaces of $\parentSpace$ (see Eq.\ (\ref{eq:splitCoarseErr})),
\begin{equation}
\errorIndC_{\parentSpace} = \sum_{\childSpace \in \ancestor^{-1}(\parentSpace)} \errorIndF_{\childSpace} \,.
\end{equation}
Since these error indicators give us a rough estimate of the reduction in error achieved by refining fully to a specific vector space in $\ancestor^{-1}(\parentSpace)$, one principled way of refining $\parentSpace$ into subgroups $G_i \subset \ancestor^{-1}(\parentSpace)$ of spaces in $\ancestor^{-1}(\parentSpace)$ would be to split $\ancestor^{-1}(\parentSpace)$ in such a way that the subgroups $G_i$ have roughly equal cumulative error indicators $\sum_{\childSpace \in G_i} \errorIndF_{\childSpace}$. More precisely, given some splitting factor $0 < \alpha < 1$, we would like to find a decomposition of $\ancestor^{-1}(\parentSpace)$ into a minimal number of disjoint subgroups $G_i$ with the property that each subgroup has a cummulative error that is at most a fraction $\alpha$ of the total cummulative error $\errorIndC_{\parentSpace}$ of $\ancestor^{-1}(\parentSpace)$, i.e.,
\begin{equation}
\sum_{\childSpace \in G_i} \errorIndF_{\childSpace} \leq \alpha \errorIndC_{\parentSpace} \,.
\end{equation}
This is the well-known \emph{bin-packing problem}, which is $\mathcal{NP}$-hard. We use the inexpensive and easily implemented greedy \emph{first-fit algorithm} to generate the sets $G_i$. Greedy first-fit achieves the optimal number of bins within a factor of $2$, although there are other common approximation algorithms one may use if so inclined. Note that smaller values of $\alpha$ correspond to more aggressive splitting. Once the sets $G_i$ have been calculated, we define the corresponding spaces $\mathbb{G}_i$ as
\begin{equation}
\mathbb{G}_i \equiv \sum_{\childSpace \in G_i} \childSpace \,.
\end{equation}
From the properties of the ancestor map $\ancestor$, it is easy to see that these spaces form an orthogonal decomposition $\decomp_\parentSpace = \{ \mathbb{G}_i \}_i$ of the parent space $\parentSpace$. Hence, if $\parentSpace_1, ..., \parentSpace_b \in \genFrontier$ are marked for refinement, and $\decomp_{\parentSpace_1}, ..., \decomp_{\parentSpace_b}$ are the corresponding decompositions of these spaces computed with the greedy binning strategy above, then our refinement step becomes
\begin{equation}
\genFrontier' \gets \pRefineD{\genFrontier}{\decomp_{\parentSpace_1}, \ldots, \decomp_{\parentSpace_b}} \,.
\end{equation}
Note $\genFrontier'$ is an orthogonal decomposition by remark (\ref{re:refinement_still_frontier}) and that the full refinement $\refine{(\genFrontier)}$ of $\genFrontier$ serves as the frontier $\frontier$ in Eqs.\ (\ref{eq:genfrontier}) and (\ref{eq:genfrontier2}) for the orthogonal decomposition $\genFrontier'$. This makes $\genFrontier'$ a generalized frontier. The pseudocode for computing the refinement of $\genFrontier$ is given in algorithm \ref{alg:refine_gen_frontiers}.

\begin{algorithm}
\caption{Computation of Refined Generalized Frontiers} \label{alg:refine_gen_frontiers}
\hspace*{\algorithmicindent} \textbf{Input}: The $\Rn$-refinement tree $\tree$, the fine error indicators $\errorIndF$, the current generalized frontiers $\genFrontier_1, \ldots, \genFrontier_{\romSize_0}$.  \\
\hspace*{\algorithmicindent} \textbf{Output}: A new set of generalized frontiers $\genFrontier_1', \ldots, \genFrontier_{\initRomSize}'$ refined according to the input error indicators.
\begin{algorithmic}[1]
\Procedure{RefineFrontiers}{$\tree, \errorIndF, \genFrontier_1, \ldots, \genFrontier_{\initRomSize}$}
\State $\frontierglobalC \gets \bigsqcup_i \genFrontier_i$
\State $\frontierglobalF \gets \bigsqcup_i \refine{\genFrontier}_i$
\State $\ancestor \gets \textsc{GetGlobalAncestorMap}\left(\frontierglobalF, \frontierglobalC\right)$ \Comment{Compute the map in Eq.\ (\ref{eq:global_ancestor_map}) sending every space to its ancestor.}
\State $\prolong \gets \textsc{ComputeProlongationOperator}\left(\frontierglobalC, \frontierglobalF\right)$ \Comment{Compute prolongation operator using Eq.\ (\ref{eq:prolong_def}).}
\State $\errorIndC = \errorIndF \prolong$ \Comment{Compute coarse error indicators using Eq.\ (\ref{eq:coarse_err_prolong}).}
\State $\eta \gets \frac{1}{|\frontierglobalC|} \sum_{\parentSpace \in \frontierglobalC} \errorIndC_\parentSpace$ \Comment{Compute the average of the coarse error indicators.}
\State $S \gets \{ \parentSpace \in \frontierglobalC \mid \errorIndC_\parentSpace \geq \eta \}$ \Comment{Select the spaces in $\frontierglobalC$ whose coarse error indicator is greater than average.}
	\For{$i \in \nat{\initRomSize}$} \Comment{For each frontier $\frontier_i$}
	\State $S_i \gets \frontier_i \cap S$ \Comment{Extract the elements of $S$
	that came from $\frontier_i$.}
	\State $\genFrontier_i' \gets \genFrontier_i$
	\For{$\parentSpace \in S_i$} \Comment{For each space $\parentSpace$ to be refined.}
	\State $\{ G_1, ..., G_b \} \gets \textsc{GreedyBinPacking}(\errorIndF_{\ancestor^{-1}(\parentSpace)}, \alpha \errorIndC_\parentSpace)$ \Comment{Perform first-fit algorithm with item sizes $\errorIndF_{\ancestor^{-1}(\parentSpace)}$ and bin size $\alpha \errorIndC_\parentSpace$. $G_1, ..., G_b$ are the resulting bin contents.}
	\For{$j \in \nat{b}$}
		\State $\mathbb{G}_j \gets \sum_{\childSpace \in G_j} \childSpace$ \Comment{Compute subspaces corresponding to bin contents.}
	\EndFor
	\State $\decomp_\parentSpace \gets \{ \mathbb{G}_1, ..., \mathbb{G}_b \}$ \Comment{Construct the orthogonal decomposition of $\parentSpace$.}
	\State $\genFrontier_i' \gets \pRefineD{\genFrontier_i'}{\decomp_\parentSpace}$ \Comment{Split the space $\parentSpace$ on $\genFrontier$ by using the decomposition $\decomp_\parentSpace$.}
	\EndFor
\EndFor
\State \Return $(\genFrontier_1', \ldots, \genFrontier_{\initRomSize}')$ \Comment{Return the refined frontiers.}
\EndProcedure
\end{algorithmic}
\end{algorithm}

\section{Proofs of propositions} \label{sec:proofs}

\propsubset*

\begin{proof}
If $\childSpace$ is a descendant of $\parentSpace$, then recursive application
	of the property (2) immediately gives that $\childSpace \subset
	\parentSpace$. Conversely, consider the case where $\childSpace \subset
	\parentSpace$. Let $\vecSpace{A} \in \treeNodes$ be the first common
	ancestor of $\childSpace$ and $\parentSpace$. If $\vecSpace{A} =
	\parentSpace$, then the desired result holds automatically. If $\vecSpace{A}
	= \childSpace$, then the forward direction implies that $\parentSpace
	\subset \childSpace$ and thus $\parentSpace = \childSpace$ and the desired
	result holds. Otherwise, suppose $\vecSpace{A} \neq \childSpace$ and
	$\vecSpace{A} \neq \parentSpace$. Because $\vecSpace{A}$ is the first common ancestor of $\childSpace$ and $\parentSpace$, $\childSpace$ and $\parentSpace$ must be descendant from different children $\childSpace'$ and $\parentSpace'$ of $\vecSpace{A}$ respectively. Property (3) implies $\childSpace' \perp \parentSpace'$ and the forward direction of the proof implies that $\childSpace \subset \childSpace'$ and $\parentSpace \subset \parentSpace'$. Hence, $\childSpace \perp \parentSpace$. Furthermore, since every vertex in the tree has a leaf as an ancestor, the forward direction of the proof combined with property (4) implies that all spaces in $\treeNodes$ have dimension at least $1$. This fact, combined with $\childSpace \perp \parentSpace$ and the assumption $\childSpace \subset \parentSpace$ gives a contradiction, proving the backward direction.
\end{proof}

\propdesctree*

\begin{proof}
For the forward direction, if the assumption holds, then the first common
	ancestor of $\parentSpace$ and $\childSpace$ is neither $\parentSpace$ nor
	$\childSpace$. The desired result was now proved as part of the proof of
	proposition (\ref{subset_proposition}). For the backward direction, since
	all spaces in $\treeNodes$ have at least dimension $1$, so $\childSpace
	\perp \parentSpace$ implies that neither $\childSpace$ nor $\parentSpace$ is
	a subset of the other. The desired result follows from proposition
	(\ref{subset_proposition}).
\end{proof}

\propcharfrontiers*

\begin{proof}
Suppose $\frontier \subset \treeNodes$ is a frontier and suppose for
	contradiction that there is a leaf $\leafSpace$ that is not descendant from any space in $\frontier$. Since $\leafSpace$ is a leaf, this means that $\leafSpace$ is incomparable with every element in $\frontier$, which by (\ref{incomparable_prop}) means that $\leafSpace$ is orthogonal to every space in $\frontier$. Hence, $\frontier$ cannot be a decomposition of $\rootSpace$.

Conversely, suppose $\frontier \subset \treeNodes$ is a set of vertices such that every leaf in $\leafSpace$ is descendant from exactly one element in $\frontier$. For each $\parentSpace_i \in \frontier$ we denote $\leafSpace_{ij}$ as the leaf spaces descendant from $\parentSpace_i$. Recursive application of properties (1) and (2) of a refinement tree tells us that
\begin{equation}
\parentSpace_i = \sum_j \leafSpace_{ij} \,.
\end{equation}
Therefore, since $\parentSpace_i$ is incomparable with all $\leafSpace_{rj}$ such that $r \neq i$, it follows that $\parentSpace_i \perp \parentSpace_r$ for $i \neq r$. Furthermore, since the leaf spaces $\leafSpace_{ij}$ sum to $\rootSpace$, the spaces $\parentSpace_i$ must also sum to $\rootSpace$. Hence, $\frontier$ is an orthogonal decomposition of $\rootSpace$.
\end{proof}

\propancestor*

\begin{proof}
The existence of such a map is trivial, it follows from the definition of the partial order $\preceq$. The uniqueness of this map follows from the fact that the elements of $\frontier_2$ must be orthogonal, and hence $\childSpace \in \frontier_1$ can be a subspace of at most one of the spaces in $\frontier_2$. To prove Eq.\ (\ref{frontier_decomposition}), note that, by the above, we must always have $\sum_{\childSpace \in \ancestor^{-1}(\parentSpace)} \childSpace \subset \parentSpace$. So, suppose for contradiction that $\sum_{\childSpace \in \ancestor^{-1}(\parentSpace)} \childSpace \subsetneq \parentSpace$. Consider the space
\begin{equation}
\childSpace' = \left(\sum_{\childSpace \in \ancestor^{-1}(\parentSpace)} \childSpace \right)^\perp \cap \parentSpace \,.
\end{equation}
This space cannot be trivial -- by assumption, there exists a nonzero $v \in
	\parentSpace \setminus \sum_{\childSpace \in \ancestor^{-1}(\parentSpace)}
	\childSpace$ and projecting out the space $\sum_{\childSpace \in
	\ancestor^{-1}(\parentSpace)} \childSpace$ yields a nonzero vector both in
	$\parentSpace$ and in $\left(\sum_{\childSpace \in
	\ancestor^{-1}(\parentSpace)} \childSpace \right)^\perp$. Furthermore,
	$\childSpace' \perp \childSpace$ for $\childSpace \in
	\ancestor^{-1}(\parentSpace)$ by construction, and $\childSpace' \perp
	\childSpace$ for all $\childSpace \in \decomp_1$ such that
	$\ancestor(\childSpace) \neq \parentSpace$ since $\childSpace \subset
	\ancestor(\childSpace) $, $\ancestor(\childSpace) \perp \parentSpace$, and
	$\childSpace' \subset \parentSpace$. Therefore, there exists a nonzero
	vector in $\childSpace'$ that is not contained in $\sum_{\childSpace \in \decomp_1} \childSpace$ and hence $\decomp_1$ cannot be an orthogonal decomposition of $\rootSpace$, a contradiction.
\end{proof}

\propmeet*

\begin{proof}
To begin, we prove that the constituents of $\bigwedge_i \frontier_i$ are orthogonal to one another. Let $\childSpace \neq \childSpace'$ be elements of $\bigwedge_i \frontier_i$. Then, we have
\begin{equation}
\begin{split}
\childSpace = \parentSpace_1 \cap \parentSpace_2 \cap \ldots \cap \parentSpace_m \qquad \parentSpace_i \in \frontier_i \\
\childSpace' = \parentSpace_1' \cap \parentSpace_2' \cap \ldots \cap \parentSpace_m' \qquad \parentSpace_i' \in \frontier_i \\
\end{split}
\end{equation}
Because $\childSpace \neq \childSpace'$ we must have $\parentSpace_j \neq \parentSpace_j'$ for some $j$. Then $\parentSpace_j \perp \parentSpace_j'$ by the properties of $\frontier_j$. Thus, since $\childSpace \subset \parentSpace_j$ and $\childSpace' \subset \parentSpace_j'$, we have that $\childSpace \perp \childSpace'$.

Now, to prove that $\bigwedge_i \frontier_i$ spans $\rootSpace$: because
	$\treeLeaves$ spans $\rootSpace$, it suffices to show that, for each $\leafSpace \in \treeLeaves$, there exists $\childSpace \in \bigwedge_i \frontier_i$ such that $\leafSpace \subset \childSpace$, where $\treeLeaves$ are the leaves of
	$\tree$. Because $\treeLeaves \preceq \frontier_i$ for all $i$, note that each $\leafSpace \in \treeLeaves$ satisfies $\leafSpace \subset \childSpace_i$ for some $\childSpace_i \in \frontier_i$. Therefore,
$\leafSpace \subset \childSpace_1 \cap \childSpace_2 \cap \ldots \cap \childSpace_m \in \bigwedge_i \frontier_i$ and, since this holds for all $\leafSpace \in \treeLeaves$, we have that $\bigwedge_i \frontier_i$ spans $\rootSpace$.

Finally, we need to prove that 
\begin{equation}
\parentSpace_1 \cap \parentSpace_2 \cap \ldots \cap \parentSpace_m \in \treeNodes 
\end{equation}
for all $\parentSpace_i \in \frontier_i$ such that the above intersection is nontrivial. It suffices to prove that for any two $\parentSpace_1, \parentSpace_2 \in \treeNodes$ that either $\parentSpace_1 \cap \parentSpace_2 \in \treeNodes$ or $\parentSpace_1 \cap \parentSpace_2 = 0$. Corollary (\ref{incomparable_prop}) tells us that either one of $\parentSpace_1, \parentSpace_2$ is a subspace of the other, or $\parentSpace_1 \perp \parentSpace_2$. In the first case we have either $\parentSpace_1 \cap \parentSpace_2 = \parentSpace_1$ or $\parentSpace_1 \cap \parentSpace_2 = \parentSpace_2$. Either way, $\parentSpace_1 \cap \parentSpace_2 \in \treeNodes$. In the second case we have $\parentSpace_1 \cap \parentSpace_2 = 0$. This proves the desired result.
\end{proof}

\propmeetsub*

\begin{proof}
We proved in proposition (\ref{prop:meet}) that for any $\parentSpace_1, \parentSpace_2 \in \treeNodes$, $\parentSpace_1 \cap \parentSpace_2$ is equal to either $\parentSpace_1$, $\parentSpace_2$, or $0$. Then for, $\parentSpace_1 \cap \parentSpace_2 \cap \ldots \cap \parentSpace_m \in \bigwedge_i \frontier_i$, this means that $\parentSpace_1 \cap \parentSpace_2 \cap \ldots \cap \parentSpace_m = \parentSpace_j$ for some $j$ as the intersection is nontrivial.
\end{proof}

\proplargestlowerbound*

\begin{proof}
The first statement that $\bigwedge_i \frontier_i \preceq \frontier_j$ follows
	trivially from the fact that $\parentSpace_1 \cap \parentSpace_2 \cap \ldots
	\cap \parentSpace_m \subset \parentSpace_j$. For the second statement, let
	$\childSpace \in \frontierH$. Then, by assumption, $\childSpace \subset
	\parentSpace_i$ for some $\parentSpace_i \in \frontier_i$ for all $i \in \nat{m}$. Ergo, $\childSpace \subset \parentSpace_1 \cap \parentSpace_2 \cap \ldots \parentSpace_m \in \bigwedge_i \frontier_i$. Therefore, $\frontierH \preceq \bigwedge_i \frontier_i$.
\end{proof}

\propnodescendants*

\begin{proof}
For contradiction, suppose there exists $\parentSpace \in \bigwedge_i
	\frontier_i$ that has a descendant $\childSpace \in \frontier_j$ such that $\childSpace \subsetneq \parentSpace$. Then, we note that $\bigwedge_i \frontier_i \not\preceq \frontier_j$ since $\parentSpace$ cannot be a subspace of any element in $\frontier_j$ as otherwise $\childSpace$ would be a subspace of another space in $\frontier_j$, violating orthogonality. The result $\bigwedge_i \frontier_i \not\preceq \frontier_j$ then contradicts proposition (\ref{prop:lower_bound}).

Conversely, suppose $\parentSpace \in \frontier_j$ has no descendants in
	$\left(\bigcup_i \frontier_i\right) \setminus \parentSpace$. By proposition (\ref{prop:lower_bound}) above, $\bigwedge_i \frontier_i \preceq \frontier_j$, which means that there must exist $\childSpace \in \bigwedge_i \frontier_i \subset \bigcup_i \frontier_i $ such that $\childSpace \subset \parentSpace$. By assumption, we must therefore have $\childSpace = \parentSpace$, or else $\parentSpace$ would have a descendant in $\left(\bigcup_i \frontier_i\right) \setminus \parentSpace$. Thus, $\parentSpace \in \bigwedge_i \frontier_i$.
\end{proof}

\proponbelowabove*

\begin{proof}
Let $\leafSpace$ be a leaf descendant from $\parentSpace$. By the
	characterization of frontiers given in proposition
	(\ref{prop:character_frontiers}), every leaf $\leafSpace \in \treeNodes$ is
	descendant from exactly one element $\childSpace \in \frontier$. Both
	$\parentSpace$ and $\childSpace$ must then be on the same path from the root
	$\rootSpace$ to the leaf $\leafSpace$. This means one is descendant from the
	other, which proves that either $\parentSpace \subset \childSpace$ or
	$\childSpace \subset \parentSpace$ by proposition
	(\ref{subset_proposition}). To prove that (2) and (3) are exclusive, suppose
	there existed $\childSpace_1, \childSpace_2 \in \frontier$ such that
	$\parentSpace \subsetneq \childSpace_1$ and $\childSpace_2 \subsetneq
	\parentSpace$. This would imply that $\childSpace_2 \subsetneq
	\childSpace_1$, which is impossible because then $\childSpace_1 \neq \childSpace_2$, and hence $\childSpace_1, \childSpace_2$ are nontrivial subspaces such that $\childSpace_2 \perp \childSpace_1$ by the definition of a frontier, which contradicts $\childSpace_2 \subsetneq \childSpace_1$. Exclusion between (1) and (2) and between (1) and (3) can be proved similarly.
\end{proof}

\propcharacterfrontiers*

\begin{proof}
For existence, let $\genFrontier$ be a generalized frontier with a frontier $\frontier$ such that $\frontier \preceq \genFrontier$ and $\ancestor(\parentSpace) \subset \parent_\tree(\parentSpace)$ for all $\parentSpace \in \frontier$, where $\ancestor$ is the ancestor map. Suppose $(\frontier \cap \genFrontier) \setminus \treeLeaves \neq \emptyset$ and let $\parentSpace \in (\frontier \cap \genFrontier) \setminus \treeLeaves$. Consider $\frontier' \equiv \pRefine{\frontier}{\parentSpace}{T}$. Note $\frontier' \preceq \genFrontier$ since $\frontier \preceq \genFrontier$. And since $\parentSpace \not\in \treeLeaves$, $\parentSpace \not\in \frontier'$. Furthermore, all elements of $\frontier' \setminus \frontier$ are children of $\parentSpace$ which is in $\genFrontier$, and hence, these children cannot be in $\genFrontier$, as $\genFrontier$ is an orthogonal decomposition. Thus, $|(\frontier' \cap \genFrontier) \setminus \treeLeaves| < |(\frontier \cap \genFrontier) \setminus \treeLeaves|$. Induction now shows that the full-refinement exists. 

For uniqueness, suppose there are two distinct $\frontier_1$ and $\frontier_2$ satisfying the full refinement property, with corresponding ancestor maps $\ancestor_1 : \frontier_1 \longrightarrow \genFrontier$ and $\ancestor_2 : \frontier_2 \longrightarrow \genFrontier$. Without loss of generality, assume $\frontier_1 \setminus \frontier_2$ is nonempty. Let $\parentSpace \in \frontier_1 \setminus \frontier_2$. By proposition (\ref{prop:onbelowabove}), $\parentSpace$ must be either above, on, or below $\frontier_2$. It cannot be on $\frontier_2$, so that leaves two choices:
\begin{enumerate}
\item $\parentSpace$ is above $\frontier_2$: In this case, there exists $\childSpace \in \frontier_2$ descendant from $\parentSpace$, such that $\parent_\tree(\childSpace) \subset \parentSpace$. But since $\childSpace \subset \parentSpace$, we must have $\ancestor_2(\childSpace) = \ancestor_1(\parentSpace)$, as otherwise, since $\genFrontier$ is an orthogonal decomposition, we would have $\ancestor_2(\childSpace) \cap \ancestor_1(\parentSpace) = \emptyset$, but this gives a contradiction with $\childSpace \subset \parentSpace \subset \ancestor_1(\parentSpace)$ and $\childSpace \subset \ancestor_2(\childSpace)$. But then, we would ahve that $\ancestor_1(\parentSpace) = \ancestor_2(\childSpace) \subset \parent_\tree(\childSpace) \subset \parentSpace$, which tells us that $\parentSpace = \ancestor_1(\parentSpace)$. But this means that $\parentSpace \in \genFrontier$ and $\parentSpace \in \frontier_1$ and $\parentSpace$ is not a leaf, which contradicts our assumption.
\item $\parentSpace$ is below $\frontier_2$: In this case, $\parentSpace$ is descendant from some $\childSpace \in \frontier_2$. The argument above can now be run in reverse.
\end{enumerate}
Thus, the full refinement must be unique.
\end{proof}

\bibliography{wileyNJD-AMA}

\begin{thebibliography}{10}

\bibitem{amsallem2009method}
{\sc D.~Amsallem, J.~Cortial, K.~Carlberg, and C.~Farhat}, {\em A method for
  interpolating on manifolds structural dynamics reduced-order models},
  International journal for numerical methods in engineering, 80 (2009),
  pp.~1241--1258.

\bibitem{amsallem2008interpolation}
{\sc D.~Amsallem and C.~Farhat}, {\em Interpolation method for adapting
  reduced-order models and application to aeroelasticity}, AIAA journal, 46
  (2008), pp.~1803--1813.

\bibitem{amsallem2012nonlinear}
{\sc D.~Amsallem, M.~J. Zahr, and C.~Farhat}, {\em Nonlinear model order
  reduction based on local reduced-order bases}, International Journal for
  Numerical Methods in Engineering, 92 (2012), pp.~891--916.

\bibitem{arian2000trust}
{\sc E.~Arian, M.~Fahl, and E.~W. Sachs}, {\em Trust-region proper orthogonal
  decomposition for flow control}, tech. rep., Institute for Computer
  Applications in Science and Engineering, 2000.

\bibitem{astrid2008missing}
{\sc P.~Astrid, S.~Weiland, K.~Willcox, and T.~Backx}, {\em Missing point
  estimation in models described by proper orthogonal decomposition}, IEEE
  Transactions on Automatic Control, 53 (2008), pp.~2237--2251.

\bibitem{barone2009stable}
{\sc M.~F. Barone, I.~Kalashnikova, D.~J. Segalman, and H.~K. Thornquist}, {\em
  Stable galerkin reduced order models for linearized compressible flow},
  Journal of Computational Physics, 228 (2009), pp.~1932--1946.

\bibitem{barrault2004eim}
{\sc M.~Barrault, Y.~Maday, N.~C. Nguyen, and A.~T. Patera}, {\em An `empirical
  interpolation' method: application to efficient reduced-basis discretization
  of partial differential equations}, Comptes Rendus Math\'ematique Acad\'emie
  des Sciences, 339 (2004), pp.~667--672.

\bibitem{benner2013survey}
{\sc P.~Benner, S.~Gugercin, and K.~Willcox}, {\em A survey of model reduction
  methods for parametric systems},  (2013).

\bibitem{bos2004als}
{\sc R.~Bos, X.~Bombois, and P.~Van~den Hof}, {\em Accelerating large-scale
  non-linear models for monitoring and control using spatial and temporal
  correlations}, Proceedings of the American Control Conference, 4 (2004),
  pp.~3705--3710.

\bibitem{bui2008model}
{\sc T.~Bui-Thanh, K.~Willcox, and O.~Ghattas}, {\em {Model reduction for
  large-scale systems with high-dimensional parametric input space}}, SIAM
  Journal on Scientific Computing, 30 (2008), pp.~3270--3288.

\bibitem{bui2008parametric}
\leavevmode\vrule height 2pt depth -1.6pt width 23pt, {\em Parametric
  reduced-order models for probabilistic analysis of unsteady aerodynamic
  applications}, AIAA Journal, 46 (2008), pp.~2520--2529.

\bibitem{carlberg2015adaptive}
{\sc K.~Carlberg}, {\em Adaptive h-refinement for reduced-order models},
  International Journal for Numerical Methods in Engineering, 102 (2015),
  pp.~1192--1210.

\bibitem{carlbergGalDiscOpt}
{\sc K.~Carlberg, M.~Barone, and H.~Antil}, {\em Galerkin v.\ least-squares
  {P}etrov--{G}alerkin projection in nonlinear model reduction}, Journal of
  Computational Physics, 330 (2017), pp.~693--734.

\bibitem{CarlbergGappy}
{\sc K.~Carlberg, C.~Bou-Mosleh, and C.~Farhat}, {\em {Efficient non-linear
  model reduction via a least-squares Petrov--Galerkin projection and
  compressive tensor approximations}}, International Journal for Numerical
  Methods in Engineering, 86 (2011), pp.~155--181.

\bibitem{carlberg2009adaptive}
{\sc K.~Carlberg and C.~Farhat}, {\em An adaptive {POD}-{K}rylov reduced-order
  model for structural optimization}, 8th World Congress on Structural and
  Multidisciplinary Optimization, Lisbon, Portugal,  (2009).

\bibitem{carlberg2013gnat}
{\sc K.~Carlberg, C.~Farhat, J.~Cortial, and D.~Amsallem}, {\em The {GNAT}
  method for nonlinear model reduction: effective implementation and
  application to computational fluid dynamics and turbulent flows}, Journal of
  Computational Physics, 242 (2013), pp.~623--647.

\bibitem{carlbergKrylov2015}
{\sc K.~Carlberg, V.~Forstall, and R.~Tuminaro}, {\em Krylov-subspace recycling
  via the {POD}-augmented conjugate-gradient algorithm}, SIAM Journal on Matrix
  Analysis and Applications, 37 (2016), pp.~1304--1336.

\bibitem{chaturantabut2010nonlinear}
{\sc S.~Chaturantabut and D.~C. Sorensen}, {\em Nonlinear model reduction via
  discrete empirical interpolation}, SIAM Journal on Scientific Computing, 32
  (2010), pp.~2737--2764.

\bibitem{dihlmann2011model}
{\sc M.~Dihlmann, M.~Drohmann, and B.~Haasdonk}, {\em Model reduction of
  parametrized evolution problems using the reduced basis method with adaptive
  time-partitioning}, Proc. of ADMOS, 2011 (2011), p.~64.

\bibitem{drohmann2011adaptive}
{\sc M.~Drohmann, B.~Haasdonk, and M.~Ohlberger}, {\em Adaptive reduced basis
  methods for nonlinear convection--diffusion equations}, in Finite Volumes for
  Complex Applications VI Problems \& Perspectives, Springer, 2011,
  pp.~369--377.

\bibitem{drohmannEOI}
\leavevmode\vrule height 2pt depth -1.6pt width 23pt, {\em Reduced basis
  approximation for nonlinear parametrized evolution equations based on
  empirical operator interpolation}, SIAM Journal on Scientific Computing, 34
  (2012), pp.~A937--A969.

\bibitem{efendiev2016online}
{\sc Y.~Efendiev, E.~Gildin, and Y.~Yang}, {\em Online adaptive local-global
  model reduction for flows in heterogeneous porous media}, Computation, 4
  (2016), p.~22.

\bibitem{eftang2010hp}
{\sc J.~L. Eftang, A.~T. Patera, and E.~M. R{\o}nquist}, {\em An" hp" certified
  reduced basis method for parametrized elliptic partial differential
  equations}, SIAM Journal on Scientific Computing, 32 (2010), pp.~3170--3200.

\bibitem{sirovichOrigGappy}
{\sc R.~Everson and L.~Sirovich}, {\em {K}arhunen--{L}o\`{e}ve procedure for
  gappy data}, Journal of the Optical Society of America A, 12 (1995),
  pp.~1657--1664.

\bibitem{galbally2009non}
{\sc D.~Galbally, K.~Fidkowski, K.~Willcox, and O.~Ghattas}, {\em {Non-linear
  model reduction for uncertainty quantification in large-scale inverse
  problems}}, International Journal for Numerical Methods in Engineering, 81
  (2009), pp.~1581--1608.

\bibitem{haasdonk2011training}
{\sc B.~Haasdonk, M.~Dihlmann, and M.~Ohlberger}, {\em A training set and
  multiple bases generation approach for parameterized model reduction based on
  adaptive grids in parameter space}, Mathematical and Computer Modelling of
  Dynamical Systems, 17 (2011), pp.~423--442.

\bibitem{hesthaven2015certified}
{\sc J.~S. Hesthaven, G.~Rozza, and B.~Stamm}, {\em Certified reduced basis
  methods for parametrized partial differential equations}, Springer, 2015.

\bibitem{kim2009skipping}
{\sc T.~Kim and D.~L. James}, {\em Skipping steps in deformable simulation with
  online model reduction}, in ACM transactions on graphics (TOG), vol.~28, ACM,
  2009, p.~123.

\bibitem{legresley2006application}
{\sc P.~A. LeGresley}, {\em Application of proper orthogonal decomposition
  ({POD}) to design decomposition methods}, Citeseer, 2006.

\bibitem{ohlberger2015error}
{\sc M.~Ohlberger and F.~Schindler}, {\em Error control for the localized
  reduced basis multiscale method with adaptive on-line enrichment}, SIAM
  Journal on Scientific Computing, 37 (2015), pp.~A2865--A2895.

\bibitem{peherstorfer2018model}
{\sc B.~Peherstorfer}, {\em Model reduction for transport-dominated problems
  via online adaptive bases and adaptive sampling}, arXiv preprint
  arXiv:1812.02094,  (2018).

\bibitem{peherstorfer2014localized}
{\sc B.~Peherstorfer, D.~Butnaru, K.~Willcox, and H.-J. Bungartz}, {\em
  Localized discrete empirical interpolation method}, SIAM Journal on
  Scientific Computing, 36 (2014), pp.~A168--A192.

\bibitem{peherstorfer2015online}
{\sc B.~Peherstorfer and K.~Willcox}, {\em Online adaptive model reduction for
  nonlinear systems via low-rank updates}, SIAM Journal on Scientific
  Computing, 37 (2015), pp.~A2123--A2150.

\bibitem{peherstorfer2016dynamic}
\leavevmode\vrule height 2pt depth -1.6pt width 23pt, {\em Dynamic data-driven
  model reduction: adapting reduced models from incomplete data}, Advanced
  Modeling and Simulation in Engineering Sciences, 3 (2016), p.~11.

\bibitem{peng2014online}
{\sc L.~Peng and K.~Mohseni}, {\em An online manifold learning approach for
  model reduction of dynamical systems}, SIAM Journal on Numerical Analysis, 52
  (2014), pp.~1928--1952.

\bibitem{rozza2007reduced}
{\sc G.~Rozza, D.~B.~P. Huynh, and A.~T. Patera}, {\em Reduced basis
  approximation and a posteriori error estimation for affinely parametrized
  elliptic coercive partial differential equations}, Archives of Computational
  Methods in Engineering, 15 (2007), p.~1.

\bibitem{ryckelynck2005priori}
{\sc D.~Ryckelynck}, {\em A priori hyperreduction method: an adaptive
  approach}, Journal of computational physics, 202 (2005), pp.~346--366.

\bibitem{teng2015subspace}
{\sc Y.~Teng, M.~Meyer, T.~DeRose, and T.~Kim}, {\em Subspace condensation:
  full space adaptivity for subspace deformations}, ACM Transactions on
  Graphics (TOG), 34 (2015), p.~76.

\bibitem{weickum2009multi}
{\sc G.~Weickum, M.~Eldred, and K.~Maute}, {\em A multi-point reduced-order
  modeling approach of transient structural dynamics with application to robust
  design optimization}, Structural and Multidisciplinary Optimization, 38
  (2009), p.~599.

\bibitem{white2003trajectory}
{\sc J.~K. White}, {\em A trajectory piecewise-linear approach to model order
  reduction of nonlinear dynamical systems}, PhD thesis, Massachusetts
  Institute of Technology, 2003.

\bibitem{yang2017online}
{\sc Y.~Yang, E.~Gildin, Y.~Efendiev, V.~Calo, et~al.}, {\em Online adaptive
  {POD}-{DEIM} model reduction for fast simulation of flows in heterogeneous
  media}, in SPE Reservoir Simulation Conference, Society of Petroleum
  Engineers, 2017.

\bibitem{zimmermann2018geometric}
{\sc R.~Zimmermann, B.~Peherstorfer, and K.~Willcox}, {\em Geometric subspace
  updates with applications to online adaptive nonlinear model reduction}, SIAM
  Journal on Matrix Analysis and Applications, 39 (2018), pp.~234--261.

\end{thebibliography}
\bibliographystyle{siam}

\end{document}